\documentclass[11pt]{article}
\usepackage{amsmath,amsthm,amssymb}
\usepackage[usenames,dvipsnames]{xcolor}
\usepackage{enumerate}
\usepackage{graphicx}
\usepackage{cite}
\usepackage{comment}
\usepackage{oands}
\usepackage{tikz}
\usepackage{changepage}
\usepackage{bbm}
\usepackage{mathtools}
\usepackage[margin=1in]{geometry}
\usepackage[pagewise,mathlines]{lineno}
\usepackage{appendix}
\usepackage{stmaryrd} 
\usepackage{multicol}
\usepackage{microtype}
\usepackage[colorinlistoftodos]{todonotes}

\usepackage[pdftitle={The dimension of the boundary of a Liouville quantum gravity metric ball},
  pdfauthor={Ewain Gwynne},
colorlinks=true,linkcolor=NavyBlue,urlcolor=RoyalBlue,citecolor=PineGreen,bookmarks=true,bookmarksopen=true,bookmarksopenlevel=2,unicode=true,linktocpage]{hyperref}

\theoremstyle{plain}
\newtheorem{thm}{Theorem}[section]

\newtheorem{lem}[thm]{Lemma}
\newtheorem{prop}[thm]{Proposition}

\newtheorem{notation}[thm]{Notation}

\def\@rst #1 #2other{#1}
\newcommand\MR[1]{\relax\ifhmode\unskip\spacefactor3000 \space\fi
  \MRhref{\expandafter\@rst #1 other}{#1}}
\newcommand{\MRhref}[2]{\href{http://www.ams.org/mathscinet-getitem?mr=#1}{MR#2}}

\theoremstyle{definition}
\newtheorem{defn}[thm]{Definition}
\newtheorem{remark}[thm]{Remark}

\numberwithin{equation}{section}

\newcommand{\dsb}{\begin{adjustwidth}{2.5em}{0pt}
\begin{footnotesize}}
\newcommand{\dse}{\end{footnotesize}
\end{adjustwidth}}

\newcommand{\ssb}{\begin{adjustwidth}{2.5em}{0pt}}
\newcommand{\sse}{\end{adjustwidth}}

\newcommand{\aryb}{\begin{eqnarray*}}
\newcommand{\arye}{\end{eqnarray*}}
\def\alb#1\ale{\begin{align*}#1\end{align*}}
\def\allb#1\alle{\begin{align}#1\end{align}}
\newcommand{\eqb}{\begin{equation}}
\newcommand{\eqe}{\end{equation}}
\newcommand{\eqbn}{\begin{equation*}}
\newcommand{\eqen}{\end{equation*}}

\newcommand{\BB}{\mathbbm}
\newcommand{\ol}{\overline}
\newcommand{\ul}{\underline}
\newcommand{\op}{\operatorname}

\newcommand{\frk}{\mathfrak}
\newcommand{\eqD}{\overset{d}{=}}
\newcommand{\ep}{\varepsilon}
\newcommand{\rta}{\rightarrow}

\newcommand{\wt}{\widetilde}
\newcommand{\wh}{\widehat} 
\newcommand{\mcl}{\mathcal}

\newcommand{\bdy}{\partial}
\newcommand{\rng}{\mathring}

\newcommand{\esssup}{{\mathrm{ess\,sup}}}

\let\originalleft\left
\let\originalright\right
\renewcommand{\left}{\mathopen{}\mathclose\bgroup\originalleft}
\renewcommand{\right}{\aftergroup\egroup\originalright}

\title{The dimension of the boundary of a Liouville quantum gravity metric ball}
 \date{ }
\author{Ewain Gwynne\footnote{\url{eg558@cam.ac.uk}}  \\ {\it University of Cambridge}}

\begin{document}

\maketitle

\begin{abstract}
Let $\gamma \in (0,2)$, let $h$ be the planar Gaussian free field, and consider the $\gamma$-Liouville quantum gravity (LQG) metric associated with $h$.  
We show that the essential supremum of the Hausdorff dimension of the boundary of a $\gamma$-LQG metric ball with respect to the Euclidean (resp.\ $\gamma$-LQG) metric is $2 - \frac{\gamma}{d_\gamma}\left(\frac{2}{\gamma} + \frac{\gamma}{2} \right)  + \frac{\gamma^2}{2d_\gamma^2}$ (resp.\ $d_\gamma-1$), where $d_\gamma$ is the Hausdorff dimension of the whole plane with respect to the $\gamma$-LQG metric. 
For $\gamma = \sqrt{8/3}$, in which case $d_{\sqrt{8/3}}=4$, we get that the essential supremum of Euclidean (resp.\ $\sqrt{8/3}$-LQG) dimension of a $\sqrt{8/3}$-LQG ball boundary is $5/4$ (resp.\ $3$). 

We also compute the essential suprema of the Euclidean and $\gamma$-LQG Hausdorff dimensions of the intersection of a $\gamma$-LQG ball boundary with the set of metric $\alpha$-thick points of the field $h$ for each $\alpha\in \mathbb R$. Our results show that the set of $\gamma/d_\gamma$-thick points on the ball boundary has full Euclidean dimension and the set of $\gamma$-thick points on the ball boundary has full $\gamma$-LQG dimension. 
\end{abstract}

\tableofcontents

\section{Introduction}
\label{sec-intro}

\subsection{Overview}
\label{sec-overview}

Let $U\subset \BB C$ be an open domain, let $\gamma \in (0,2)$, and let $h$ be the Gaussian free field (GFF) on $U$, or some minor variant thereof.
The \emph{$\gamma$-Liouville quantum gravity (LQG)} surface associated with $(U,h)$ is the random two-dimensional Riemannian manifold parametrized by $U$ with Riemannian metric tensor $e^{\gamma h} \,(dx^2+dy^2)$, where $dx^2+dy^2$ is the Euclidean metric tensor. LQG surface were first introduced by Polyakov~\cite{polyakov-qg1} as canonical models of random two-dimensional Riemannian manifolds. 
One sense in which these models are canonical is that LQG surfaces describe the scaling limits of discrete random surfaces, such as random planar maps. 
In particular, the special case when $\gamma=\sqrt{8/3}$ (sometimes called ``pure gravity") corresponds to uniform random planar maps.
Other values of $\gamma$ (sometimes called ``gravity coupled to matter") correspond to random planar maps weighted by the partition function of a statistical mechanics model on the map, such as the uniform spanning tree ($\gamma=\sqrt 2$) or the Ising model ($\gamma=\sqrt 3$). 

The above Riemannian metric tensor for an LQG surface does not make literal sense since the Gaussian free field is a random distribution, or generalized function, not a true function, so $e^{\gamma h}$ is not well-defined. 
However, one can make rigorous sense of the Riemannian volume form and distance function on $U$ associated with an LQG surface via appropriate regularization procedures. 
Basically, one can consider a sequence of continuous functions $\{h_n\}_{n\in\BB N}$ which approximate the GFF as $n\rta\infty$, exponentiate a constant times $ h_n$, re-scale appropriately, and then send $n\rta\infty$. 
The volume form, a.k.a.\ the $\gamma$-LQG area measure, is a random measure on $U$. This measure is a special case of a more general family of random measures called \emph{Gaussian multiplicative chaos}~\cite{kahane,rhodes-vargas-review,berestycki-gmt-elementary,shamov-gmc}. 
A number of important facts about the $\gamma$-LQG area measure, such as the convergence of the circle-average approximation and a version of the KPZ formula (which is a weaker variant of the formula from~\cite{kpz-scaling}), were established in~\cite{shef-kpz}. 
 
The Riemannian distance function, a.k.a.\ the $\gamma$-LQG metric, is a random metric on $U$ which we will denote by $D_h$.
This metric will be our primary interest in this paper.
This $\gamma$-LQG metric for general $\gamma \in (0,2)$ was constructed in~\cite{dddf-lfpp,gm-uniqueness}. 
More precisely,~\cite{dddf-lfpp} established the tightness of the approximating metrics and~\cite{gm-uniqueness} showed that the limiting metric is unique and scale invariant in the appropriate sense, building on results from~\cite{local-metrics,lqg-metric-estimates,gm-confluence}.  
See also~\cite{ding-dunlap-lqg-fpp,df-lqg-metric,ding-dunlap-lgd} for earlier tightness results (preceding~\cite{dddf-lfpp}); and~\cite{gm-coord-change,gp-kpz} which establish the conformal coordinate change formula and the KPZ formula~\cite{kpz-scaling,shef-kpz}, respectively, for the LQG metric. 
For each $\gamma\in(0,2)$, the metric $D_h$ is uniquely characterized by a list of four simple axioms, which we state in Definition~\ref{def-lqg-metric} below.

There is also an earlier construction of the LQG metric in the special case when $\gamma=\sqrt{8/3}$ due to Miller and Sheffield~\cite{lqg-tbm1,lqg-tbm2,lqg-tbm3}. This construction uses a completely different regularization procedure from the one in~\cite{dddf-lfpp,gm-uniqueness} which only works for $\gamma=\sqrt{8/3}$.
However, the Miller-Sheffield construction gives additional information about the $\sqrt{8/3}$-LQG metric which is not apparent from the construction in~\cite{dddf-lfpp,gm-uniqueness}, such as certain Markov properties for LQG metric balls and the connection to the Brownian map~\cite{legall-uniqueness,miermont-brownian-map}. 

The $\gamma$-LQG metric $D_h$ induces the same topology on $U$ as the Euclidean metric, but it is fractal in the sense that the Hausdorff dimension $d_\gamma$ of the metric space $(U,D_h)$ is strictly larger than 2. 
It is known, e.g., from the Miller-Sheffield construction that $d_{\sqrt{8/3}} = 4$, but for other values of $\gamma \in (0,2)$ the value of $d_\gamma$ is unknown even at a physics level of rigor. See~\cite{dg-lqg-dim,gp-lfpp-bounds,ang-discrete-lfpp} for the best currently known rigorous upper and lower bounds for $d_\gamma$.

The best-known guess for the value of $d_\gamma$ is Watabiki's prediction~\cite{watabiki-lqg}, 
\eqb \label{eqn-watabiki}
d_\gamma^{\op{Wat}} = 1  + \frac{\gamma^2}{4} + \frac14\sqrt{(4+\gamma^2)^2 + 16\gamma^2} .
\eqe
This prediction is known to be wrong at least for small values of $\gamma$~\cite{ding-goswami-watabiki}, but it matches up reasonably well with numerical simulations; see, e.g.,~\cite{ambjorn-budd-lqg-geodesic}. A possible alternative guess, first put forward in~\cite{dg-lqg-dim}, is 
\eqb \label{eqn-quad}
d_\gamma^{\op{Quad}} := 2+\frac{\gamma^2}{2} + \frac{\gamma}{\sqrt 6} .
\eqe 
The formula~\eqref{eqn-quad} is consistent with all known bounds for $d_\gamma$. 
Moreover, recent numerical simulations by Barkley and Budd~\cite{bb-lqg-dim} fit much more closely with~\eqref{eqn-quad} than with~\eqref{eqn-watabiki}. 
However, there is currently no theoretical justification, even at a heuristic level, for~\eqref{eqn-quad}. 

A major motivation for computing $d_\gamma$ is that many quantities associated with $\gamma$-LQG surfaces and random planar maps can be expressed in terms of $d_\gamma$.
Such quantities include the growth exponents for metric balls~\cite{dg-lqg-dim} and random walk~\cite{gm-spec-dim,gh-displacement} on random planar maps, the optimal H\"older exponents between the LQG metric and the Euclidean metric~\cite{lqg-metric-estimates}, the KPZ formula for the metric~\cite{gp-kpz}, and the Hausdorff dimensions of various sets which we discuss just below. 

Just like other natural random fractal objects such as Brownian motion or SLE~\cite{schramm0}, the $\gamma$-LQG metric gives rise to a number of interesting random fractal sets in the plane. These include geodesics, boundaries of metric balls, boundaries of Voronoi cells, and certain special subsets of these sets.  
In this paper, we will give the first non-trivial calculation of the Euclidean and quantum Hausdorff dimensions of one of these sets, namely the boundary of an LQG metric ball (Theorem~\ref{thm-bdy-dim} just below).
We will also compute for each $\alpha\in\BB R$ the Euclidean and quantum Hausdorff dimensions of the intersection of an LQG metric ball boundary with the \emph{metric $\alpha$-thick points} of the field, which we define in~\eqref{eqn-metric-thick} (Theorem~\ref{thm-thick-dim}). 

The background knowledge needed to understand this paper is rather minimal. 
Our proofs use only the axiomatic definition of the LQG metric (Definition~\ref{def-lqg-metric}), some basic facts about the Gaussian free field (as reviewed in Appendix~\ref{sec-gff-prelim}), and a few estimates for the LQG metric from~\cite{lqg-metric-estimates,gp-kpz} which we review as needed.

\subsubsection*{Acknowledgments} 
We thank an anonymous referee for many helpful comments on an earlier version of this article.
We thank Nina Holden, Jason Miller, Joshua Pfeffer, Scott Sheffield, and Xin Sun for helpful discussions.
This research was conducted in part while the author was supported by a Herchel Smith fellowship and in part while the author was supported by a Clay Research Fellowship. 
The author was also supported by a Trinity College junior research fellowship. 

\subsection{Hausdorff dimension of an LQG metric ball boundary}
\label{sec-bdy-dim}

Before stating our main theorems, we introduce some notation. Define the exponents
\eqb \label{eqn-xi-Q}
\xi = \xi_\gamma := \frac{\gamma}{d_\gamma} \quad \text{and} \quad Q = Q_\gamma := \frac{2}{\gamma} + \frac{\gamma}{2} .
\eqe
There exponents govern the scaling behavior of $D_h$ when we add a constant to $h$ and when we scale space, respectively; see Definition~\ref{def-lqg-metric}. 

Recall that for $\Delta > 0$, the $\Delta$-\emph{Hausdorff content} of a metric space $(X,D)$ is the number
\eqb
C^\Delta(X,D) := \inf\left\{\sum_{j=1}^\infty r_j^\Delta : \text{there is a cover of $X$ by $D$-balls of radii $\{r_j\}_{j\geq 0}$} \right\} 
\eqe
and the \emph{Hausdorff dimension} of $(X,D)$ is defined to be $\inf\{\Delta > 0: C^\Delta(X,D) = 0\}$. See, e.g.,~\cite[Chapter 4]{peres-bm} for more on Hausdorff dimension. 

For simplicity, we will mostly restrict attention to the case when $U=\BB C$ and $h$ is a whole-plane GFF with the additive constant chosen so that the circle average over the unit circle is zero (we review the definition and basic properties of the whole-plane GFF in Appendix~\ref{sec-gff-prelim}). 
Other cases can be deduced from this case via local absolute continuity.
For a set $X\subset \BB C$, we write $\dim_{\mcl H}^0 X$ for the Hausdorff dimension of the metric space $X$ with respect to the Euclidean metric, i.e., the Hausdorff dimension of the metric space $(X,|\cdot|)$.
We similarly define $\dim_{\mcl H}^\gamma$ to be the Hausdorff dimension of $X$ with respect to the $\gamma$-LQG metric $D_h$.
We refer to these quantities as the Euclidean and quantum dimensions of $X$, respectively. 

For $\BB s > 0$, let $\mcl B_{\BB s} = \mcl B_{\BB s}(0;D_h)$ be the $D_h$-ball of radius $\BB s$ centered at zero.  
We note that $\bdy\mcl B_{\BB s}$ is typically not connected since $\mcl B_{\BB s}$ has ``holes". 

For a real-valued random variable $X$, we define its essential supremum by
\eqb \label{eqn-esssup}
\esssup X = \sup\left\{x \in \BB R : \BB P\left[X \geq x \right] > 0 \right\} .
\eqe
The first main result of this paper is the following theorem. 

\begin{thm} \label{thm-bdy-dim}
For each $\BB s > 0$, we have (in the notation~\eqref{eqn-xi-Q}) 
\eqb
\esssup \dim_{\mcl H}^0 \bdy\mcl B_{\BB s} = 2 - \xi Q + \xi^2/2 \quad \text{and} \quad
\esssup \dim_{\mcl H}^\gamma \bdy \mcl B_{\BB s} = d_\gamma  -1 .
\eqe 
\end{thm}

The reason why we only compute the essential suprema of $\dim_{\mcl H}^0 \mcl B_{\BB s}$ and $\dim_{\mcl H}^\gamma \mcl B_{\BB s}$ is as follows. To prove a lower bound for these dimensions, we use the usual argument (as in, e.g.,~\cite{hmp-thick-pts,beffara-dim,mww-nesting}) which involves constructing a so-called Frostman measure on $\mcl B_{\BB s}$ with positive probability; see Section~\ref{sec-outline} for details. This argument gives a lower bound for the essential supremum of the dimension. In most applications of this technique, constructing the Frostman measure is the main step in the proof and there is a simple zero-one law argument which says that the Hausdorff dimension must be a.s.\ equal to a deterministic constant. 
In our setting, the zero-one law step appears to be non-trivial. 
We are currently working on another paper with Joshua Pfeffer and Scott Sheffield which will show that the Euclidean and quantum Hausdorff dimensions of several different sets associated with the LQG metric, including the boundary of an LQG metric ball, are a.s.\ equal to their essential suprema. 

In the special case when $\gamma=\sqrt{8/3}$, we have $\xi=1/\sqrt 6$ and $Q=5/\sqrt 6$ and hence the $\esssup$ of the Euclidean (resp.\ quantum) dimension of $\bdy\mcl B_{\BB s}$ is $5/4$ (resp.\ 3). For other values of $\gamma$, we do not know the value of $d_\gamma$ hence we do not know these dimensions explicitly. However, we get upper and lower bounds for $\esssup \dim_{\mcl H}^0 \bdy\mcl B_{\BB s}$ and $\esssup \dim_{\mcl H}^\gamma \bdy \mcl B_{\BB s}$ by plugging in the known bounds for $d_\gamma$ from~\cite{dg-lqg-dim,gp-lfpp-bounds,ang-discrete-lfpp}; see Figure~\ref{fig-bdy-dim}. For example, we know that
\eqb
  \esssup \dim_{\mcl H}^0 \bdy\mcl B_{\BB s} \leq 1.2584,\quad\forall \gamma \in (0,2) \quad \text{and} \quad 
\eqe
\eqb
  1.2343 \leq \esssup\dim_{\mcl H}^0 \bdy\mcl B_{\BB s} \leq \frac54 \quad \text{for $\gamma=\sqrt 2$}. 
\eqe
We emphasize that the Euclidean and quantum dimensions in Theorem~\ref{thm-bdy-dim} are \emph{not} related by the KPZ formula from~\cite{gp-kpz} since $\bdy\mcl B_{\BB s}$ is not independent from $h$.

\begin{figure}[t!]
 \begin{center}
\includegraphics[scale=.55]{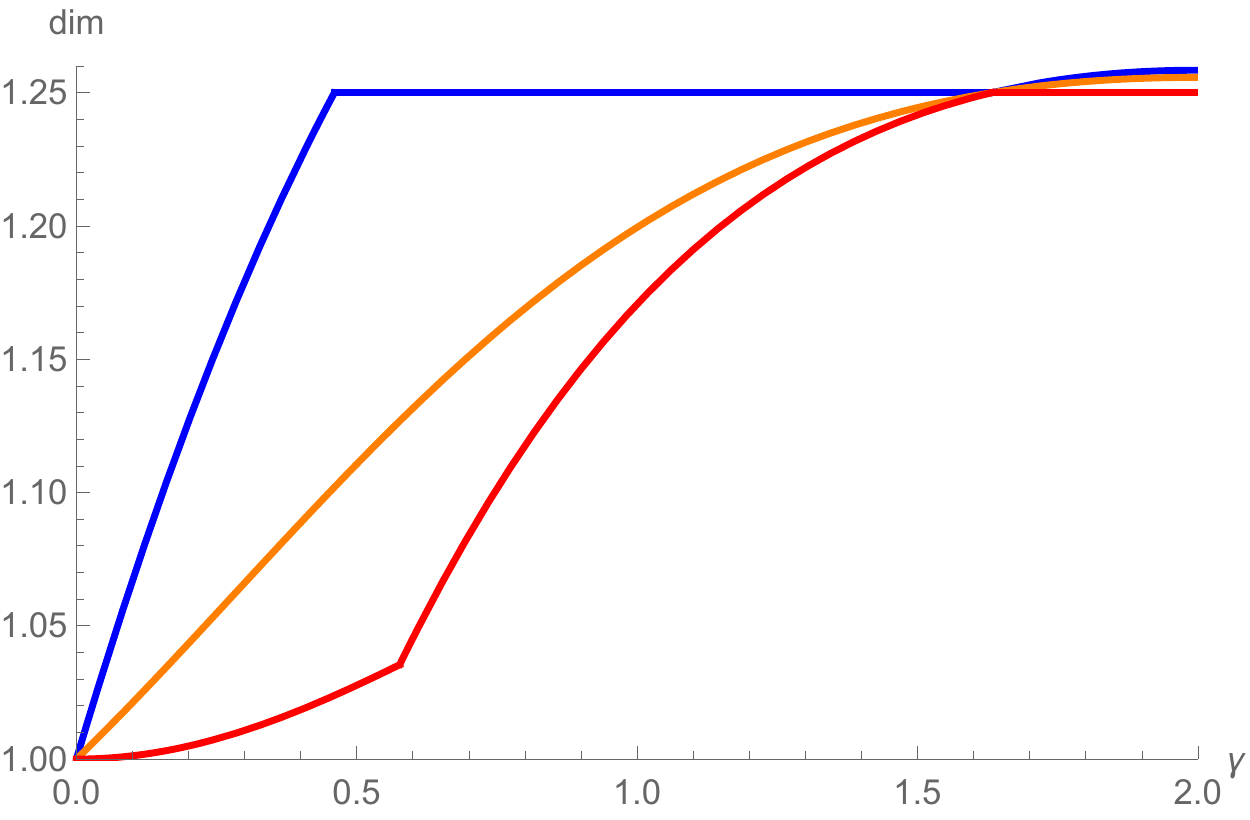} \hspace{15pt}  \includegraphics[scale=.55]{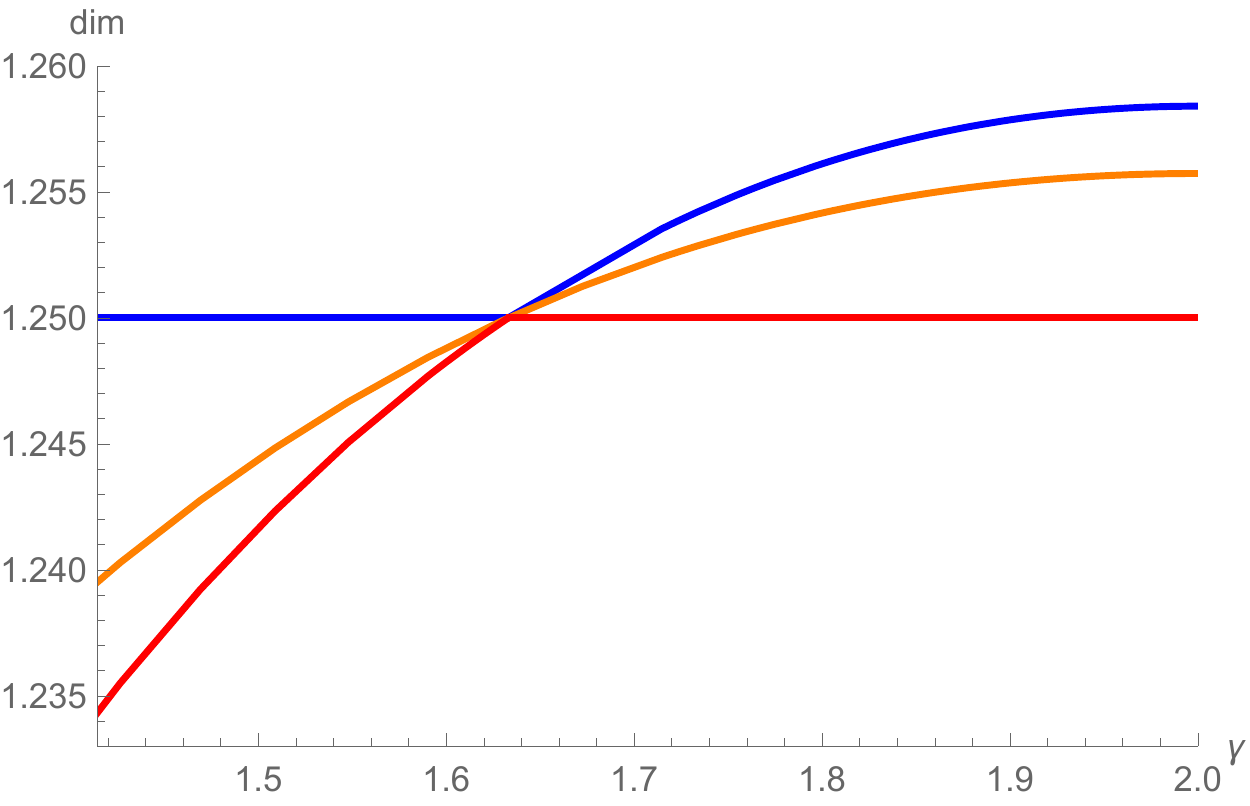} \\ 
\includegraphics[scale=.45]{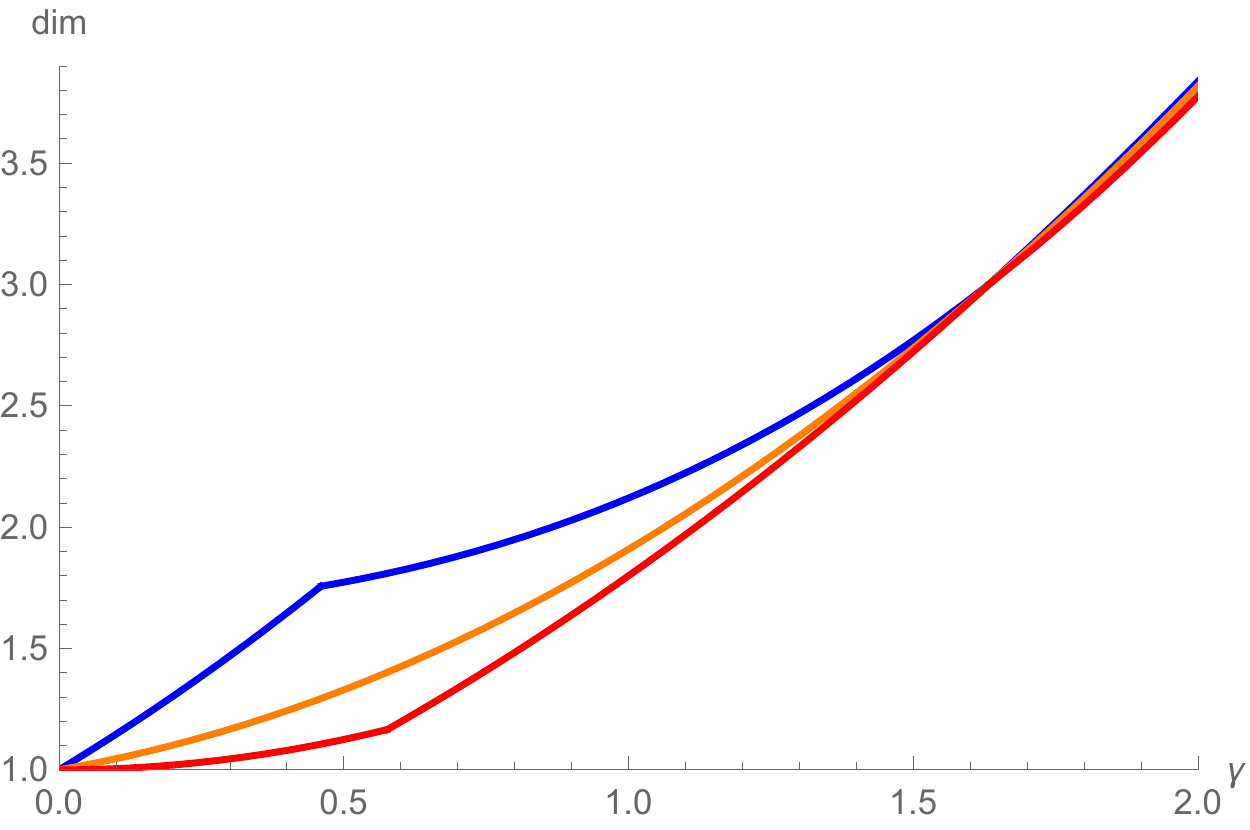} \hspace{15pt} \includegraphics[scale=.55]{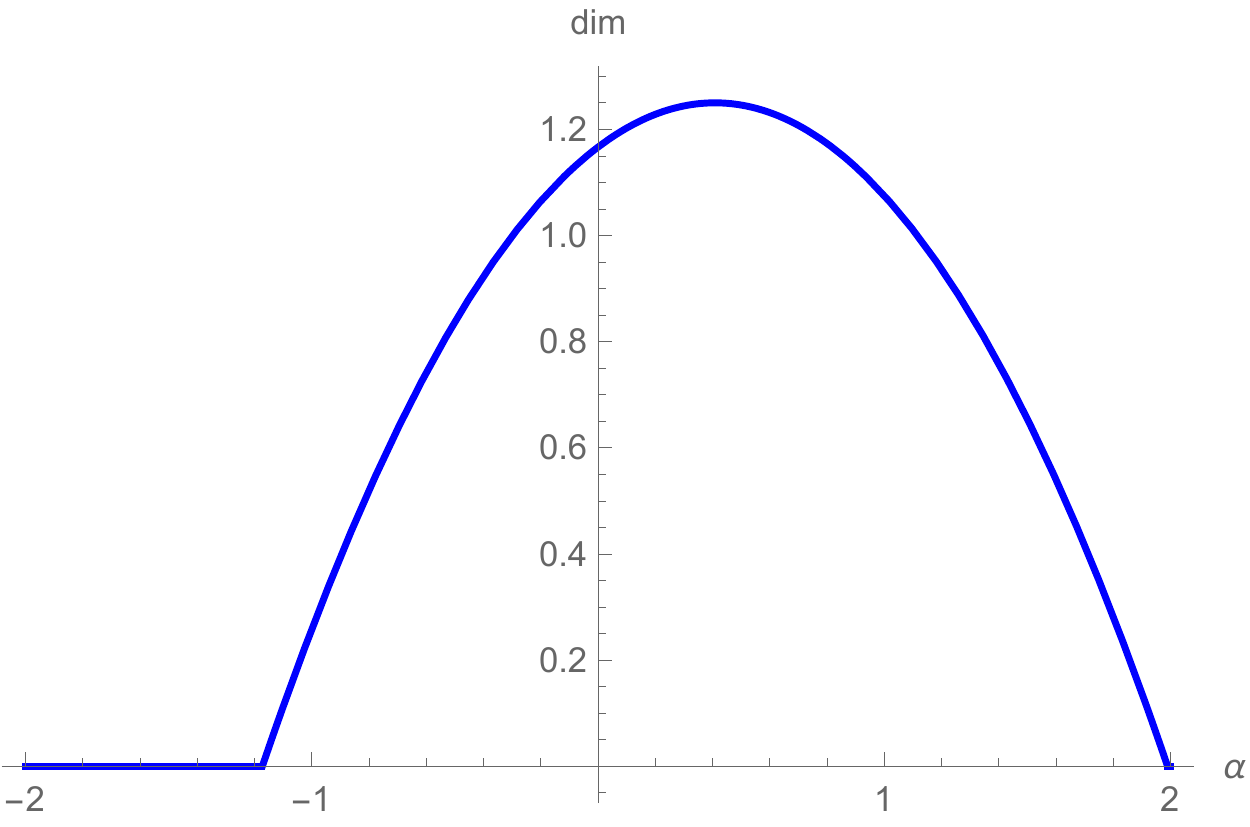}
\caption{ \textbf{Top Left.} Graph of known upper and lower bounds for the essential supremum of the Euclidean dimension of the boundary of an $\gamma$-LQG metric ball (red and blue) and the value of this essential supremum assuming the quadratic guess~\eqref{eqn-quad} (orange). 
The upper and lower bounds come from combining Theorem~\ref{thm-bdy-dim} with the bounds for $d_\gamma$ from~\cite{dg-lqg-dim,gp-lfpp-bounds,ang-discrete-lfpp}.
\textbf{Top Right.} Graph of the same bounds, but with $\gamma$ restricted to $[\sqrt 2,2]$. 
\textbf{Bottom Left.} The same setup as the top left, but for quantum dimension instead of Euclidean dimension.
\textbf{Bottom Right.} The formula for the metric $\alpha$-thick point dimension $\esssup \dim_{\mcl H}^0 (\bdy\mcl B_{\BB s} \cap \wh{\mcl T}_h^\alpha)$ from Theorem~\ref{thm-thick-dim} for $\gamma=\sqrt{8/3}$, as a function of $\alpha$. The maximum dimension occurs at $\alpha = \xi = 1/\sqrt 6$ and the dimension is non-zero for $\alpha \in \left(1/\sqrt 6 - \sqrt{5/2} , 1/\sqrt 6  +\sqrt{5/2}     \right)$. 
}\label{fig-bdy-dim}
\end{center}
\vspace{-1em}
\end{figure}

\subsection{Thick points on the boundary of an LQG metric ball}
\label{sec-thick-pts}

In the course of proving Theorem~\ref{thm-bdy-dim}, we also compute the dimension of the intersection of $\bdy\mcl B_{\BB s}$ with a variant of the set of $\alpha$-thick points of $h$ for each $\alpha \in \BB R$. 
Following~\cite{hmp-thick-pts}, for $\alpha\in\BB R$, we define the set of \emph{$\alpha$-thick points} of $h$ by 
\eqb \label{eqn-thick-def}
\mcl T_h^\alpha := \left\{z\in\BB C : \lim_{\ep\rta 0} \frac{h_\ep(z)}{\log\ep^{-1}} = \alpha \right\} ,
\eqe
where $h_\ep(z)$ is the average of $h$ over the circle of radius $\ep$ centered at $z$ (see Section~\ref{sec-gff-prelim} for more on the circle average process).

When working with the metric, it is natural to consider a variant of the definition~\eqref{eqn-thick-def} where thickness is defined in terms of $D_h$-distances rather than circle averages. For $\alpha  \in \BB R$, we define the set of \emph{metric $\alpha$-thick points} of $h$ by 
\eqb \label{eqn-metric-thick}
\wh{\mcl T}_h^\alpha := \left\{z\in\BB C : \lim_{\ep\rta 0} \frac{\log \sup_{u,v\in B_\ep(z)} D_h(u,v)}{\log \ep} = \xi (Q-\alpha) \right\} .
\eqe
It is easy to see from the scaling properties of the metric (Definition~\ref{def-lqg-metric}) that typically $ \sup_{u,v\in B_\ep(z)} D_h(u,v) \approx e^{\xi h_\ep(z)} \ep^{\xi Q}$.
However, the random variable $  e^{-\xi h_\ep(z)} \ep^{-\xi Q} \sup_{u,v\in B_\ep(z)} D_h(u,v)  $ has a heavy (power-law) upper tail; see~\cite{lqg-metric-estimates}. As a consequence of this, one does not have $\mcl T_h^\alpha = \wh{\mcl T}_h^\alpha$.
Nevertheless, we expect that the sets $\mcl T_h^\alpha$ and $\wh{\mcl T}_h^\alpha$ have similar properties. 
 
The following theorem will be established as part of the proof of Theorem~\ref{thm-bdy-dim}. 

\begin{thm}[Metric thick point dimension] \label{thm-thick-dim}
For each $\BB s >0$ and each $\alpha\in (\xi -\sqrt{4-2\xi Q +\xi^2} ,\xi + \sqrt{4-2\xi Q +\xi^2})$, 
\allb \label{eqn-thick-dim}
\esssup\dim_{\mcl H}^0\left( \bdy\mcl B_{\BB s} \cap \wh{\mcl T}_h^\alpha \right) = \esssup\dim_{\mcl H}^0\left( \bdy\mcl B_{\BB s} \cap \mcl T_h^\alpha  \cap \wh{\mcl T}_h^\alpha \right)  &= 2-\xi(Q-\alpha) - \alpha^2/2  
\quad \text{and}  \notag\\
 \esssup\dim_{\mcl H}^\gamma\left( \bdy\mcl B_{\BB s} \cap \wh{\mcl T}_h^\alpha \right) = \esssup\dim_{\mcl H}^\gamma\left( \bdy\mcl B_{\BB s} \cap \mcl T_h^\alpha  \cap \wh{\mcl T}_h^\alpha \right)   &=  \frac{2-\alpha^2/2}{\xi(Q-\alpha)} - 1    .
\alle
For each $\alpha \notin [\xi -\sqrt{4-2\xi Q +\xi^2} ,\xi + \sqrt{4-2\xi Q +\xi^2}]$, a.s.\ $\bdy\mcl B_{\BB s} \cap \wh{\mcl T}_h^\alpha = \emptyset$. 
\end{thm}

Note that $(\xi -\sqrt{4-2\xi Q +\xi^2} ,\xi + \sqrt{4-2\xi Q +\xi^2})$ is precisely the set of $\alpha\in\BB R$ for which the right sides of the formulas in~\eqref{eqn-thick-dim} are positive. 
See Figure~\ref{fig-bdy-dim}, bottom right for a graph of the Euclidean dimension formula from Theorem~\ref{thm-thick-dim}, as a function of $\alpha$, when $\gamma=\sqrt{8/3}$.

The formula for $\esssup\dim_{\mcl H}^0\left( \bdy\mcl B_{\BB s} \cap \mcl T_h^\alpha \cap \wh{\mcl T}_h^\alpha \right)$ in Theorem~\ref{thm-thick-dim} is maximized when $\alpha  = \xi$, in which case it coincides with the formula $2-\xi Q +\xi^2/2$ for $\dim_{\mcl H}^0 \bdy\mcl B_{\BB s}$ from Theorem~\ref{thm-bdy-dim}. 
Hence, the set of $\xi$-thick points in $\bdy\mcl B_{\BB s}$ has full Euclidean dimension.  
Similarly, the formula for $ \esssup\dim_{\mcl H}^\gamma\left( \bdy\mcl B_{\BB s} \cap \mcl T_h^\alpha  \cap \wh{\mcl T}_h^\alpha  \right)$ is maximized when $\alpha=\gamma$, so the set of $\gamma$-thick points in $\bdy\mcl B_{\BB s}$ has full quantum dimension.  

It is not hard to show that a.s.\ $\dim_{\mcl H}^\gamma(\wh{\mcl T}_h^\alpha )  = \frac{2-\alpha^2/2}{\xi(Q-\alpha)}$ (see~\cite[Theorem 1.5]{gp-kpz} for the analogous statement for $\mcl T_h^\alpha$), so Theorem~\ref{thm-thick-dim} says that 
\eqb
\esssup\dim_{\mcl H}^\gamma\left( \bdy\mcl B_{\BB s} \cap \wh{\mcl T}_h^\alpha \right)  = \dim_{\mcl H}^\gamma\wh{\mcl T}_h^\alpha - 1 .
\eqe
One has analogous statements with $ \mcl T_h^\alpha \cap \wh{\mcl T}_h^\alpha$ in place of $\wh{\mcl T}_h^\alpha$. 

We expect that Theorem~\ref{thm-thick-dim} is also true with $\mcl T_h^\alpha$ in place of $\wh{\mcl T}_h^\alpha$ or $\mcl T_h^\alpha \cap \wh{\mcl T}_h^\alpha$, but we do not prove an upper bound for $\dim_{\mcl H}(\bdy \mcl B_{\BB s} \cap \mcl T_h^\alpha)$ here.

\subsection{Definition of the LQG metric}
\label{sec-metric-def}

The $\gamma$-LQG metric can be constructed as the limit of an explicit approximation scheme (called \emph{Liouville first passage percolation}) and is uniquely characterized by a certain list of axioms. In this paper we will only need the axiomatic definition, which we state in this section. 
Before stating the axioms, we need some preliminary definitions. 

\begin{defn} \label{def-metric-stuff}
Let $(X,D)$ be a metric space.
\begin{itemize}
\item
For a curve $P : [a,b] \rta X$, the \emph{$D$-length} of $P$ is defined by 
\eqbn
\op{len}\left( P ; D  \right) := \sup_{T} \sum_{i=1}^{\# T} D(P(t_i) , P(t_{i-1})) 
\eqen
where the supremum is over all partitions $T : a= t_0 < \dots < t_{\# T} = b$ of $[a,b]$. Note that the $D$-length of a curve may be infinite.
\item
We say that $(X,D)$ is a \emph{length space} if for each $x,y\in X$ and each $\ep > 0$, there exists a curve of $D$-length at most $D(x,y) + \ep$ from $x$ to $y$. 
\item
For $Y\subset X$, the \emph{internal metric of $D$ on $Y$} is defined by
\eqb \label{eqn-internal-def}
D(x,y ; Y)  := \inf_{P \subset Y} \op{len}\left(P ; D \right) ,\quad \forall x,y\in Y 
\eqe 
where the infimum is over all paths $P$ in $Y$ from $x$ to $y$. 
Note that $D(\cdot,\cdot ; Y)$ is a metric on $Y$, except that it is allowed to take infinite values.  
\item
If $X$ is an open subset of $\BB C$, we say that $D$ is  a \emph{continuous metric} if it induces the Euclidean topology on $X$. 
We equip the set of continuous metrics on $X$ with the local uniform topology on $X\times X$ and the associated Borel $\sigma$-algebra.
\end{itemize}
\end{defn}

We are now ready to state the definition of the LQG metric. 

\begin{defn}[The LQG metric]
\label{def-lqg-metric}
For $U\subset \BB C$, let $\mcl D'(U)$ be the space of distributions (generalized functions) on $\BB C$, equipped with the usual weak topology.   
A \emph{$\gamma$-LQG metric} is a collection of measurable functions $h\mapsto D_h$, one for each open set $U\subset\BB C$, from $\mcl D'(U)$ to the space of continuous metrics on $U$ with the following properties. 
Let $U\subset \BB C$ and let $h$ be a \emph{GFF plus a continuous function} on $U$: i.e., $h$ is a random distribution on $U$ which can be coupled with a random continuous function $f$ in such a way that $h-f$ has the law of the (zero-boundary or whole-plane, as appropriate) GFF on $U$.  Then the associated metric $D_h$ satisfies the following axioms.
\begin{enumerate}[I.]
\item \textbf{Length space.} Almost surely, $(U,D_h)$ is a length space, i.e., the $D_h$-distance between any two points of $U$ is the infimum of the $D_h$-lengths of $D_h$-continuous paths (equivalently, Euclidean continuous paths) in $U$ between the two points. \label{item-metric-length}
\item \textbf{Locality.} Let $V \subset U$ be a deterministic open set. 
The $D_h$-internal metric $D_h(\cdot,\cdot ; V)$ is a.s.\ equal to $D_{h|_V}$, so in particular it is a.s.\ determined by $h|_V$.  \label{item-metric-local}
\item \textbf{Weyl scaling.} Let $\xi = \gamma/d_\gamma$ be as in~\eqref{eqn-xi-Q}. For a continuous function $f : U\rta \BB R$, define
\eqb \label{eqn-metric-f}
(e^{\xi f} \cdot D_h) (z,w) := \inf_{P : z\rta w} \int_0^{\op{len}(P ; D_h)} e^{\xi f(P(t))} \,dt , \quad \forall z,w\in U,
\eqe 
where the infimum is over all continuous paths from $z$ to $w$ in $U$ parametrized by $D_h$-length.
Then a.s.\ $ e^{\xi f} \cdot D_h = D_{h+f}$ for every continuous function $f: U\rta \BB R$. \label{item-metric-f}
\item \textbf{Conformal coordinate change.} Let $\wt U\subset \BB C$ and let $\phi : U \rta \wt U$ be a deterministic conformal map. Then, with $Q = 2/\gamma+\gamma/2$ as in~\eqref{eqn-xi-Q}, a.s.\ \label{item-metric-coord}
\eqb \label{eqn-metric-coord}
 D_h \left( z,w \right) = D_{h\circ\phi^{-1} + Q\log |(\phi^{-1})'|}\left(\phi(z) , \phi(w) \right)  ,\quad  \forall z,w \in U.
\eqe    
\end{enumerate}
\end{defn}

The following theorem was proven in~\cite{dddf-lfpp,local-metrics,lqg-metric-estimates,gm-confluence,gm-uniqueness,gm-coord-change}.

\begin{thm}[Existence and uniqueness of the LQG metric] \label{thm-lqg-metric} 
For each $\gamma \in (0,2)$, there exists a $\gamma$-LQG metric in the sense of Definition~\ref{def-lqg-metric}.
If $D$ and $\wt D$ are two such metrics, then there is a deterministic constant $C>0$ such that whenever $h$ is a GFF plus a continuous function, a.s.\ $\wt D_h = C D_h$.  
\end{thm}

More precisely,~\cite[Theorem 1.2]{gm-uniqueness}, building on the tightness result of~\cite{dddf-lfpp} as well as the papers~\cite{local-metrics,lqg-metric-estimates,gm-confluence} shows that for each $\gamma \in (0,2)$, there is a unique (up to a deterministic global multiplicative constant) measurable function $h\mapsto D_h$ from $\mcl D'(\BB C)$ to the space of continuous metrics on $\BB C$ which satisfies Definition~\ref{def-lqg-metric} for $U=\BB C$ (note that this means $\phi$ in Axiom~\ref{item-metric-coord} is required to be a complex affine map). As explained in~\cite[Remark 1.5]{gm-uniqueness}, this gives a way to define $D_h$ whenever $h$ is a GFF plus a continuous function on an open domain $U\subset\BB C$ in such a way that Axioms~\ref{item-metric-length} through~\ref{item-metric-f} hold.
It is shown in~\cite[Theorem 1.1]{gm-coord-change} that with the above definition, Axiom~\ref{item-metric-coord} holds. 
 
Because of Theorem~\ref{thm-lqg-metric}, we may refer to the metric satisfying Definition~\ref{def-lqg-metric} as \emph{the $\gamma$-LQG metric}. Technically, the metric is unique only up to a global deterministic multiplicative constant. We will always assume that this constant is fixed in some arbitrary way (e.g., by requiring that the median of $D_h(0,\bdy\BB D)$ is 1 when $h$ is a whole-plane GFF normalized so that its circle average over $\bdy\BB D$ is zero). The choice of multiplicative constant plays no role in our results or proofs.

\subsection{Outline}
\label{sec-outline0}

The rest of this paper is structured as follows.
In Section~\ref{sec-upper-bound} we prove the upper bounds for the dimensions in Theorems~\ref{thm-bdy-dim} and~\ref{thm-thick-dim}. 
This is done by upper-bounding the probability of a certain event. See the beginning of Section~\ref{sec-one-pt} for an explanation of the key ideas of the argument (which in particular explains where the formulas for the dimensions in Theorems~\ref{thm-bdy-dim} and~\ref{thm-thick-dim} come from).
Our arguments in fact give upper bounds for the Hausdorff dimensions of a general class of subsets of $\bdy\mcl B_{\BB s}$; see Theorem~\ref{thm-gen-upper} for a precise statement. 

Sections~\ref{sec-outline}, \ref{sec-short-range}, \ref{sec-long-range}, and~\ref{sec-onescale-prob} are devoted to the proofs of the lower bounds in Theorems~\ref{thm-bdy-dim} and~\ref{thm-thick-dim}. An outline of the proof of these lower bounds is given in Section~\ref{sec-outline}. 

Appendix~\ref{sec-gff-prelim} contains a review of the definitions of the whole-plane and zero-boundary GFF and the properties of these objects which are used in this paper.
We encourage the reader to review this appendix before reading the rest of the paper if he or she is not already familiar with the GFF.

\subsection{Basic notation}
\label{sec-notation}

\noindent
We write $\BB N = \{1,2,3,\dots\}$ and $\BB N_0 = \BB N \cup \{0\}$. 
For $a < b$, we define the discrete interval $[a,b]_{\BB Z}:= [a,b]\cap\BB Z$. 
\medskip
 
\noindent
If $f  :(0,\infty) \rta \BB R$ and $g : (0,\infty) \rta (0,\infty)$, we say that $f(\ep) = O_\ep(g(\ep))$ (resp.\ $f(\ep) = o_\ep(g(\ep))$) as $\ep\rta 0$ if $f(\ep)/g(\ep)$ remains bounded (resp.\ tends to zero) as $\ep\rta 0$. 
We similarly define $O(\cdot)$ and $o(\cdot)$ errors as a parameter goes to infinity. 
\medskip

\noindent
If $f,g : (0,\infty) \rta [0,\infty)$, we say that $f(\ep) \preceq g(\ep)$ if there is a constant $C>0$ (independent from $\ep$ and possibly from other parameters of interest) such that $f(\ep) \leq  C g(\ep)$. We write $f(\ep) \asymp g(\ep)$ if $f(\ep) \preceq g(\ep)$ and $g(\ep) \preceq f(\ep)$. 
\medskip

\noindent
We often specify requirements on the dependencies on rates of convergence in $O(\cdot)$ and $o(\cdot)$ errors, implicit constants in $\preceq$, etc., in the statements of lemmas/propositions/theorems, in which case we implicitly require that errors, implicit constants, etc., in the proof satisfy the same dependencies. 
\medskip

\noindent
For $z\in\BB C$ and $r>0$, we write $B_r(z)$ for the Euclidean ball of radius $r$ centered at $z$. We also define the open annulus
\eqb \label{eqn-annulus-def}
\BB A_{r_1,r_2}(z) := B_{r_2}(z) \setminus \ol{B_{r_1}(z)} ,\quad\forall 0 < r_r < r_2 < \infty .
\eqe
\medskip 

\noindent
For a metric space $(X,D)$, $A\subset X$, and $r>0$, we write $\mcl B_r(A;D)$ for the open ball consisting of the points $x\in X$ with $D (x,A) < r$.  
If $A = \{y\}$ is a singleton, we write $\mcl B_r(\{y\};D) = \mcl B_r(y;D)$.

\section{Upper bounds}
\label{sec-upper-bound}

In this section we will prove the upper bounds for Hausdorff dimension in Theorems~\ref{thm-bdy-dim} and~\ref{thm-thick-dim}. 
In doing so, we will see where the formulas for the dimensions in these theorems come from. 
Throughout this section, $h$ denotes a whole-plane GFF normalized so that its circle average over $\bdy\BB D$ is zero. 

\subsection{The one-point upper bound}
\label{sec-one-pt}

The main input in the proof of Theorems~\ref{thm-bdy-dim} and~\ref{thm-thick-dim} is an upper bound for the probability of a certain event (Proposition~\ref{prop-event-upper}). 
Let us now define the event we consider. 

Fix a small $\zeta >0$ (which we will send to zero at the very end of the proof).
It will be convenient to work at positive distance from 0, so we also fix a bounded open set $V\subset \BB C$ with $0\notin \ol V$.   

For $\alpha \in \BB R$ and $z\in \BB C$, let
\allb \label{eqn-one-pt-event}
E_\alpha^\ep(z)   
&:= \left\{  D_h(0,z ) \in \left[\BB s -  \ep^{\xi(Q-\alpha) -   \zeta}  , \BB s +  \ep^{\xi(Q-\alpha) -  \zeta} \right] \right\} \notag \\
&\qquad \qquad \qquad \cap \left\{  \sup_{u,v\in B_\ep(z)} D_h(u,v )  \in \left[  \ep^{\xi(Q-\alpha) + \zeta} , \ep^{\xi(Q-\alpha) - \zeta} \right]   \right\} .
\alle 
The following trivial lemma allows us to connect the events $E_\alpha^\ep(z)$ to the set $\bdy\mcl B_{\BB s} $. 
 
\begin{lem} \label{lem-one-pt-contain}
If $z\in \BB C$ such that $B_\ep(z) \cap \mcl B_{\BB s} \not=\emptyset$ and $\sup_{u,v\in B_\ep(z)} D_h(u,v )  \in \left[  \ep^{\xi(Q-\alpha) + \zeta} , \ep^{\xi(Q-\alpha) - \zeta} \right]$, then $E_\alpha^\ep(z)$ occurs.
\end{lem}
\begin{proof}
If $B_\ep(z) \cap \bdy\mcl B_{\BB s}  \not=\emptyset$, then there is a $w\in B_\ep(z)$ such that $D_h(0,w ) = \BB s$. By the triangle inequality, if also $\sup_{u,v\in B_\ep(z)} D_h(u,v ) \leq \ep^{\xi (Q-\alpha) -\zeta}$, then 
\alb
D_h(0,z ) 
&\in \left[D_h(0,w ) - \sup_{u,v\in B_\ep(z)} D_h(u,v ) ,  D_h(0,w ) +   \sup_{u,v\in B_\ep(z)} D_h(u,v  ) \right]  \notag\\
&\subset \left[ \BB s -  \ep^{\xi (Q-\alpha) -\zeta} , \BB s +  \ep^{\xi (Q-\alpha) -\zeta} \right] .
\ale
\end{proof}

The rest of this subsection is devoted to the proof of the following one-point estimate.

\begin{prop} \label{prop-event-upper}
For each $\alpha \in [-2  ,2 ]$, each $z\in V$, and each $\ep > 0$,  
\eqb \label{eqn-event-upper} 
\BB P\left[ E_\alpha^\ep(z) \right] \leq \ep^{   \xi(Q-\alpha) + \alpha^2/2  + o_\zeta(1) +  o_\ep(1)    }
\eqe
where the rate of the $o_\zeta(1)$ depends only on $\gamma$ and the rate of the $o_\ep(1)$ depends only on $ V, \alpha,\zeta,\gamma $ (not on the particular choice of $z$).
\end{prop}

Throughout this section, all $o_\zeta(1)$ and $o_\ep(1)$ errors are required to satisfy the dependencies in Proposition~\ref{prop-event-upper}. 
Proposition~\ref{prop-event-upper} will be used directly to obtain the upper bound in Theorem~\ref{thm-thick-dim}.
To obtain the upper bound in Theorem~\ref{thm-bdy-dim}, we will look at the ``worst case" value of $\alpha$ in Proposition~\ref{prop-event-upper} (see Proposition~\ref{prop-bdy-prob}). 

We now explain the idea of the proof of Proposition~\ref{prop-event-upper}. 
The event $E_\alpha^\ep(z)$ of~\eqref{eqn-one-pt-event} is the intersection of a ``short-range event (the one involving $\sup_{u,v\in B_\ep(z)} D_h(u,v )$) and a ``long-range event" (the one involving $D_h(0,z )$).
The idea of the proof of Proposition~\ref{prop-event-upper} is that the short-range and long-range events are approximately independent from each other. 
We will first show, using basic moment estimates for $D_h$-diameters from~\cite{lqg-metric-estimates}, that the probability of the short-range event in~\eqref{eqn-one-pt-event} is bounded above by $\ep^{\alpha^2/2+o_\zeta(1) + o_\ep(1)}$ (Lemma~\ref{lem-diam-tail}).

We will then deal with the long-range event as follows. 
Let $\phi$ be a smooth, compactly supported bump function which is identically equal to 1 on a neighborhood of 0 and which is identically equal to 0 on $\ol V$.
Let $X$ be sampled uniformly from $[0,1]$, independently from $h$. 
Using Weyl scaling, one can show that the law of $D_{h+X\phi}(0,z)$ has a bounded density with respect to Lebesgue measure on $[0,\infty)$ (see Lemma~\ref{lem-theta-deriv}), from which we infer that the probability that
\eqbn
D_{h+X\phi}(0,z ) \in  \left[\BB s -  \ep^{\xi(Q-\alpha) -   \zeta}  , \BB s +  \ep^{\xi(Q-\alpha) -  \zeta} \right] 
\eqen
is at most $\ep^{\xi(Q-\alpha) +o_\zeta(1) + o_\ep(1)}$ (Lemma~\ref{lem-scaled-field-dist}). 
From this, we obtain a version of Proposition~\ref{prop-event-upper} with $h+X\phi$ in place of $h$ (Lemma~\ref{lem-event-upper'}). 
We then deduce Proposition~\ref{prop-event-upper} by bounding the Radon-Nikodym derivative between the laws of $h$ and $h+X\phi$ (which is described explicitly in Lemma~\ref{lem-gff-abs-cont}).

Let us now proceed with the details. 
The factor of $\ep^{\alpha^2/2 }$ on the right side of~\eqref{eqn-event-upper} comes from the following tail estimate for $D_h$-diameters.

\begin{lem} \label{lem-diam-tail}
For each $z \in V$, each $\ep > 0$, and each $\alpha \in [0,2]$, 
\eqb \label{eqn-diam-tail-pos}
\BB P\left[ \sup_{u,v\in B_\ep(z)} D_h(u,v  ) > \ep^{\xi (Q-\alpha)} \right] \leq \ep^{ \alpha^2/2 + o_\ep(1)    } .
\eqe
Moreover, for each $\alpha < 0$, 
\eqb \label{eqn-diam-tail-neg}
\BB P\left[ \sup_{u,v\in B_\ep(z)} D_h(u,v  ) <  \ep^{\xi (Q-\alpha)} \right] \leq \ep^{ \alpha^2/2 + o_\ep(1)    } .
\eqe  
\end{lem}
\begin{proof} 
We will prove~\eqref{eqn-diam-tail-pos}, then comment on the modifications necessary to get~\eqref{eqn-diam-tail-neg} at the end of the proof.
By the moment bound for LQG diameters from~\cite[Proposition 3.9]{lqg-metric-estimates} (with $\BB r =\ep$, $\ol{B_1(0)}$ in place of $K$, $B_2(0)$ in place of $U$), for each $z\in V$, each $\ep > 0$, and each $p\in [0,4d_\gamma/\gamma^2)$, 
\eqb \label{eqn-use-diam-moment}
\BB E\left[ \left(\ep^{-\xi Q} e^{-\xi h_{2\ep}(z)} \sup_{u,v\in B_\ep(z)} D_h\left(u,v  ; B_{2\ep}(z) \right) \right)^p \right]  \preceq 1
\eqe
with the implicit constant depending only on $ V,p,\gamma$. 
 
The random variable $h_{2\ep}(z) - h_1(z)$ is independent from $(h-h_{2\ep}(z))|_{B_{2\ep}(z)}$.
By Axiom~\ref{item-metric-f}, $D_{h-h_{2\ep}(z)} = e^{-\xi h_{2\ep}(z)} D_h$. 
By this and the locality of the metric (Axiom~\ref{item-metric-local}), the internal metric $e^{-\xi h_{2\ep}(z)} D_h\left(\cdot,\cdot ; B_{2\ep}(z) \right)$ is independent from $h_{2\ep}(z) - h_1(z)$. 
Since $h_{2\ep}(z) - h_1(z)$ is centered Gaussian with variance $\log((2\ep)^{-1})$, we can therefore apply~\eqref{eqn-use-diam-moment} to obtain
\allb \label{eqn-diam-moment-scaled}
&\BB E\left[ \left( e^{-\xi h_1(z)} \sup_{u,v\in B_\ep(z)} D_h(u,v   ) \right)^p \right]   \notag \\
&\qquad \leq \BB E\left[ \left( e^{-\xi h_1(z)} \sup_{u,v\in B_\ep(z)} D_h(u,v ; B_{2\ep}(z) ) \right)^p \right]   \notag \\
&  \qquad = \BB E\left[ \left(e^{-\xi h_{2\ep}(z)} \sup_{u,v\in B_\ep(z)} D_h\left(u,v ; B_{2\ep}(z) \right) \right)^p \right] 
\BB E\left[ e^{p \xi (h_{2\ep}(z) - h_1(z))} \right]
= \ep^{\xi Q p - \xi^2 p^2/2 + o_\ep(1)} .
\alle

For $z\in V$, the random variable $h_1(z)$ is centered Gaussian with variance bounded above by a constant depending only on $V$. 
In particular, $\BB E[e^{q h_1(z)}]$ is bounded above by a constant depending only on $V,q$ for each $q \in \BB R$. 
By applying H\"older's inequality with exponents $1+\delta$ and $(1+\delta)/\delta$ and then sending $\delta \rta 0$ sufficiently slowly as $\ep\rta 0$, we therefore obtain from~\eqref{eqn-diam-moment-scaled} that
\eqb \label{eqn-diam-moment-unscaled} 
\BB E\left[ \left(   \sup_{u,v\in B_\ep(z)} D_h(u,v  ) \right)^p \right]  =  \ep^{\xi Q p - \xi^2 p^2/2 + o_\ep(1)} .
\eqe
By the Chebyshev inequality and~\eqref{eqn-diam-moment-unscaled}, for $\alpha \geq 0$ and $p\in [0,4d_\gamma/\gamma^2)$, 
\eqb \label{eqn-diam-moment-p-upper}
\BB P\left[   \sup_{u,v\in B_\ep(z)} D_h(u,v   ) > \ep^{\xi(Q-\alpha)} \right] 
\leq \ep^{p \xi \alpha - \xi^2 p^2/2 + o_\ep(1)} .
\eqe  
The exponent $p\xi \alpha - \xi^2p^2/2$ on the right side of each of~\eqref{eqn-diam-moment-p-upper} is maximized for a fixed choice of $\alpha$ when $p = \alpha/\xi$, in which case it equals $\alpha^2/2$. Note that for $\alpha \in [0,2]$, we have $\alpha/\xi  = \alpha d_\gamma/\gamma \in [0,4d_\gamma/\gamma^2)$.
By setting $p=\alpha/\xi$, we now obtain~\eqref{eqn-diam-tail-pos} from~\eqref{eqn-diam-moment-p-upper}.

The proof of~\eqref{eqn-diam-tail-neg} is essentially identical, except we use the lower bound $D_h(0,\bdy B_\ep(z))$ which comes from~\cite[Proposition 3.1]{lqg-metric-estimates} instead of~\cite[Proposition 3.9]{lqg-metric-estimates}. Note that $D_h(0,\bdy B_\ep(z))$ is determined by $h|_{B_\ep(z)}$ and is at most $\sup_{u,v\in B_\ep(z)} D_h(u,v)$. 
\end{proof}

Lemma~\ref{lem-diam-tail} gives us an upper bound of $\ep^{\alpha^2/2 + o_\zeta(1) + o_\ep(1)}$ for the probability of the short-range event in the definition of $E_\alpha^\ep(z)$ from~\eqref{eqn-one-pt-event}. 
In order to separate this event from the long-range event $\left\{  \sup_{u,v\in B_\ep(z)} D_h(u,v )  \in \left[  \ep^{\xi(Q-\alpha) + \zeta} , \ep^{\xi(Q-\alpha) - \zeta} \right]   \right\}$ appearing in~\eqref{eqn-one-pt-event}, we will introduce an auxiliary field. 

Let $O$ be an open set containing 0 which lies at positive distance from $V$ and from $\bdy\BB D$.  
Let $\phi : \BB C \rta [0,1]$ be a smooth compactly supported bump function which is equal to 1 on $O$ and which vanishes on $V\cup \bdy \BB D$, chosen in a manner depending only on $O,V$ (the reason why we need $\phi$ to vanish on $\bdy\BB D$ is so that adding a multiple of $\phi$ to $h$ does not change the fact that the average of $h$ over $\bdy\BB D$ is zero). 
Let $X$ be a uniform $[0,1]$ random variable sampled independently from $h$. 
The following lemma plus an absolute continuity argument will lead to our desired upper bound for the long-range event in the definition of $E_\alpha^\ep(z)$. 

\begin{lem} \label{lem-scaled-field-dist}
For each $\delta > 0$, each $z \in V$, and each $\BB s > 0$, a.s.\ 
\eqb \label{eqn-scaled-field-dist}
\BB P\left[ D_{h+X\phi}(0,z  ) \in [\BB s -\delta ,\BB s+\delta] \,\big| \, h \right] \preceq \frac{\delta}{D_h(0,\bdy O)} 
\eqe
with a deterministic implicit constant depending only on $\gamma$.
\end{lem}

We will deduce Lemma~\ref{lem-scaled-field-dist} from the following more general statement, which will be re-used later.

\begin{figure}[t!]
 \begin{center}
\includegraphics[scale=.75]{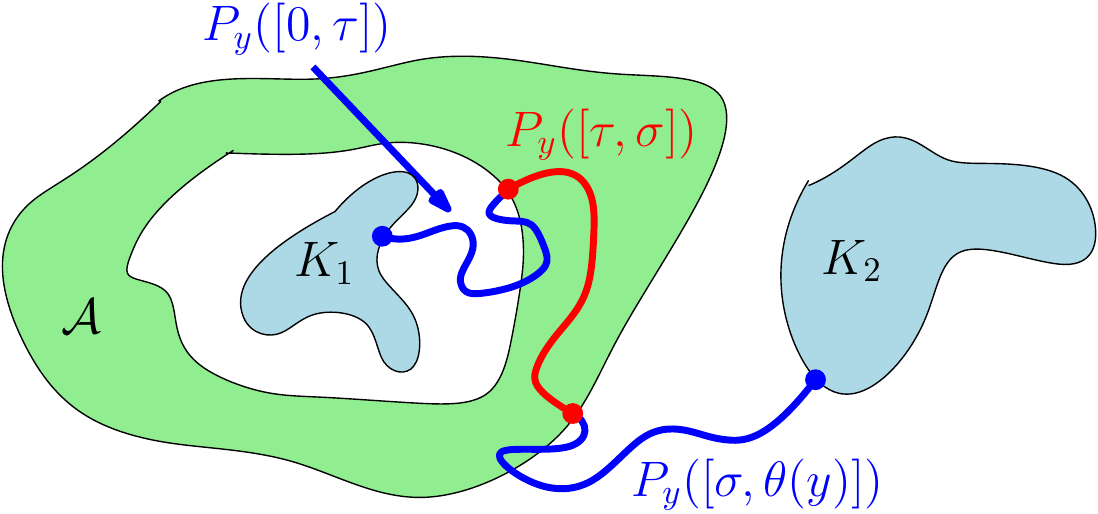}
\vspace{-0.01\textheight}
\caption{Illustration of the proof of Lemma~\ref{lem-theta-deriv}. To prove a lower bound for $(\theta(y) - \theta(x))/(y-x)$, we consider a $D_{h+y\phi }$-geodesic $P_y$ from $K_1$ to $K_2$. We then use Weyl scaling (Axiom~\ref{item-metric-f}) to upper-bound the $D_{h+x\phi }$-length of the segment $P_y|_{[\sigma,\tau]}$, which is contained in $\mcl A$.
}\label{fig-theta-deriv}
\end{center}
\vspace{-1em}
\end{figure}

\begin{lem} \label{lem-theta-deriv}
Let $K_1,K_2\subset \BB C$ be compact sets and let $\mcl A\subset\BB C$ be a region with the topology of a closed Euclidean annulus. Assume that $K_1$ (resp.\ $K_2$) is contained in the bounded (resp.\ unbounded) connected component of $\BB C\setminus \mcl A$. 
Let $\phi : \BB C\rta [0,1]$ be a continuous function which is identically equal to 1 on $\mcl A$. 
If we set $\theta(x) := D_{h+x\phi}(K_1,K_2)$ for each $x\geq 0$, then a.s.\ $\theta$ is strictly increasing and locally Lipschitz continuous and 
\eqb \label{eqn-theta-deriv}
\theta'(x) \geq \xi e^{\xi x} D_h\left( \bdy_{\op{in}} \mcl A , \bdy_{\op{out}} \mcl A \right) ,\quad \text{for Lebesgue-a.e.\ $x\geq 0$}
\eqe
where $\bdy_{\op{in}}\mcl A$ and $\bdy_{\op{out}}\mcl A$ denote the inner and outer boundaries of $\mcl A$.
\end{lem}
\begin{proof}
See Figure~\ref{fig-theta-deriv} for an illustration of the statement and proof. 
Since $\phi$ is non-negative it is clear from Weyl scaling (Axiom~\ref{item-metric-f}) that $\theta$ is a.s.\ non-decreasing and locally Lipschitz continuous.
In particular, $\theta$ is absolutely continuous. 
 
Let $y > x \geq 0$. To prove~\eqref{eqn-theta-deriv} we will prove a lower bound for $(\theta(y) - \theta(x))/(y-x)$ then send $y\rta x$. 
For this purpose, let $P_{y} : [0,  \theta(y) ] \rta \BB C$ be a path from $K_1$ to $K_2$ of minimal $D_{h+y \phi }$-length, parameterized by its $D_{h+y\phi}$-length (such a geodesic exists by a straightforward compactness argument).
We will upper-bound the $D_{h+x\phi }$-length of $P_y$.
 
Let $\sigma$ be the first time that $P_{y}$ hits $\bdy_{\op{out}} \mcl A$ and let $\tau$ be the last time before $\sigma$ at which $P_y$ hits $\bdy_{\op{in}} \mcl A$, so that $P_y|_{[\tau,\sigma]}\subset \mcl A$.
Since $\phi $ is identically equal to 1 on $\mcl A$, the internal metric of $D_{h+y\phi }$ on $\mcl A$ is precisely $e^{\xi y}$ times the corresponding internal metric of $D_h$.  
Since $P_y$ is parametrized by $D_{h+y\phi }$-length,
\eqb \label{eqn-across-annulus0}
\sigma-\tau 
\geq D_{h+y\phi}\left( \bdy_{\op{in}} \mcl A , \bdy_{\op{out}} \mcl A \right) 
= e^{\xi y} D_h\left( \bdy_{\op{in}} \mcl A , \bdy_{\op{out}} \mcl A \right)   .
\eqe 
 
Since $\phi $ is identically equal to 1 on $\mcl A$, the $D_{h+ x\phi }$-length of $P_{y}|_{[\tau,\sigma]}$ is exactly $e^{-\xi(y-x)} (\sigma-\tau)$. 
Since $y\geq x$ and $\phi  \geq 0$, the $D_{h+x\phi }$-length of the union of the complementary segments $P_{y}|_{[0,\tau] \cup [\sigma,\theta(y)]}$ is bounded above by its $D_{h+y\phi }$-length, which is exactly $  \theta(y)  - (\sigma-\tau)$. 
Therefore,
\allb
\theta(x) 
&= D_{h+x\phi }\left( K_1,K_2 \right) \notag\\
&\leq \left( \text{$D_{h+x\phi }$-length of $P_{y}$} \right)  \notag \\
&\leq  \theta(y)  - (\sigma-\tau)  + e^{-\xi(y-x)} (\sigma-\tau)   \notag \\
&= \theta(y)   - (1-e^{-\xi(y-x)})(\sigma-\tau)  \notag \\
&\leq  \theta(y)   -  (1-e^{-\xi(y-x)}) e^{\xi y}  D_h\left( \bdy_{\op{in}} \mcl A , \bdy_{\op{out}} \mcl A \right)  \quad \text{(by~\eqref{eqn-across-annulus0})} .
\alle
Re-arranging gives
\eqb
\frac{\theta(y) - \theta(x)}{y - x} \geq   \frac{1-e^{-\xi(y-x)}}{y-x} e^{\xi y }  D_h\left( \bdy_{\op{in}} \mcl A , \bdy_{\op{out}} \mcl A \right)   .
\eqe
Sending $y\rta x$ now gives~\eqref{eqn-dist-deriv}.
\end{proof}

\begin{proof}[Proof of Lemma~\ref{lem-scaled-field-dist}]
Let $\phi$ be as in the lemma statement and let $\theta(x) := D_{h+x\phi}(0,z)$. 
Let $r $ be a small positive number which is less than the Euclidean distance from 0 to $\bdy O$. By Lemma~\ref{lem-theta-deriv} applied with $K_1 = \{0\}$, $K_2 = \{z\}$, and $\mcl A = \ol{O\setminus B_r(0)}$, we get that $\theta$ is strictly increasing and absolutely continuous and $\theta'(x) \geq \xi e^{\xi x} D_h(\bdy B_r(0) , \bdy O)$ for each $x\geq 0$. Sending $r \rta 0 $ and restricting to $x\in [0,1]$ leads to $\theta'(x) \succeq D_h(0,\bdy O)$ for each $x\in [0,1]$. 
By integrating, we get $\theta(y) - \theta(x) \succeq (y-x) D_h(0,\bdy O)$ for each $x,y\in [0,1]$, which implies that $\theta^{-1}(b) - \theta^{-1}(a) \preceq (b-a) D_h(0,\bdy O)^{-1}$ for each $a,b \in [\theta(0) ,\theta(1)]$. 
Since $X$ is uniform on $[0,1]$ and is independent from $h$, we infer that
\allb
\BB P\left[ D_{h+X\phi}(0,z  ) \in [\BB s -\delta ,\BB s+\delta] \,\big| \, h \right] 
&= \left| \left\{ x\in [0,1] : \theta(x) \in [\BB s-\delta , \BB s + \delta ] \right\} \right| \notag \\
&= \theta^{-1}(\BB s +\delta) - \theta^{-1}(\BB s -\delta) \notag \\ 
&\preceq \frac{\delta}{D_h(0,\bdy O)} .
\alle
\end{proof}

We can now prove the analog of Proposition~\ref{prop-event-upper} with $h+X\phi$ in place of $h$.

\begin{lem} \label{lem-event-upper'}
Define the event $ E_\alpha^\ep(z ; h + X\phi)$ in exactly the same manner as in~\eqref{eqn-one-pt-event} above but with $h+X\phi$ in place of $h$.
For each $z\in V$, each $\ep > 0$, and each $\alpha \in [-2,2]$, 
\eqb \label{eqn-event-upper'} 
\BB P\left[   E_\alpha^\ep(z ; h + X\phi) \right] \leq \ep^{ \xi(Q-\alpha) + \alpha^2/2  + o_\zeta(1) +  o_\ep(1)    } .
\eqe 
\end{lem}
\begin{proof}
To lighten notation, we define the event
\alb
G_\alpha^\ep(z) 
&:= \left\{ \sup_{u,v\in B_\ep(z)} D_h(u,v  ) \in \left[ e^{-\xi} \ep^{\xi(Q-\alpha) + \zeta} , e^{ \xi} \ep^{\xi(Q-\alpha) - \zeta} \right] \right\} \notag\\
&\supset \left\{ \sup_{u,v\in B_\ep(z)}  D_{h+X\phi}(u,v  ) \in \left[  \ep^{\xi(Q-\alpha) + \zeta} , \ep^{\xi(Q-\alpha) - \zeta} \right] \right\}  .
\ale
Note that the inclusion holds by Weyl scaling (Axiom~\ref{item-metric-f}) and since $X\in [0,1]$. 

By Lemma~\ref{lem-diam-tail} (applied with $\alpha  \pm \zeta/\xi$ in place of $\alpha$), for each $z\in V$,
\eqb \label{eqn-use-diam-tail}
\BB P[G_\alpha^\ep(z)] \leq \ep^{ \alpha^2/2 + o_\zeta(1) + o_\ep(1)   } . 
\eqe
Since $G_\alpha^\ep(z) \in \sigma(h)$, Lemma~\ref{lem-scaled-field-dist} implies that  
\eqb \label{eqn-use-scaled-field-dist}
\BB P\left[ D_{h+X\phi}(0,z  ) \in \left[\BB s - \ep^{\xi(Q-\alpha)-\zeta} ,\BB s+  \ep^{\xi(Q-\alpha)-\zeta} \right] \,\big| \, h \right] \BB 1_{G_\alpha^\ep(z)} 
\preceq \frac{  \ep^{\xi(Q-\alpha)-\zeta}}{D_h(0,\bdy O )} \BB 1_{G_\alpha^\ep(z)}    ,
\eqe
with a deterministic implicit constant depending only on $\gamma$. 

Taking unconditional expectations of both sides of~\eqref{eqn-use-scaled-field-dist} and recalling that $ E_\alpha^\ep(z ; h + X\phi)$ is defined as in~\eqref{eqn-one-pt-event} but with $h+X\phi$ in place of $h$ shows that
\eqb \label{eqn-uncond-prob}
\BB P\left[   E_\alpha^\ep(z; h+X\phi) \right] \preceq  \ep^{\xi(Q-\alpha)-\zeta} \BB E\left[  D_h(0,\bdy O)^{-1} \BB 1_{G_\alpha^\ep(z)} \right]  . 
\eqe
By~\cite[Proposition 3.1]{lqg-metric-estimates}, $D_h(0, \bdy O)$ has finite moments of all negative orders. 
By H\"older's inequality, if $p,q> 1$ with $1/p  + 1/q = 1$ then  
\allb \label{eqn-uncond-holder}
\BB E\left[  D_h(0,\bdy O)^{-1} \BB 1_{G_\alpha^\ep(z)} \right]
&\preceq \BB E\left[  D_h(0,\bdy O)^{-p} \right]^{1/p} \BB P\left[ G_\alpha^\ep(z)  \right]^{1/q} \notag \\
&\leq \ep^{ \alpha^2/(2q) + o_\zeta(1) +  o_\ep(1)    } \quad \text{(by~\eqref{eqn-use-diam-tail})} . 
\alle 
Sending $q \rta 1$ and plugging~\eqref{eqn-uncond-holder} into~\eqref{eqn-uncond-prob} gives~\eqref{eqn-event-upper}.
\end{proof}

\begin{proof}[Proof of Proposition~\ref{prop-event-upper}]
By a standard Radon-Nikodym derivative calculation for the GFF (see Lemma~\ref{lem-gff-abs-cont}), the law of $h $ is absolutely continuous with respect to the conditional law of $h+X\phi$ given $X$, with Radon-Nikodym derivative
\eqb
M_X
= M_X(h  + X\phi) 
= \exp\left(   - X (h  + X\phi ,   \phi )_\nabla  +  \frac{X^2}{2} (\phi,\phi)_\nabla \right)  .
\eqe 
Here we emphasize that $\phi$ vanishes on $\bdy\BB D$ so $h$ and $h+X\phi$ have the same choice of additive constant.
By H\"older's inequality and since $(h,\phi)_\nabla$ is Gaussian with variance $(\phi,\phi)_\nabla$, for any exponents $p,q>1$ with $1/p+1/q = 1$,  
\allb \label{eqn-rn-holder}
\BB P\left[  E_\alpha^\ep(z) \,|\, X \right]
&= \BB E\left[ M_X  \BB 1_{ E_\alpha^\ep(z ; h+X\phi )} \,|\, X \right] \notag \\ 
&\leq \BB E\left[ M_X^p \,|\, X \right]^{1/p} \BB P\left[   E_\alpha^\ep(z ; h+X\phi) \,|\, X \right]^{1/q} \notag \\
&= \exp\left( \left( \frac{p - 1}{2}   \right) X^2 (\phi,\phi)_\nabla  \right) \BB P\left[   E_\alpha^\ep(z ; h+X\phi) \,|\, X \right]^{1/q} \notag \\
&\preceq \BB P\left[  E_\alpha^\ep(z ; h+X\phi) \,|\, X \right]^{1/q}  , 
\alle
with the implicit constant depending only on $p ,\gamma$ (equivalently, only on $q,\gamma$), where in the last line we used that $X$ takes values in $[0,1]$. 
Taking unconditional expectations of both sides of~\eqref{eqn-rn-holder} and using that $x\mapsto x^{1/q}$ is concave (to bring the $1/q$ outside the outer expectation) gives
\eqb
\BB P\left[  E_\alpha^\ep(z) \right] \preceq  \BB P\left[  E_\alpha^\ep(z; h+X\phi) \right]^{1/q} .
\eqe
Since $q$ can be made arbitrarily close to 1, we see that~\eqref{eqn-event-upper} follows from  Lemma~\ref{lem-event-upper'}.
\end{proof}

\subsection{Proofs of Hausdorff dimension upper bounds}
\label{sec-dim-upper-proof}

In several places in what follows, we will truncate on the H\"older continuity event
\eqb \label{eqn-holder-event}
\mcl H^\ep := \left\{ D_h(u,v) \in \left[\ep^{\xi(Q+2)+\zeta}, \ep^{\xi(Q-2)-\zeta} \right] ,\: \forall u,v\in B_{2\ep}(V) \: \text{with $|u-v| \leq 2\ep$} \right\} .
\eqe
By~\cite[Theorem 1.7]{lqg-metric-estimates}, $\BB P\left[\mcl H^\ep\right] \rta 1$ as $\ep\rta 0$. 
In order to prove the upper bounds for $\dim_{\mcl H}^0 \bdy \mcl B_{\BB s}$ and $\dim_{\mcl H}^\gamma \bdy\mcl B_{\BB s}$ in Theorem~\ref{thm-bdy-dim}, we will need the following one-point estimate, which comes from taking the ``worst-case" value of $\alpha$ in Proposition~\ref{prop-event-upper}. 

\begin{prop} \label{prop-bdy-prob} 
For each $z\in V  $ and each $\ep > 0$, 
\eqb \label{eqn-bdy-prob} 
\BB P\left[ B_\ep(z) \cap \bdy \mcl B_{\BB s}  \not= \emptyset , \: \mcl H^\ep \right] \leq \ep^{\xi Q  - \xi^2/2 + o_\zeta(1) + o_\ep(1)} .
\eqe 
Furthermore, for each $p  \in [0,2d_\gamma/\gamma-1]$, 
\eqb \label{eqn-bdy-moment}
\BB E\left[ \left(\sup_{u,v\in B_\ep(z)} D_h(u,v) \right)^p \BB 1\left\{B_\ep(z) \cap \bdy \mcl B_{\BB s}  \not= \emptyset ,\: \mcl H^\ep \right\} \right] 
\leq  \ep^{(p+1) \xi Q - (p+1)^2\xi^2/2 + o_\zeta(1) + o_\ep(1)}. 
\eqe
\end{prop}
\begin{proof}
The basic idea is to consider the ``worst case" value of $\alpha$ in Proposition~\ref{prop-event-upper}. 
Since there are uncountably many possibilities for $\alpha$, we first need to discretize the set of possibilities.

Fix a partition $-2 = \alpha_0 < \dots < \alpha_N = 2$ with $\sup_{j\in [1,N]_{\BB Z}} (\alpha_j - \alpha_{j-1}) \leq \zeta/\xi$. We can arrange that $N$ depends only on $\alpha,\zeta$. Then the intervals $[\xi(Q-\alpha_j)-\zeta , \xi(Q-\alpha_j)+\zeta]$ for $j \in [0,N]_{\BB Z}$ cover $[-2-\zeta,2+\zeta]$.
If the event $\mcl H^\ep$ of~\eqref{eqn-holder-event} occurs, then $\sup_{u,v\in B_\ep(z)} D_h(u,v ) \in [\ep^{\xi(Q+2)+\zeta} , \ep^{\xi(Q-2)-\zeta}]$ for each $z \in V$, so for each such $z$ we have $\sup_{u,v\in B_\ep(z)} D_h(u,v ) \in [\ep^{\xi(Q-\alpha_j)+\zeta}, \ep^{\xi(Q-\alpha_j)-\zeta}]$ for some $j\in [0,N]_{\BB Z}$. 
By combining this with Lemma~\ref{lem-one-pt-contain}, we see that if $\mcl H^\ep$ occurs and $B_\ep(z) \cap \bdy\mcl B_{\BB s}  \not=\emptyset$, then $E_{\alpha_j}^\ep(z)$ occurs for some $j\in [0,N]_{\BB Z}$. 

By Proposition~\ref{prop-event-upper}, for each $z\in V$,
\eqb \label{eqn-bdy-prob-sum}
\BB P\left[ B_\ep(z) \cap \bdy \mcl B_{\BB s}  \not= \emptyset , \: \mcl H^\ep \right]
\leq \sum_{j=0}^N \BB P\left[ E_{\alpha_j}^\ep(z) \right] 
\leq \sum_{j=0}^N \ep^{\xi(Q-\alpha_j) + \alpha_j^2/2 + o_\zeta(1) + o_\ep(1)}  .
\eqe
The maximum over all $\alpha \in [-2,2]$ of the quantity $\xi(Q-\alpha)  + \alpha^2/2$ is attained at $\alpha=\xi$, where it equals $\xi Q -\xi^2/2$. 
Hence the right side of~\eqref{eqn-bdy-prob-sum} is at most $(N+1) \ep^{\xi Q - \xi^2/2 + o_\zeta(1) + o_\ep(1)} = \ep^{\xi Q - \xi^2/2 + o_\zeta(1) + o_\ep(1)}$. 
This gives~\eqref{eqn-bdy-prob}.

To prove~\eqref{eqn-bdy-moment}, we make a similar computation using Proposition~\ref{prop-event-upper}: 
\allb \label{eqn-bdy-moment-sum}
\BB E\left[ \left(\sup_{u,v\in B_\ep(z)} D_h(u,v) \right)^p \BB 1\left\{B_\ep(z) \cap \bdy \mcl B_{\BB s}  \not= \emptyset ,\: \mcl H^\ep \right\} \right]
&\leq \sum_{j=0}^N \ep^{   \xi(Q-\alpha_j) p  +o_\zeta(1)} \BB P\left[ E_{\alpha_j}^\ep(z) \right] \notag\\
&\leq \sum_{j=0}^N \ep^{  \xi(Q-\alpha_j) (p+1) + \alpha_j^2/2 + o_\zeta(1) + o_\ep(1)} .
\alle
The maximum over all $\alpha \in [-2,2]$ of the quantity $ \xi(Q-\alpha) (p+1)  + \alpha^2/2$ is attained at $\alpha= (p+1)\xi$ (which is in $[-2,2]$ for $p\in [0,2d_\gamma/\gamma-1]$), where it equals $(p+1) \xi Q -  (p+1)^2 \xi^2/2$. 
Combining this with~\eqref{eqn-bdy-moment-sum} gives~\eqref{eqn-bdy-moment}. 
\end{proof}

\begin{proof}[Proof of Theorem~\ref{thm-bdy-dim}, upper bound] 
By letting $V$ increase to all of $\BB C\setminus \{0\}$ and using the countable stability of Hausdorff dimension, we see that it suffices to show that for any fixed choice of $V$ as in the beginning of Section~\ref{sec-one-pt}, a.s.\ 
\eqb \label{eqn-bdy-dim-show}
\dim_{\mcl H}^0(\bdy \mcl B_{\BB s}  \cap V) \leq 2 - \xi Q + \xi^2/2
\quad \text{and} \quad
\dim_{\mcl H}^\gamma(\bdy\mcl B_{\BB s}  \cap V ) \leq d_\gamma-1  .
\eqe
\medskip

\noindent\textit{Proof of~\eqref{eqn-bdy-dim-show} for Euclidean dimension}.
Choose a finite collection $\mcl Z^\ep$ of at most $O_\ep(\ep^{-2})$ points in $V$ such that the union of the balls $B_\ep(z)$ for $z\in\mcl Z^\ep$ covers $V$. 
By~\eqref{eqn-bdy-prob} of Proposition~\ref{prop-bdy-prob}, 
\eqb
\BB E\left[ \#\left\{ z\in\mcl Z^\ep : B_\ep(z) \cap \bdy \mcl B_{\BB s}  \not=\emptyset \right\} \BB 1_{\mcl H_\ep} \right] \leq \ep^{-2 + \xi Q - \xi^2/2 + o_\zeta(1) + o_\ep(1) } .
\eqe 
Since $\BB P[\mcl H^\ep] \rta 1$ as $\ep\rta 0$ and by Markov's inequality, it holds with probability tending to 1 as $\ep\rta 0$ that $\bdy B_{\BB s}  \cap V$ can be covered by at most $\ep^{-2 + \xi Q -\xi^2/2 + o_\zeta(1) + o_\ep(1)}$ Euclidean balls of radius $\ep$. 
In particular, a.s.\ there is a (random) sequence of $\mcl E$ of $\ep$-values tending to zero such that for each $\ep \in \mcl E$, $\bdy B_{\BB s}  \cap V$ can be covered by at most $\ep^{-2 + \xi Q -\xi^2/2 + o_\zeta(1) + o_\ep(1)}$ Euclidean balls of radius $\ep$. 
By the definition of Hausdorff dimension, it follows that a.s.\ $\dim_{\mcl H}^0(\bdy B_{\BB s}  \cap V ) \leq 2 - \xi Q + \xi^2/2 + o_\zeta(1)$. Sending $\zeta \rta 0$ now shows that a.s.\ $\dim_{\mcl H}^0(\bdy B_{\BB s}  \cap V ) \leq 2 - \xi Q + \xi^2/2$, as required.
\medskip

\noindent\textit{Proof of~\eqref{eqn-bdy-dim-show} for quantum dimension}.
By summing~\eqref{eqn-bdy-moment} of Proposition~\ref{prop-bdy-prob} over all $z\in\mcl Z^\ep$, we get that for each $p\in [0,2d_\gamma/\gamma-1] $, 
\eqb \label{eqn-quantum-bdy-dim-moment}
\BB E\left[\sum_{z\in\mcl Z^\ep} \left(\sup_{u,v\in B_\ep(z)} D_h(u,v) \right)^p \BB 1\left\{B_\ep(z) \cap \bdy \mcl B_{\BB s}   \not= \emptyset ,\: \mcl H^\ep \right\} \right] \leq \ep^{\xi Q (p+1) - (p+1)^2\xi^2/2 - 2 + o_\zeta(1) + o_\ep(1)} .
\eqe 
Recalling the definitions of $\xi$ and $Q$ from~\eqref{eqn-xi-Q}, we see that $\xi Q (p+1) - (p+1)^2 \xi^2/2  -2  = 0$ for $p = d_\gamma-1$. Furthermore, the quadratic function $p\mapsto \xi Q (p+1) - (p+1)^2 \xi^2/2  -2$ attains its maximum at 
\eqbn
p = Q/\xi - 1  = \left(\frac{2}{\gamma^2} + \frac12 \right) d_\gamma - 1 > d_\gamma-1 ,
\eqen
 where it equals $Q^2/2-2 > 0$. Since this function is quadratic, it follows that $\xi Q (p+1) - (p+1)^2 \xi^2/2  -2 > 0$ for each $p  \in (d_\gamma - 1 , Q/\xi -1)$. 
Hence, for such a choice of $p$ and a small enough choice of $\zeta > 0$ the right side of~\eqref{eqn-quantum-bdy-dim-moment} tends to zero as $\ep\rta 0$. 
Since $\BB P[\mcl H^\ep] \rta 1$ as $\ep\rta 0$ and by Markov's inequality, it holds with probability tending to 1 as $\ep\rta 0$ that $\bdy B_{\BB s} \cap V$ can be covered by a collection of Euclidean balls such that the sum of the $p$th powers of their $D_h$-diameters tends to zero as $\ep\rta 0$. 
Sending $\zeta \rta 0$ and $p\rta (d_\gamma-1)^+$ now shows that a.s.\ $\dim_{\mcl H}^\gamma(\bdy B_{\BB s}  \cap V ) \leq d_\gamma-1$, as required.
\end{proof}

The following lemma will allow us to deduce the upper bound for quantum dimension in Theorem~\ref{thm-thick-dim} from the upper bound for Euclidean dimension. 

\begin{lem} \label{lem-thick-dim-compare}
Let $\alpha\in [-2,2]$ and define the set of metric $\alpha$-thick points $\wh{\mcl T}_h^\alpha$ as in~\eqref{eqn-metric-thick}. 
Almost surely, it holds simultaneously for every Borel set $X\subset \BB C$ that
\eqb \label{eqn-thick-dim-compare}
\dim_{\mcl H}^\gamma\left(X\cap \wh{\mcl T}_h^\alpha\right) \leq \frac{\dim_{\mcl H}^0\left(X\cap \wh{\mcl T}_h^\alpha\right) }{\xi(Q-\alpha)} . 
\eqe
\end{lem}
\begin{proof}
Fix $\zeta > 0$ and for $\delta > 0$, let
\eqb \label{eqn-thick-uniform-def}
X^\alpha(\delta) := \left\{z\in X :  \sup_{u,v\in B_r(z)} D_h(u,v )  \in \left[ r^{\xi (Q-\alpha) +  \zeta } , r^{\xi (Q-\alpha) - \zeta } \right] ,\, \forall r \in (0,\delta] \right\} .
\eqe
Then $ X\cap \wh{\mcl T}_h^\alpha \subset \bigcup_{n\in\BB N} X^\alpha(1/n)$. By the countably stability of Hausdorff dimension and since $\zeta>0$ is arbitrary, it suffices to show that for each fixed $\zeta,\delta> 0$, a.s.\ 
\eqb \label{eqn-thick-dim-show}
\dim_{\mcl H}^\gamma X^\alpha(\delta) \leq \frac{1}{\xi(Q-\alpha)- \zeta} \dim_{\mcl H}^0\left(X\cap \wh{\mcl T}_h^\alpha\right) .
\eqe

To this end, let $\Delta  >  \dim_{\mcl H}^0\left(X\cap \wh{\mcl T}_h^\alpha\right)$ (which is at least the Euclidean dimension of $X^\alpha(\delta)$), let $\ep >0$, and let $\{B_j\}_{j\in \BB N}$ be a cover of $X^\alpha(\delta)$ by Euclidean balls of radii $r_j \in (0,\delta ]$ such that $\sum_{j\in\BB N} r_j^\Delta \leq \ep$.  
We can arrange that each $B_j$ is centered at a point of $X^\alpha(\delta)$.
By the definition~\eqref{eqn-thick-uniform-def} of $X^\alpha(\delta)$ and since each $r_j$ is at most $\delta$,  
\eqb
\sum_{j\in\BB N} \left( \sup_{u,v\in B_j} D_h(u,v) \right)^{\Delta/(\xi(Q-\alpha)-\zeta)} \leq \sum_{j\in\BB N } r_j^\Delta \leq \ep .
\eqe
Since $\ep > 0$ is arbitrary and $\Delta$ can be made arbitrarily close to $\dim_{\mcl H}^0\left(X\cap \wh{\mcl T}_h^\alpha\right)$, we obtain~\eqref{eqn-thick-dim-show}.
\end{proof}

\begin{proof}[Proof of Theorem~\ref{thm-thick-dim}, upper bound]
We will prove the upper bound for $\dim_{\mcl H}^0 (\bdy\mcl B_{\BB s} \cap \wh{\mcl T}_h^\alpha)$. This immediately implies the desired upper bound for $\dim_{\mcl H}^\gamma(\bdy\mcl B_{\BB s} \cap \wh{\mcl T}_h^\alpha)$ due to Lemma~\ref{lem-thick-dim-compare}.
\medskip
 
\noindent\textit{Step 1: reductions.}
As in the discussion immediately preceding~\eqref{eqn-bdy-dim-show}, it suffices to show that for $\alpha \in [\xi -\sqrt{4-2\xi Q +\xi^2} ,\xi + \sqrt{4-2\xi Q +\xi^2}]$ and $V$ as above, a.s.\ 
\eqb \label{eqn-metric-thick-show}
\dim_{\mcl H}^0\left(  \bdy \mcl B_{\BB s}  \cap  \wh{\mcl T}_h^\alpha \cap V \right) \leq 2 - \xi(Q-\alpha) - \alpha^2/2 ;
\eqe
and for $\alpha \notin  [\xi -\sqrt{4-2\xi Q +\xi^2} ,\xi + \sqrt{4-2\xi Q +\xi^2}]$, a.s.\ $ \bdy \mcl B_{\BB s}  \cap  \wh{\mcl T}_h^\alpha \cap V =\emptyset$. 

For $\delta > 0$, let
\eqb \label{eqn-thick-truncate}
\wh{\mcl T}_h^\alpha(\delta) 
:= \left\{z\in \BB C : \sup_{u,v\in B_r(z)} D_h(u,v  )  \in \left[ r^{\xi (Q-\alpha) +  \zeta/2} , r^{\xi (Q-\alpha) - \zeta/2} \right] ,\, \forall r \in (0,\delta]  \right\} .
\eqe
Then $\wh{\mcl T}_h^{\alpha } \subset \bigcup_{n\in\BB N} \wh{\mcl T}_h^{\alpha }(1/n)$. 
By the countable stability of Hausdorff dimension, to prove~\eqref{eqn-metric-thick-show} it therefore suffices to show that for each fixed $\delta >0$ and $\alpha \in [\xi -\sqrt{4-2\xi Q +\xi^2} ,\xi + \sqrt{4-2\xi Q +\xi^2}]$, a.s.\ 
\eqb \label{eqn-metric-thick-show'}
\dim_{\mcl H}^0\left(  \bdy \mcl B_{\BB s}  \cap \wh{\mcl T}_h^{\alpha }(\delta)  \cap V \right) \leq 2 - \xi(Q-\alpha) - \alpha^2/2 ;
\eqe
and for each $\alpha \notin [\xi -\sqrt{4-2\xi Q +\xi^2} ,\xi + \sqrt{4-2\xi Q +\xi^2}]$, a.s.\ $ \bdy \mcl B_{\BB s} \cap \wh{\mcl T}_h^{\alpha }(\delta)  \cap V = \emptyset$. 
\medskip

\noindent\textit{Step 2: relating to $E_\alpha^\ep(z)$.}
To make the connection between~\eqref{eqn-metric-thick-show'} and the one-point estimate from Proposition~\ref{prop-event-upper}, we will argue that if $\ep \in (0,\delta/4]$ is chosen to be sufficiently small and $z\in V$ such that $B_\ep(z) \cap  \bdy \mcl B_{\BB s} \cap \wh{\mcl T}_h^{\alpha }(\delta)   \not=\emptyset$, then $E_\alpha^{2\ep}(z)$ occurs.

To see this, let $w\in B_\ep(z) \cap \bdy \mcl B_{\BB s} \cap \wh{\mcl T}_h^{\alpha }(\delta) $. 
By the definition~\eqref{eqn-thick-truncate} of $\wh{\mcl T}_h^\alpha(\delta)$, 
\eqb \label{eqn-ball-pt-diam}
\sup_{u,v\in B_{4\ep}(w)} D_h(u,v ) \leq (4\ep)^{\xi(Q-\alpha)-\zeta/2 }  
\quad \text{and} \quad
\sup_{u,v\in B_\ep(w)} D_h(u,v ) \geq \ep^{\xi(Q-\alpha)+ \zeta/2} .
\eqe
Since $B_\ep(w) \subset B_{2\ep}(z) \subset B_{4\ep}(w)$, if $\ep$ is chosen to be sufficiently small, then~\eqref{eqn-ball-pt-diam} implies that
\eqb
\sup_{u,v\in B_{2\ep}(z)} D_h(u,v  ) \in \left[(2\ep)^{\xi(Q-\alpha)+\zeta} , (2\ep)^{\xi(Q-\alpha)-\zeta}  \right] 
\eqe
and then Lemma~\ref{lem-one-pt-contain} implies that $E_\alpha^{2\ep}(z)$ occurs, as required.
\medskip

\noindent\textit{Step 3: proof of the upper bound for Euclidean dimension.}
Choose a finite collection $\mcl Z^\ep$ of at most $O_\ep(\ep^{-2})$ points in $V$ such that the union of the balls $B_\ep(z)$ for $z\in\mcl Z^\ep$ covers $V$. 
By the preceding step and Proposition~\ref{prop-event-upper}, for $z\in\mcl Z^\ep$,
\eqbn
\BB P\left[ B_\ep(z) \cap  \bdy \mcl B_{\BB s}  \cap \wh{\mcl T}_h^{\alpha }(\delta)  \not=\emptyset \right] 
\leq \BB P\left[ E_\alpha^{2\ep}(z) \right] 
\leq \ep^{\xi(Q-\alpha) + \alpha^2/2+ o_\zeta(1) + o_\ep(1)} .
\eqen
Summing over all $z\in\mcl Z^\ep$ then shows that
\eqb
\BB E\left[ \#\left\{z\in\mcl Z^\ep :  B_\ep(z) \cap  \bdy \mcl B_{\BB s}  \cap \wh{\mcl T}_h^{\alpha }(\delta)  \not=\emptyset \right\} \right] \leq \ep^{-2  + \xi (Q-\alpha)  + \alpha^2/2   + o_\zeta(1) + o_\ep(1)} .
\eqe
By Markov's inequality, it holds with probability tending to 1 as $\ep\rta 0$ that $ \bdy \mcl B_{\BB s}  \cap \wh{\mcl T}_h^{\alpha }(\delta) \cap V$ can be covered by at most $\ep^{-2  + \xi (Q-\alpha)  + \alpha^2/2   + o_\zeta(1) + o_\ep(1)} $ Euclidean balls of radius $\ep$. 
If $\alpha \in [\xi -\sqrt{4-2\xi Q +\xi^2} ,\xi + \sqrt{4-2\xi Q +\xi^2}]$, then sending $\zeta \rta 0$ gives~\eqref{eqn-metric-thick-show'}.  

If $\alpha \notin [\xi -\sqrt{4-2\xi Q +\xi^2} ,\xi + \sqrt{4-2\xi Q +\xi^2}]$, then $-2  + \xi (Q-\alpha)  + \alpha^2/2  >0$. From this, we get that if $\zeta$ is chosen to be sufficiently small, then with probability tending to 1 as $\ep\rta 0$, there are no points $z\in\mcl Z^\ep$ such that $ B_\ep(z) \cap  \bdy \mcl B_{\BB s}  \cap \wh{\mcl T}_h^{\alpha }(\delta)  \not=\emptyset$. Since the balls $B_\ep(z)$ for $z\in\mcl Z^\ep$ cover $\mcl Z^\ep$, this implies that a.s.\ $\bdy \mcl B_{\BB s}  \cap \wh{\mcl T}_h^{\alpha }(\delta) \cap V = \emptyset$. 
\end{proof}

\subsection{Generalized upper bound}
\label{sec-gen-upper}

The argument of the preceding two subsections in fact yields upper bounds for the Hausdorff dimensions of a wide variety of distinguished subsets of $\bdy \mcl B_{\BB s}$. 
To be precise, we state just below a theorem which is proven in essentially the same manner as the upper bounds in Theorems~\ref{thm-bdy-dim} and~\ref{thm-thick-dim}.  
This theorem is not used in the rest of the present paper, so the reader who is only interested in seeing the proofs of the results already stated can safely skip this subsection. 
Several applications of the theorem of this subsection will appear in forthcoming joint work by the author, J.\ Pfeffer, and S.\ Sheffield.

\begin{thm} \label{thm-gen-upper}
Suppose that we are given events $\{F^\ep(z) : \ep > 0 , z\in\BB C\}$ and non-negative exponents $\{q_\alpha\}_{\alpha\in [-2,2]}$ with the following property. 
For any $\alpha\in [-2,2]$, any $\zeta \in (0,1)$, any bounded open set $V\subset\BB C$ with $\ol V\subset \BB C\setminus \{0\}$, the following is true.
\begin{enumerate}
\item For each $z\in V$,
\eqb \label{eqn-gen-upper-prob}
\BB P\left[ F^\ep(z) \cap \left\{  \sup_{u,v\in B_\ep(z)} D_h(u,v )  \in \left[  \ep^{\xi(Q-\alpha) + \zeta} , \ep^{\xi(Q-\alpha) - \zeta} \right]   \right\} \right] 
\leq \ep^{\alpha^2/2 + q_\alpha  + o_\zeta(1)+ o_\ep(1) }, 
\eqe
where the rate of the $o_\zeta(1)$ depends only on $\alpha, \gamma$ and the rate of the $o_\ep(1)$ depends only on $ V, \alpha,\zeta,\gamma $ (not on the particular choice of $z$).  
\item There exists an open set $U \subset \BB C$ which contains zero and lies at positive distance from $V$ such that for each small enough $\ep > 0$ (depending on $V$), each of the events $F^\ep(z)$ for $z\in V$ is a.s.\ determined by $h|_{\BB C\setminus U}$.  \label{item-gen-upper-local} 
\end{enumerate}
For $\BB s  >0$, let $\mcl Y_{\BB s}$ be the set of $z\in\bdy\mcl B_{\BB s}$ such that
\eqb
\bigcup_{\ep > 0} \bigcap_{r\in (0,\ep) \cap \BB Q} \bigcap_{w\in B_r(z) \cap \BB Q^2} F^{ r}(w) 
\eqe
occurs, i.e., for each small enough rational $r > 0$, the event $F^{ r}(w)$ occurs for every $w\in B_r(z)\cap \BB Q^2 $ (we consider rational values of $r$ and $w$ to avoid measurability issues).  
Then a.s.\ for each $\alpha \in \BB R$, 
\allb \label{eqn-gen-upper-alpha}
\dim_{\mcl H}^0 \left( \mcl Y_{\BB s} \cap \wh{\mcl T}_h^\alpha \right) &\leq \max\left\{0 , 2-\xi(Q-\alpha) - \alpha^2/2 - q_\alpha \right\} \quad \text{and} \notag\\
\dim_{\mcl H}^\gamma \left( \mcl Y_{\BB s} \cap \wh{\mcl T}_h^\alpha \right) &\leq \max\left\{0 , \frac{2 - \alpha^2/2 -q_\alpha}{ \xi(Q-\alpha) } -1 \right\}  .
\alle
Furthermore, a.s.\ 
\allb \label{eqn-gen-upper-full} 
\dim_{\mcl H}^0 \mcl Y_{\BB s} &\leq \max\left\{ 0 , \sup_{ \alpha\in [-2,2] } \left( 2-\xi(Q-\alpha) - \alpha^2/2 - q_\alpha \right)  \right\} \quad \text{and} \quad \notag\\
\dim_{\mcl H}^\gamma \mcl Y_{\BB s} &\leq \max\left\{ 0 ,  \sup_{ \alpha\in [-2,2]} \left(\frac{2 - \alpha^2/2 -q_\alpha}{ \xi(Q-\alpha) } -1 \right) \right\} .
\alle 
\end{thm}

We now explain how to adapt the arguments of the preceding subsections to get Theorem~\ref{thm-gen-upper}. 
Fix $\zeta \in (0,1)$ and an open set $V$ with $\ol V\subset \BB C\setminus \{0\}$. 
We define the events $E_\alpha^\ep(z)$ for $\alpha\in [-2,2]$, $z \in \BB C$, and $\ep > 0$ as in~\eqref{eqn-one-pt-event}. 
The following proposition is the generalization of Proposition~\ref{prop-event-upper} to our setting.

\begin{prop} \label{prop-event-gen-upper}
For each $\alpha \in [-2  ,2 ]$, each $z\in V$, and each $\ep > 0$,  
\eqb 
\BB P\left[ E_\alpha^\ep(z) \cap F^\ep(z) \right] \leq \ep^{   \xi(Q-\alpha) + \alpha^2/2 + q_\alpha + o_\zeta(1) +  o_\ep(1)    }
\eqe 
where the rate of the $o_\zeta(1)$ depends only on $\alpha, \gamma$ and the rate of the $o_\ep(1)$ depends only on $ V, \alpha,\zeta,\gamma $ (not on the particular choice of $z$).
\end{prop}
\begin{proof}
This follows from exactly the same argument used to prove Proposition~\ref{prop-event-upper}, except that~\eqref{eqn-gen-upper-prob} is used in place of Lemma~\ref{lem-diam-tail} to bound the probability of the ``short range" event. Note that we can arrange that the bump function $\phi$ introduced just above Lemma~\ref{lem-scaled-field-dist} is supported on the set $U$ from condition~\ref{item-gen-upper-local} in Theorem~\ref{thm-gen-upper}. Said condition then ensures that adding a multiple of $\phi$ to $h$ does not change the occurrence of $F^\ep(z)$, i.e., $F^\ep(z)$ occurs for $h$ if and only if it occurs with $h+X\phi$ in place of $h$. 
\end{proof}

With Proposition~\ref{prop-event-gen-upper} in hand, we can now follow exactly the same argument as in Section~\ref{sec-dim-upper-proof} to obtain the dimension upper bounds asserted in Theorem~\ref{thm-gen-upper}.

\section{Outline of the lower bound proof}
\label{sec-outline}

The rest of the paper is devoted to proving the lower bounds for Hausdorff dimension in Theorems~\ref{thm-bdy-dim} and~\ref{thm-thick-dim}.
Before proceeding with the proofs, in this section we give an outline of the argument. 

Our main task is to prove the lower bound
\eqb \label{eqn-thick-dim-outline} 
\esssup \dim_{\mcl H}^0\left( \bdy\mcl B_{\BB s} \cap \mcl T_h^\alpha \cap \wh{\mcl T}_h^\alpha \right) \geq 2-\xi(Q-\alpha) - \alpha^2/2
\eqe
 from Theorem~\ref{thm-thick-dim}. The desired lower bound for $\esssup \dim_{\mcl H}^\gamma(\bdy\mcl B_{\BB s} \cap \mcl T_h^\alpha \cap \wh{\mcl T}_h^\alpha) $ follows from~\eqref{eqn-thick-dim-outline} and a general result relating the Euclidean and quantum dimensions of subsets of the set of $\alpha$-thick points~\cite[Proposition 2.5]{gp-kpz}. 
The lower bound for the $\esssup$ of the Euclidean (resp.\ quantum) dimension of $ \mcl B_{\BB s}$ from Theorem~\ref{thm-bdy-dim} follows from the lower bound for the $\esssup$ of the Euclidean (resp.\ quantum) dimension of $ \mcl B_{\BB s} \cap\mcl T_h^\alpha \cap  \wh{\mcl T}_h^\alpha$ applied with $\alpha = \xi$ (resp.\ $\alpha=\gamma$). 
Hence, we just need to prove~\eqref{eqn-thick-dim}.
\medskip

\noindent\textbf{General strategy.} 
We will prove~\eqref{eqn-thick-dim-outline} using the standard strategy for establishing lower bounds for Hausdorff dimensions. 
That is, we will fix a sequence $\{\ep_n\}_{n\in\BB N_0}$ which converges to zero at a factorial rate and we will define events $\mcl G^n(z)$ for $n\in\BB N$ and $z\in \BB A_{1/4,1/3}(0)$ which satisfy (roughly speaking) the following properties.
\begin{enumerate}[A.]
\item \textbf{Estimates for circle averages and distances on $\mcl G^n(z)$.} If $\mcl G^n(z)$ occurs, then for each $r \in [\ep_n,\ep_0] $ we have $h_r(z) =  (\alpha + o_r(1)) \log r^{-1}$ and $\sup_{u,v\in \BB A_{\ep_n,r}(z)} D_h(u,v)  = r^{\xi(Q-\alpha) + o_r(1) }$. Furthermore, $D_h(0,B_{\ep_n}(z)) \in [\BB s , \BB s + \ep_n^{\xi(Q-\alpha) + o_n(1)} ]$ (see Lemma~\ref{lem-perfect-contain}). \label{item-G-perfect}
\item \textbf{One-point estimate.} For each $z\in \BB A_{1/4,1/3}(0)$ we have $\BB P\left[\mcl G^n(z) \right] =  \ep_n^{\xi(Q-\alpha)  + \alpha^2/2 + o_n(1)}$ (Proposition~\ref{prop-end-1pt}).  \label{item-G-1pt}
\item \textbf{Two-point estimate.} For each distinct $z,w\in  \BB A_{1/4,1/3}(0)$, we have $\BB P\left[ \mcl G^n(z) \cap \mcl G^n(w) \right] \leq |z-w|^{-\xi(Q-\alpha) - \alpha^2/2 - o_{|z-w|}(1)} \BB P\left[\mcl G^n(z) \right] \BB P\left[\mcl G^n(w) \right]$, where the rate of convergence of the $o_{|z-w|}(1)$ does not depend on $n$ or on the particular choices of $z$ and $w$ (Proposition~\ref{prop-end-2pt}). \label{item-G-2pt}
\end{enumerate}
This will allow us to show via standard arguments that for each $\Delta < 2-\xi(Q-\alpha) - \alpha^2/2$, it holds with positive probability that there is a so-called \emph{Frostman measure} $\nu$ on $\bdy\mcl B_{\BB s} \cap \mcl T_h^\alpha \cap \wh{\mcl T}_h^\alpha$, i.e., $\nu$ satisfies
\eqb
\nu\left(\bdy\mcl B_{\BB s} \cap \mcl T_h^\alpha \cap \wh{\mcl T}_h^\alpha  \right) > 0 \quad \text{and} \quad
\iint_{(\bdy\mcl B_{\BB s} \cap \mcl T_h^\alpha \cap \wh{\mcl T}_h^\alpha) \times (\bdy\mcl B_{\BB s} \cap \mcl T_h^\alpha \cap \wh{\mcl T}_h^\alpha)} |z-w|^{-\Delta} \,d\nu(z) \,d\nu(w) < \infty .
\eqe
By Frostman's lemma~\cite[Theorem 4.27]{peres-bm}, whenever such a measure $\nu$ exists we have $\dim_{\mcl H}^0 (\bdy\mcl B_{\BB s} \cap \mcl T_h^\alpha) \geq \Delta$.
We therefore obtain~\eqref{eqn-thick-dim-outline}.

The remainder of this paper is devoted to defining the events $\mcl G^n(z)$ and proving the above three properties. 
Like the events $E_\alpha^\ep(z)$ used in Section~\ref{sec-upper-bound}, the event $\mcl G^n(z)$ will be the intersection of a ``short-range" event $\mcl E^n(z)$ (which involves conditions for the GFF on a neighborhood of $z$) and a ``long-range" event $G_{z,n}$ (which controls $D_h(0,z)$).
The short-range event is dealt with using estimates for the GFF and the $\gamma$-LQG metric. The long-range event is dealt with by adding a random smooth function to the GFF and using that this affects the law of the GFF in an absolutely continuous way, as in Section~\ref{sec-upper-bound}. 
\medskip

\noindent\textbf{Short-range event.}
In \textbf{Section~\ref{sec-short-range}}, we focus exclusively on the short-range events $\mcl E^n(z)$. 
The definition of the events is given in Section~\ref{sec-short-def}. 
We will first define for each $j\in\BB N_0$ an event $E_{z,j}$ which depends on $h|_{\BB A_{\ep_j , \ep_{j-1}/3}(z)}$ and events $H_{z,j}^{\op{out}}$ and $H_{z,j}^{\op{in}}$ which depend on $h|_{\BB A_{\ep_j/3,\ep_j}(z)}$. 

The event $E_{z,j}$ consists of the condition that $h_{\ep_j}(z) - h_{\ep_{j-1}}(z) \approx \alpha\log(\ep_{j-1}/\ep_j)$ (which will ensure that $z$ is ``approximately $\alpha$-thick" on $\mcl E^n(z)$) plus a list of regularity conditions which have constant-order probability.
These regularity conditions have several different purposes, such as preventing $D_h(0,z)$ from being extremely large or small and controlling the conditional probability of $E_{z,j}$ given $h|_{\BB C\setminus B_{\ep_{j-1}}(z)}$ on the event $E_{z,j-1}$. 
The purpose of each regularity condition is explained in Remark~\ref{remark-E}. 

The purpose of the events $H_{z,j}^{\op{in}}$ and $H_{z,j}^{\op{out}}$ is to prevent pathological behavior on the intermediate annuli $\BB A_{\ep_j/3,\ep_j}(z) $. The events $E_{z,j}$ cannot include conditions which depend on the restrictions of $h$ to these annuli. Indeed, we need $E_{z,j}$ to depend on the restriction of $h$ to a compact subset of $B_{\ep_{j-1}}(z)$ in order to control the conditional probability of $E_{z,j}$ given $h|_{\BB C\setminus B_{\ep_{j-1}}(z)}$ on the event $E_{z,j-1}$. 

We set $\mcl E^n(z) := \bigcap_{j=0}^n E_{z,j} \cap \bigcap_{j=0}^{n-1}( H_{z,j}^{\op{in}}\cap H_{z,j}^{\op{out}} )$, so that $\mcl E^n(z)$ is the event that we have the desired behavior up to scale $n$. In Sections~\ref{sec-short-twopoint} and~\ref{sec-short-estimates}, will establish analogs of the above three properties for the events $\mcl E^n(z)$. 
\begin{enumerate}[A.]
\item \textbf{Estimates for circle average and distance on $\mcl E^n(z)$.} If $\mcl E^n(z)$ occurs, then for each $r \in [\ep_n,\ep_0]$ we have $h_r(z) = (\alpha + o_r(1)) \log r^{-1}$ and $\sup_{u,v\in \BB A_{\ep_n,r}(z)} D_h(u,v)  =  r^{\xi(Q-\alpha) + o_r(1)}$. Furthermore, $D_h(0,z)$ is of constant order (Lemmas~\ref{lem-E-thick} and~\ref{lem-E-dist}). \label{item-E-perfect}
\item \textbf{One-point estimate for $\mcl E^n(z)$.} For $z\in \BB A_{1/4,1/3}(0)$ we have $\BB P\left[\mcl E^n(z) \right] =  \ep_n^{  \alpha^2/2 + o_n(1)}$ (Lemma~\ref{lem-thick-1pt}).  \label{item-E-1pt}
\item \textbf{Two-point estimate for $\mcl E^n(z)$.} For distinct $z,w\in  \BB A_{1/4,1/3}(0)$, we have $\BB P\left[ \mcl E^n(z) \cap \mcl E^n(w) \right] \leq |z-w|^{-  \alpha^2/2 - o_{|z-w|}(1)} \BB P\left[\mcl E^n(z) \right] \BB P\left[\mcl E^n(w) \right]$, where the rate of convergence of the $o_{|z-w|}(1)$ does not depend on $n$ or on the particular choices of $z$ and $w$ (Lemma~\ref{lem-thick-2pt}). \label{item-E-2pt}
\end{enumerate}
These properties are established using the definition of $\mcl E^n(z)$ together with standard local absolute continuity properties of the GFF and estimates for the LQG metric from~\cite{lqg-metric-estimates}. 
The most involved step is showing that $\BB P[E_{z,j} ] \geq (\ep_j/\ep_{j-1})^{\alpha^2/2+o_j(1)}$ (Proposition~\ref{prop-E-prob}).
We require non-trivial estimates to show that some of the regularity conditions in the definition of $E_{z,j}$ occur with high probability (see the discussion just after Proposition~\ref{prop-E-prob} for details) so the proof of this proposition is deferred until \textbf{Section~\ref{sec-onescale-prob}} to avoid interrupting the main argument.
\medskip

\noindent\textbf{Long-range event.}
In \textbf{Section~\ref{sec-long-range}}, we first define $G_{z,n} := \left\{ D_h(0,B_{\ep_n}(z) ) \in [\BB s , \BB s  +\ep_n^\beta] \right\}$ for $\beta$ slightly smaller than $\xi(Q-\alpha)$. We then set $\mcl G^n(z) := \mcl E^n(z) \cap G_{z,n}$. 
We will extract the desired properties for $\mcl G^n(z)$ from the analogous properties of $\mcl E^n(z)$, as follows. 
Property~\ref{item-E-perfect} for $\mcl G^n(z)$ is immediate from the analogous property of $\mcl E^n(z)$ and the definition of $G_{z,n}$. 

The one-point estimate for $\mcl G^n(z)$ is proven in Section~\ref{sec-end-1pt}. We let $\phi_{z,1}$ be a smooth bump function supported on an appropriate sub-annulus of $\BB A_{\ep_1 , \ep_0/3}(z)$. 
We also let $R > 0$ be large and we sample $X$ from Lebesgue measure on $[0,R]$, independently from $h$. 
Since we know that $D_h(0,B_{\ep_n}(z))$ is of constant order on $\mcl E^n(z)$, if $R$ is chosen to be sufficiently large (independently of the choice of $z$ or $n$) then we can use Weyl scaling (Axiom~\ref{item-metric-f}) to deduce that 
\eqb \label{eqn-outline-1pt-cond}
\BB P\left[\text{$G_{z,n}$ occurs with $h+X\phi_{z,1}$ in place of $h$} \,\big|\, h \right] \BB 1_{\mcl E^n(z)} \asymp \ep_n^\beta \BB 1_{\mcl E^n(z)} . 
\eqe 

By taking unconditional expectations in~\eqref{eqn-outline-1pt-cond} and recalling that $\mcl G^n(z) = G_{z,n} \cap \mcl E^n(z)$, we obtain
\eqb \label{eqn-outline-1pt'}
\BB P\left[\text{$\mcl G^n(z)$ occurs with $h+X\phi_{z,1}$ in place of $h$} \right] \asymp \ep_n^\beta \BB P\left[\mcl E^n(z) \right] .
\eqe 
Actually, there is some subtlety here since $\mcl E^n(z)$ does not imply $\mcl E^n(z)$ with $h+X\phi_{z,1}$ in place of $h$. To get around this we need to vary some of the parameters in the definition of $\mcl E^n(z)$ by an $R$-dependent amount; see Section~\ref{sec-end-1pt} for details. 
The law of $h+X\phi_{z,1}$ is mutually absolutely continuous w.r.t.\ the law of $h$, and the regularity conditions in the definition of $\mcl E^n(z)$ allow us to bound the Radon-Nikodym derivative. From this, we see that~\eqref{eqn-outline-1pt'} together with our one-point estimate for $\mcl E^n(z)$ implies the desired one-point estimate for $\mcl G^n(z)$. 

The two-point estimate for $\mcl G^n(z)$ is proven in Section~\ref{sec-end-2pt}. For this, we follow a similar argument to the case of the one-point estimate. If $|z-w| \in [\ep_{m+1},\ep_m]$, we let $\phi_{z,1}$ be as above and we let $\phi_{w,m+2}$ be an appropriate smooth bump function on a sub-annulus of $\BB A_{\ep_{m+2} , \ep_{m+1}/3}(w)$. 
We then sample $X_z,X_w$ uniformly from $[0,1]$ independently from each other and from $h$ and set $\wt h := h + X_z \phi_{z,1} + X_w\phi_{w,m+2}$. 
The reason why we want $\phi_{w,m+2}$ to be supported on $\BB A_{\ep_{m+2} , \ep_{m+1}/3}(w)$ is that the support of $\phi_{w,m+2}$ does not disconnect $z$ from 0, so changing $X_w$ does not have much of an effect on $D_{\wt h}(0,z)$ (c.f.\ Lemma~\ref{lem-geo-away}). This prevents the joint conditional density of $D_{\wt h}(0,z)$ and $D_{\wt h}(0,w)$ given $h$ from concentrating in a subset of $\BB R^2$ with small Lebesgue measure. 
Via a similar argument to the one above, we then obtain a two-point estimate for $\mcl G^n(z)$ with $\wt h$ in place of $h$ using the two-point estimate for $\mcl E^n(z)$.
We then deduce the desired two-point estimate for $\mcl G^n(z)$ using absolute continuity.

\section{Short-range events}
\label{sec-short-range}

\subsection{Definition of short-range events}
\label{sec-short-def}

Let
\eqb \label{eqn-ep-choice}
\ep_0 := 1/3 \quad \text{and} \quad \ep_j := \frac{1}{100j!} ,\quad\forall j \in \BB N .
\eqe
The reason for a factorial choice of $\ep_j$ is to make it so that the gaps between successive scales, i.e., the ratios $\ep_j/\ep_{j-1}$, tend to $\infty$ as $j\rta\infty$ but still satisfy $\ep_j/\ep_{j-1} = \ep_j^{o_j(1)}$. The first property is important since it allows us to absorb factors of $c^j$, for $c>0$ and $j\in\BB N$, into multiplicative errors of the form $\ep_j^{o_j(1)}$ (see, e.g., the proof of Lemma~\ref{lem-thick-1pt}). The second property is important, e.g., since it allows us to show that the probabilities of our events change by at most a factor of $\ep_j^{o_j(1)}$ when we increase or decrease $j$ by a constant-order amount (see, e.g., the proof of Lemma~\ref{lem-thick-2pt}). 

Fix large constants $A , K , L > 0$ with $K , L \geq A$, which will be the parameters of our events. We will choose $A$ and $L$ in a manner depending only on $\alpha,\gamma$ in Proposition~\ref{prop-E-prob} and we will choose $K$ in a manner depending on $A$ in Lemma~\ref{lem-Elower-prob}. 
The parameter $K$ will replace $A$ in some of the conditions in the definition of the event at scale $j=1$; it does not appear in any of the definitions for scale $j\geq 2$. 

\begin{remark}
Many of the definitions of events in this and subsequent sections involve fixed numerical constants ($1/3,5/6,6$, etc.), especially when specifying the sizes of Euclidean annuli. The particular values of these constants are usually not important for our purposes, just their relative numerical order. 
\end{remark}

\subsubsection{Events at a single scale}
\label{sec-E-def}

We want to define for each $z\in\BB C$ and $j\in\BB N$ an event $E_{z,j} \in \sigma\left( (h-h_{\ep_{j-1}}(z)) |_{\BB A_{\ep_j,\ep_{j-1}/3}(z)}\right)$ and events $H_{z,j}^{\op{out}} , H_{z,j}^{\op{in}} \in \sigma\left( (h-h_{\ep_{j-1}}(z)) |_{\BB A_{\ep_j/3 ,\ep_j}(z)} \right)$. 
Before defining our events, we first need a few preliminary definitions.
For $z\in\BB C$ and $j\in\BB N_0$, let $\frk h_{z,j}$ be the harmonic part of $(h-h_{\ep_j}(z))|_{B_{\ep_j}(z)}$.
 
Also let $\frk h_{z,j}^{\op{out}}$ be the harmonic function on $\BB A_{\ep_j/3,\ep_j}(z)$ whose boundary data on $\bdy B_{\ep_j }(z)$ coincides with that of $h - h_{\ep_j}(z)$ (equivalently, that of $\frk h_{z,j}$) and whose boundary data on $\bdy B_{\ep_j/3}(z)$ is identically equal to zero.
Symmetrically, let $\frk h_{z,j}^{\op{in}}$ be the harmonic function on $\BB A_{\ep_j/3,\ep_j}(z)$ whose boundary data on $\bdy B_{\ep_j/3}(z)$ coincides with that of $h- h_{\ep_j}(z)$ (equivalently, that of $\frk h_{z,j}$) and whose boundary data on $\bdy B_{\ep_j}(z)$ is identically equal to zero. 
Then $  \frk h_{z,j}^{\op{in}} + \frk h_{z,j}^{\op{out}}$ is the harmonic part of $(h-h_{\ep_j}(z))|_{\BB A_{\ep_j/3,\ep_j}(z)}$.
The purpose of the harmonic functions $  \frk h_{z,j}^{\op{in}} $ and $ \frk h_{z,j}^{\op{out}}$ will be discussed further in Section~\ref{sec-H-def}. 

Finally, fix, once and for all, a smooth bump function $\phi : \BB C\rta [0,1]$ with 
\eqb \label{eqn-phi-choice}
\phi|_{\BB A_{2,3}(0)} \equiv 1  \quad \text{and} \quad \phi|_{\BB C\setminus \BB A_{1,4}(0)} \equiv 0 .
\eqe

\begin{figure}[t!]
 \begin{center}
\includegraphics[scale=1]{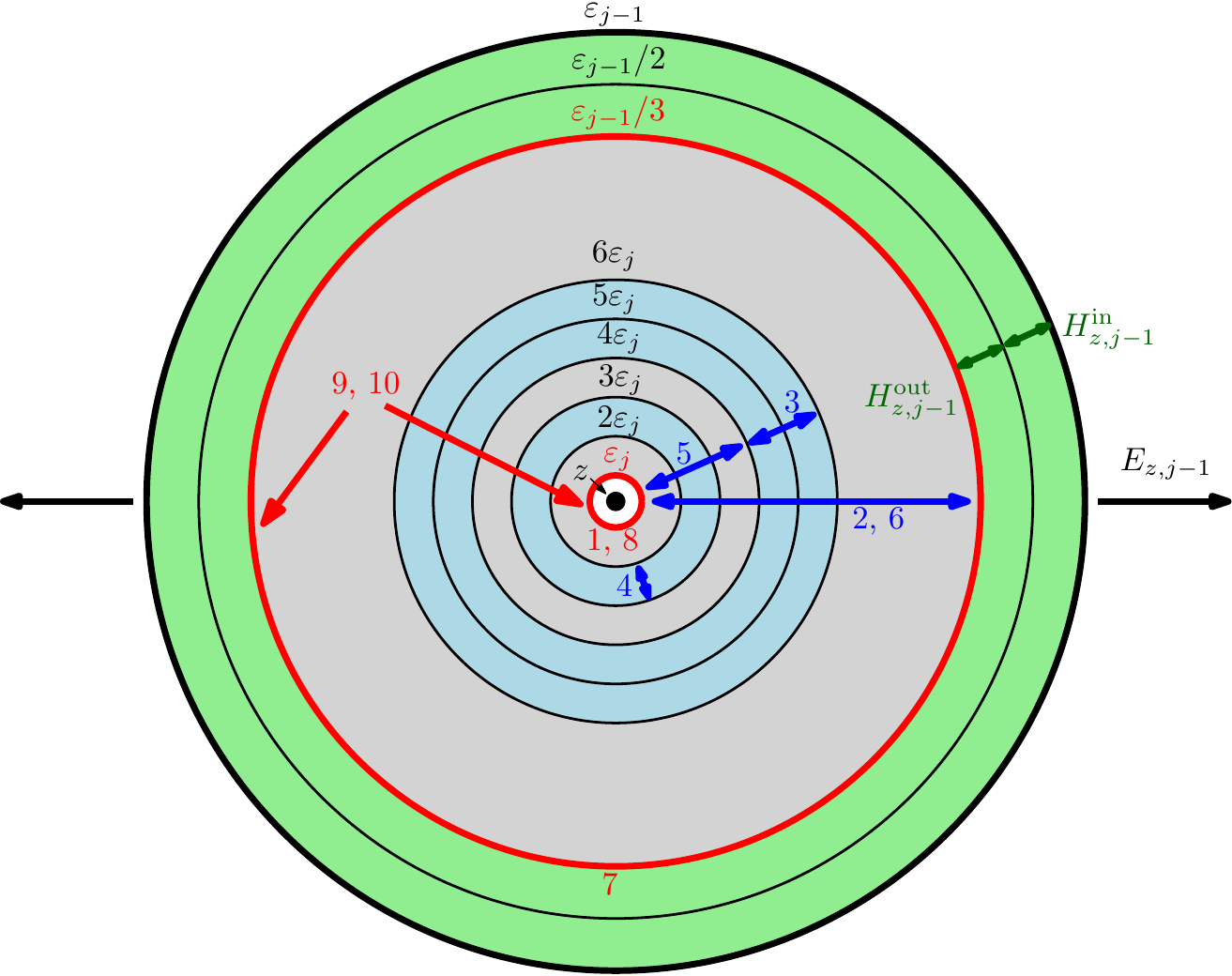}
\vspace{-0.01\textheight}
\caption{Schematic illustration of the event $E_{z,j}$. Ratios of radii of concentric circles are not shown to scale. The part of the field $h-h_{\ep_{j-1}}(z)$ which each condition in the definition of $E_{z,j}$ depends on is indicated with either a red circle (for conditions which only depend on the restriction of $h-h_{\ep_{j-1}}(z)$ to one or more circles) or a double-sided blue arrow (for conditions which depend on the restriction of $h-h_{\ep_{j-1}}(z)$ to an annulus). The same is done for the events $H_{z,j-1}^{\op{out}}$ and $H_{z,j-1}^{\op{in}}$, with green double-sided arrows. The event $E_{z,j}$ is determined by the gray and blue annulus $(h-h_{\ep_{j-1}}(z))|_{\BB A_{\ep_j,\ep_{j-1}/3}(z)}$. The event $E_{z,j-1}$ is determined by the restriction of $h$ to the region outside of the green annulus. Note that $\ep_{j-1} / \ep_j \rta \infty$ as $j\rta\infty$ so the aspect ratio of the outermost grey annulus is huge when $j$ is large. The aspect ratios of the other annuli are of constant order.
}\label{fig-onescale-event}
\end{center}
\vspace{-1em}
\end{figure} 
 
Let us now define the main event we will consider for a single scale. We comment on the purpose of each condition in Remark~\ref{remark-E} below.
For $z\in \BB C$ and $j\in\BB N$, let $E_{z,j} = E_{z,j}(A,K,L)$ be the event that the following is true.
\medskip

\noindent\textbf{$\alpha$-thickness for the circle average process:}

\begin{enumerate}
\item $| h_{\ep_j}(z) - h_{\ep_{j-1} }(z)  - \alpha  \log(\ep_{j-1}/\ep_j)| \leq A [\log(\ep_{j-1}/\ep_j)]^{3/4}$. \label{item-onescale-thick}
\item $\sup_{r\in [\ep_j,\ep_{j-1}/3]} |h_r(z) - h_{\ep_{j-1}}(z) - \alpha\log(\ep_{j-1}/r)| $ is bounded above by $ K [\log(\ep_{j-1}/\ep_j)]^{3/4}$ if $j=1$ or by $A[\log(\ep_{j-1}/\ep_j)]^{3/4}$ if $j\geq 2$. \label{item-onescale-sup}
\end{enumerate}

\noindent\textbf{Geometric conditions:}

\begin{enumerate}
\setcounter{enumi}{2}
\item There is a path in $\BB A_{5\ep_j,6\ep_j} (z)$ which disconnects the inner and outer boundaries of $\BB A_{5\ep_j , 6\ep_j} (z)$ and whose $D_h$-length is at most $\frac{1}{100} D_h\left( \bdy B_{4\ep_j}(z) , \bdy B_{5\ep_j}(z) \right)$. \label{item-onescale-around}
\item $ D_h\left( \bdy  B_{2\ep_j}(z)  , \bdy  B_{3\ep_j}(z) \right) \geq A^{-1} \ep_j^{\xi Q} e^{\xi h_{\ep_j}(z)}$.  \label{item-onescale-annulus} 
\item With $\phi$ as in~\eqref{eqn-phi-choice}, let $\phi_{z,j}(u) := \phi\left(\ep_j^{-1}(u - z)\right)$, so that $\phi_{z,j}$ is supported on $\BB A_{\ep_j , 4\ep_j}(z)  $ and $\phi_{z,j}|_{\BB A_{2\ep_j , 3\ep_j}(z)}\equiv 1$. Then the Dirichlet inner product satisfies $|(h,\phi_{z,j})_\nabla |  \leq K$ if $j=1$ or $|(h,\phi_{z,j})_\nabla |  \leq A$ if $j\geq 2$.  \label{item-onescale-dirichlet}
\item $\sup_{u,v\in \BB A_{ \ep_j,\ep_{j-1}/3}(z)} D_h\left(u,v ; \BB A_{ \ep_j , \ep_{j-1}/3}(z)\right)$ is bounded above by $K \ep_{j-1}^{\xi Q} e^{\xi h_{\ep_{j-1}}(z)}$ if $j=1$ or by $A\ep_{j-1}^{\xi Q} e^{\xi h_{\ep_{j-1}}(z)}$ if $j\geq 2$. \label{item-onescale-diam}
\end{enumerate}

\noindent\textbf{Conditions for controlling the Radon-Nikodym derivative given what happens at a previous scale:}

\begin{enumerate}
\setcounter{enumi}{6}
\item Let $\frk F_{z,j-1}$ the set of harmonic functions $\frk f$ on $B_{\ep_{j-1}}(z) $ which satisfy $\frk f(0) = 0$ and $\sup_{u\in B_{\ep_{j-1}/2}(z) } |\frk f(u)| \leq A$. Let $\rng h$ be a zero-boundary GFF on $B_{\ep_{j-1}}(z)$ . 
For each $\frk f\in \frk F_{z,j-1}$, the Radon-Nikodym derivative of the law of $(\rng h + \frk f)|_{B_{\ep_{j-1}/3}(z)}$ w.r.t.\ the law of $(h - h_{\ep_{j-1 }}(z))|_{B_{\ep_{j-1}/3}(z)}$ is bounded above by $L$ and below by $L^{-1}$.  \label{item-onescale-rn}
\item The harmonic part $\frk h_{z,j}$ of $(h-h_{\ep_j}(z))|_{B_{\ep_j}(z)}$ satisfies \label{item-onescale-harmonic}
\eqb
\sup_{u\in  B_{\ep_{j}/2}(z) } |\frk h_{z,j}(u)   | \leq A , \quad \text{i.e.,} \quad \frk h_{z,j}   \in \frk F_{z,j }. 
\eqe
\end{enumerate}

\noindent\textbf{Conditions for controlling diameters and circle averages in intermediate annuli:}

\begin{enumerate}
\setcounter{enumi}{8}
\item In the notation introduced above,   \label{item-onescale-out}
\eqb
\sup_{u\in \BB A_{\ep_{j }/3,\ep_{j }/2}(z)} |\frk h_{z,j }^{\op{out}}(z)| \leq A 
\quad \text{and} \quad 
\sup_{u\in \BB A_{\ep_{j-1}/2,\ep_{j-1} }(z)} |\frk h_{z,j-1}^{\op{in}}(z)| \leq A .
\eqe
\item With the events $H_{z,j}^{\op{in}}$ and $H_{z,j}^{\op{out}}$ defined in~\eqref{eqn-H-def} just below,    \label{item-onescale-cond}
\eqbn 
\BB P\left[H_{z,j-1}^{\op{out}} \cap H_{z,j }^{\op{in}} \,\big|\, (h-h_{\ep_{j-1}}(z) )|_{\BB A_{\ep_j,\ep_{j-1}/3}(z)} \right] \geq \frac34 .
\eqen
\end{enumerate}

\begin{remark} \label{remark-E}
In this remark we explain the purpose of each condition in the definition of $E_{z,j}$. 
\begin{enumerate}
\item Used in the proof of Lemma~\ref{lem-E-thick} to control the circle average of $h$ on $\mcl E^n(z)$. This lemma will eventually be used to make sure that the Frostman measure constructed in Section~\ref{sec-dim-lower} is supported on $\mcl T_h^\alpha$. 
\item Also used in Lemma~\ref{lem-E-thick} to deal with $h_r(z)$ for $r\notin\{\ep_j\}_{j\in\BB N_0}$. 
The reason why we separate this condition from condition~\ref{item-onescale-thick} and use $K$ instead of $A$ when $j=1$ is as follows.
In Section~\ref{sec-end-1pt} below, we want say that if $x > 0$ is large and $E_{z,1}$ occurs with $K = A$, then $E_{z,1}$ occurs with $h + x\phi_{z,1}$ in place of $h$ for a sufficiently large value of $K$ (depending on $x$). Note that adding a multiple of $\phi_{z,j}$ to $h$ does not affect $h_{\ep_j}(z) - h_{\ep_{j-1}}(z)$ since $\phi_{z,j}$ vanishes on $\bdy B_{\ep_j}(z) \cup \bdy B_{\ep_{j-1}}(z)$, which is why we do not need $K$ in condition~\ref{item-onescale-thick}. The reason why we want to consider $h+x\phi_{z,1}$ is for an absolute continuity argument similar to the ones in Section~\ref{sec-one-pt}, as discussed in Section~\ref{sec-outline}. 
\item Used in Lemma~\ref{lem-geo-away} to prevent $D_h$-geodesics between points outside of $B_{6\ep_j}(z)$ from entering $B_{4\ep_j}(z)$.
\item Used in Lemma~\ref{lem-E-dist} to lower-bound $D_h$-distances and in Lemma~\ref{lem-dist-interval} to lower-bound the effect on $D_h(0,z)$ when we add a multiple of $\phi_{z,j}$. 
\item Used in Lemmas~\ref{lem-dist-1pt-lower} and~\ref{lem-full-2pt} to control the Radon-Nikodym derivative between the laws of $h$ and $h+x\phi_{z,1}$. The purpose of $K$ is the same as for condition~\ref{item-onescale-sup}. 
\item Used in Lemma~\ref{lem-E-dist} to upper-bound $D_h$-distances, which will be important for the one-point lower bound for the long-range event in Section~\ref{sec-end-1pt} and for making sure that the Frostman measure is supported on $\wt{\mcl T}_h^\alpha$. The purpose of $K$ is the same as for condition~\ref{item-onescale-sup}. 
\item Used in Lemma~\ref{lem-multiscale-cond}, with $\frk f = \frk h_{z,j-1}$, to control the conditional probability of $E_{z,j}$ given $h|_{\BB C\setminus B_{\ep_{j-1}}(z)}$. This is the only appearance of the parameter $L$; the reason why we introduce $L$ is that the definition of $\frk F_{z,j-1}$ depends on $A$, so $L$ needs to be chosen in a manner depending on $A$. 
\item Used in Lemma~\ref{lem-multiscale-cond} to allow us to apply condition~\ref{item-onescale-rn} at the next scale. 
\item Used in Lemma~\ref{lem-E-dist} to compare the fields $h_{z,j}^{\op{out}}$ and $h_{z,j}^{\op{in}}$ defined in~\eqref{eqn-in-out-fields} to $h|_{\BB A_{\ep_j/3,\ep_j}}$. 
\item Used in Lemma~\ref{lem-multiscale-cond} to lower-bound the conditional probability given $E_{z,j}$ of events defined as intersections of the $E_{z,j}$'s, $H_{z,j}^{\op{in}}$'s, and $H_{z,j}^{\op{out}}$'s. 
\end{enumerate}
\end{remark}

\begin{lem} \label{lem-E-msrble}
The event $E_{z,j}$ is measurable w.r.t.\ the $\sigma$-algebra generated by $(h-h_{\ep_{j-1}}(z))|_{\BB A_{\ep_j,\ep_{j-1}/3}(z)}$. 
\end{lem}
\begin{proof} 
By inspection, all of the conditions in the definition of $E_{z,j}$ are determined by $(h-h_{\ep_{j-1}}(z))|_{\BB A_{\ep_j,\ep_{j-1}/3}(z)}$ except possibly condition~\ref{item-onescale-rn}. For the conditions involving $D_h$-distances, we use the locality of the metric (Axiom~\ref{item-metric-f}). 
By Lemma~\ref{lem-rn-msrble} just below (applied with to the field $h(\ep_{j-1}\cdot + z) - h_{\ep_{j-1}}(z) \eqD h$ and the domain $U = B_{1/3}(0)$), the event of condition~\ref{item-onescale-rn} in the definition of $E_{z,j}$ is determined by $(h-h_{\ep_{j-1}}(z))|_{\bdy B_{\ep_{j-1}/3}(z)}$, which is determined by $(h-h_{\ep_{j-1}}(z))|_{\BB A_{\ep_j,\ep_{j-1}/3}(z)}$.
\end{proof}

The following lemma was used in the proof of Lemma~\ref{lem-E-msrble}.

\begin{lem} \label{lem-rn-msrble}
Let $U\subset\BB D$ be a Jordan domain and let $\frk f$ be a deterministic harmonic function on $\BB D$. 
Also let $\rng h$ be a zero-boundary GFF on $\BB D$ and let $h$ be a whole-plane GFF normalized so that $h_1(0) = 0$. 
The Radon-Nikodym derivative of the law of $(\rng h + \frk f)|_U$ w.r.t.\ the law of $h|_U$ is a.s.\ determined by $h|_{\bdy U}$. 
\end{lem}
\begin{proof} 
For $\ep > 0$, let $\bdy^\ep U := B_\ep(\bdy U) \cap U$.  
By the Markov property of the whole-plane GFF, the conditional law of $h|_U$ given $h|_{\bdy^\ep U}$ is that of a zero-boundary GFF on $U \setminus \bdy^\ep U$ plus the harmonic extension of the field values of $ h $ on $\bdy^\ep U$ to $U \setminus \bdy^\ep U$. 
By the Markov property of the zero-boundary GFF and since $\frk f$ is harmonic, the same is true with $\rng h + \frk f$ in place of $h  $.
Therefore, the Radon-Nikodym derivative of the law of $(\rng h + \frk f)|_{U}$ w.r.t.\ the law of $h|_U$ is the same as the Radon-Nikodym derivative of the law of $(\rng h + \frk f)|_{\bdy^\ep U}$ w.r.t.\ the law of $h|_{\bdy^\ep U}$.
This Radon-Nikodym derivative is a measurable function of $h|_{\bdy^\ep U}$. 
Sending $\ep \rta 0$, we get that the Radon-Nikodym derivative is a.s.\ determined by $h|_{\bdy U}$. 
\end{proof}

\subsubsection{Events between scales}
\label{sec-H-def}
 
We now define the events appearing in condition~\ref{item-onescale-cond} in the definition of $E_{z,j}$. 
Using the harmonic functions $\frk h_{z,j}^{\op{out}}$ and $\frk h_{z,j}^{\op{in}}$ introduced at the beginning of Section~\ref{sec-E-def}, we define
\eqb \label{eqn-in-out-fields}
 h_{z,j}^{\op{out}} := \left(h - h_{\ep_j}(z) - \frk h_{z,j}^{\op{out}} \right) |_{\BB A_{\ep_j/3 , \ep_j}(z)} 
 \quad \text{and} \quad
 h_{z,j}^{\op{in}} := \left( h - h_{\ep_j}(z) - \frk h_{z,j }^{\op{in}} \right) |_{\BB A_{\ep_j/3 , \ep_j }(z)} .
\eqe
Note that
\eqbn
\left(h - h_{\ep_j}(z)  \right) |_{\BB A_{\ep_j/3 , \ep_j}(z)}  = h_{z,j}^{\op{out}} +   \frk h_{z,j}^{\op{out}} = h_{z,j}^{\op{in}} +   \frk h_{z,j}^{\op{in}} .
\eqen
The reason why these two different decompositions of $\left(h - h_{\ep_j}(z)  \right) |_{\BB A_{\ep_j/3 , \ep_j}(z)} $ are useful is as follows. The harmonic function $\frk h_{z,j}^{\op{out}}$ is by definition determined by the ``outside" part of the field $h-h_{\ep_j}(z)$ (i.e., by $(h-h_{\ep_j}(z))|_{\bdy B_{\ep_j}(z)}$) and, as shown in Lemma~\ref{lem-in-out-ind} just below, the field $h_{z,j}^{\op{out}}$ is conditionally independent from the outside part of the field given the ``inside" part of the field (i.e., given $(h- h_{\ep_j}(z))|_{\bdy B_{\ep_j/3}(z)}$). 
The opposite is true for $\frk h_{z,j}^{\op{in}}$ and $h_{z,j}^{\op{in}}$. 
This makes it easy to estimate the conditional probabilities of events defined in terms of $h_{z,j}^{\op{out}} ,   \frk h_{z,j}^{\op{out}},  h_{z,j}^{\op{in}} ,   \frk h_{z,j}^{\op{in}}$ (e.g., the events $H_{z,j}^{\op{in}}$ and $H_{z,j}^{\op{out}}$ which we define just below) when we condition on events which depend  on $h|_{\BB A_{\ep_k,\ep_{k-1}/3}(z)}$ for $k\in \BB N$ (e.g., the events $E_{z,k}$); see also Remark~\ref{remark-H}. 

\begin{lem} \label{lem-in-out-ind}
The field $h_{z,j}^{\op{out}}$ is conditionally independent from $(h- h_{\ep_j}(z))|_{\BB C\setminus \BB A_{\ep_j/3,\ep_j}(z)}$ given $(h- h_{\ep_j}(z))|_{\bdy B_{\ep_j/3}(z)}$.
Furthermore, $h_{z,j}^{\op{in}}$ is conditionally independent from $(h- h_{\ep_j}(z))|_{\BB C\setminus \BB A_{\ep_j/3, \ep_j}(z)}$ given $(h- h_{\ep_j}(z))|_{\bdy B_{\ep_j }(z)}$.
\end{lem}
\begin{proof}
By the definition of $\frk h_{z,j}^{\op{out}}$ and  $\frk h_{z,j}^{\op{in}}$, $h_{z,j}^{\op{out}} $ is the sum of the zero-boundary part of $(h- h_{\ep_j}(z))|_{\BB A_{\ep_j/3,\ep_j}(z)}$ (which is independent from $(h- h_{\ep_j}(z))|_{\BB C\setminus \BB A_{\ep_j/3,\ep_j}(z)}$) and the function $\frk h_{z,j}^{\op{in}}$ (which is determined by $(h- h_{\ep_j}(z))|_{  \bdy B_{\ep_j/3}(z)}$). This gives the desired statement for $h_{z,j}^{\op{out}}$. The statement for $h_{z,j}^{\op{in}}$ is obtained similarly.
\end{proof}

We define
\allb \label{eqn-H-def}
H_{z,j}^{\op{out}} &:= \left\{ \sup_{u,v\in \BB A_{\ep_j/3,\ep_j/2}(z)} D_{h_{z,j}^{\op{out}}}\left( u , v ; \BB A_{\ep_j/3 , \ep_j/2}(z) \right) \leq A \ep_j^{\xi Q}  \right\} \notag\\ 
&\qquad \qquad \cap \left\{ \sup_{r \in [\ep_j/3,\ep_j/2]} |(h_{z,j}^{\op{out}})_r(z) | \leq A \right\}\quad \text{and} \quad \notag \\
H_{z,j}^{\op{in}} &:= \left\{ \sup_{u,v\in \BB A_{\ep_j/2, \ep_j}(z)} D_{h_{z,j}^{\op{in}}}\left( u , v ; \BB A_{\ep_j/2 , \ep_j}(z) \right) \leq A \ep_j^{\xi Q}  \right\} \notag \\
&\qquad \qquad \cap \left\{ \sup_{r \in [\ep_j/2,\ep_j ]} |(h_{z,j}^{\op{in}})_r(z)| \leq A \right\}  ,
\alle
where $(h_{z,j}^{\op{out}})_r(z)$ and $(h_{z,j}^{\op{in}})_r(z)$ denote circle averages.

\begin{lem} \label{lem-H-msrble}
The events $H_{z,j}^{\op{out}}$ and $H_{z,j}^{\op{in}}$ are each a.s.\ determined by $h|_{\BB A_{\ep_j/3,\ep_j }(z)}$, viewed modulo additive constant. 
\end{lem}
\begin{proof}
We will prove the statement for $H_{z,j}^{\op{out}}$; the statement for $H_{z,j}^{\op{in}}$ is proven in an identical manner.
It is obvious that the circle average condition in the definition of $H_{z,j}^{\op{out}}$ is determined by $h|_{\BB A_{\ep_j/3,\ep_j }(z)}$, viewed modulo additive constant, so we only need to consider the diameter condition. 
By the locality of the metric (Axiom~\ref{item-metric-f}), $D_{h_{z,j}^{\op{out}}}\left( u , v ; \BB A_{\ep_j/3 , \ep_j/2}(z) \right) $ is a.s.\ determined by $h_{z,j}^{\op{out}}|_{\BB A_{\ep_j/3 , \ep_j }(z)}$. 
As in the proof of Lemma~\ref{lem-in-out-ind}, the field $h_{z,j}^{\op{out}}|_{\BB A_{\ep_j/3 , \ep_j }(z)}$ is the sum of the zero-boundary part of $(h- h_{\ep_j}(z))|_{\BB A_{\ep_j/3,\ep_j}(z)}$ (which is determined by $h|_{\BB A_{\ep_j/3,\ep_j }(z)}$, viewed modulo additive constant) and the function $\frk h_{z,j}^{\op{in}}$ (which is determined by   $(h- h_{\ep_j}(z))|_{\BB A_{\ep_j/3,\ep_j}(z)}$, and hence by $h|_{\BB A_{\ep_j/3,\ep_j }(z)}$, viewed modulo additive constant). 
\end{proof}

\begin{remark} \label{remark-H}
The purpose of the events $H_{z,j}^{\op{in}}$ and $H_{z,j}^{\op{out}}$ is as follows. We need a positive spacing between the scales at which the events $E_{z,j}$ are defined, i.e., we need the event $E_{z,j}$ to depend only on $h|_{\BB A_{\ep_j,\ep_{j-1}/3}(z)}$ instead of on $h|_{\BB A_{\ep_j,\ep_{j-1}}(z)}$. The reason for this is that we can control the Radon-Nikodym derivative of the conditional law of $h|_{B_{\ep_{j-1}/3}(z)}$ given $h|_{\BB C\setminus B_{\ep_{j-1}}(z)}$ w.r.t.\ the marginal law of $h|_{B_{\ep_{j-1}/3}(z)}$, but we cannot do the same with $h|_{B_{\ep_{j-1} }(z)}$ in place of $h|_{B_{\ep_{j-1}/3}(z)}$ (c.f.\ condition~\ref{item-onescale-rn}). 
However, in Lemma~\ref{lem-E-thick} we want to control $h_r(z)$ for all $r\in [\ep_j,\ep_{j-1} ]$ on $E_{z,j}$ (so that we can eventually say that the perfect points are $\alpha$-thick). Moreover, in Lemma~\ref{lem-E-dist} we want to prove an upper bound for $D_h(0,B_{\ep_n}(z))$ by summing the $D_h$-diameters of the annuli $\BB A_{\ep_j,\ep_{j-1}}(z)$ for $j\in [1,n]_{\BB Z}$. So, we need some information about what happens in $\BB A_{\ep_{j-1}/3,\ep_{j-1}}(z)$, which is provided by the events $H_{z,j-1}^{\op{in}}$ and $H_{z,j-1}^{\op{out}}$.
\end{remark}

\subsubsection{Events for multiple scales}

In addition to $E_{z,j}$, we will also need an event at ``scale 0" which includes only a subset of the conditions in the definition of $E_{z,j}$ for $j\in\BB N$.
In particular, we let $E_{z,0}$ be the event that the following is true (note that the numbering is consistent with the definition of $E_{z,j}$). 
\begin{enumerate}
\item[1.] $|h_{\ep_0}(z)| \leq A$. \label{item-scale0-thick}
\item[7.] Let $\frk h_{z,0}$ be the harmonic part of $(h-h_{\ep_0}(z))|_{B_{\ep_0}(z)}$. Then\label{item-scale0-harmonic}
\eqb
\sup_{u\in  B_{\ep_0/2}(z) } |\frk h_{z,0}(u)  | \leq A , \quad \text{i.e.,} \quad \frk h_{z,0}(u) \in \frk F_{z,0 }. 
\eqe
Here we recall that $\frk F_{z,0}$ was defined in condition~\ref{item-onescale-rn} in the definition of $E_{z,1}$. 
\item[9.] $\sup_{u\in \BB A_{\ep_{0 }/3,\ep_{0}/2}(z)} |\frk h_{z,0 }^{\op{out}}(z)| \leq A $.   \label{item-scale0-out} 
\item[10.] In the notation~\eqref{eqn-H-def},  \label{item-scale0-cond}
\eqbn 
\BB P\left[  H_{z,0 }^{\op{in}} \,\big|\, h|_{\bdy B_{\ep_0}(z) } \right] \geq \frac34 .
\eqen
\end{enumerate}
By inspection, 
\eqb \label{eqn-scale0-msrble}
E_{z,0} \in \sigma\left( h|_{\bdy B_{\ep_0}(z)} \right) . 
\eqe

For $m,n\in\BB N $ with $m \leq n$, let
\eqb \label{eqn-Ecap-def}
\mcl E_m^n(z) := \bigcap_{j=m }^n \left(  E_{z,j} \cap   H_{z,j-1 }^{\op{out}} \cap H_{z,j-1 }^{\op{in}}  \right) 
\quad \text{and} \quad 
\mcl E^n(z) := E_{z,0} \cap \mcl E_1^n(z) .
\eqe 
If $m> n$, we define $\mcl E_n^m(z) $ to be the whole probability space (the empty intersection). 
We also set $\mcl E^0(z) := E_{z,0}$. 

By Lemmas~\ref{lem-E-msrble} and~\ref{lem-H-msrble}, 
\allb \label{eqn-Ecap-msrble}
\mcl E_m^n(z) &\in \sigma\left( (h-h_{\ep_{m-1}}(z))|_{\BB A_{\ep_n,\ep_{m-1} }(z)} \right)   ,\quad\forall m,n\in\BB N
\quad \text{and} \quad \notag\\
\mcl E^n(z) &\in \sigma\left(  h|_{\BB A_{\ep_n,\ep_0}(z)}  \right)  ,\quad\forall n\in\BB N  .
\alle

\subsubsection{One-point estimate at a single scale}

The main quantitative estimate we need for the events $E_{z,j}$ is the following proposition, whose proof is given in Section~\ref{sec-onescale-prob}.

\begin{prop} \label{prop-E-prob}
There exists $A , L > 0$ depending only on $\alpha,\gamma$ such that for each parameter $K\geq A$, each $z\in\BB C$, and each $j\in\BB N$, 
\eqb \label{eqn-E-prob}
\BB P\left[ E_{z,j}\cap H_{z,j-1}^{\op{out}} \cap H_{z,j-1}^{\op{in}} \right] \geq (\ep_j / \ep_{j-1})^{\alpha^2/2 + o_j(1)} 
\eqe
with the rate of the $o_j(1)$ depending only on $A,L,K,\alpha,\gamma$.  
Moreover, for each $z\in\BB D$ we have $\BB P[E_{z,0} ] \succeq 1$, with the implicit constant depending only on $A,L,K,\alpha,\gamma$.  
\end{prop}

Throughout this section and the next we fix $A$ and $L$ as in Proposition~\ref{prop-E-prob}, so that $A$ and $L$ depend only on $\alpha,\gamma$. 
We also let $K\geq A$ to be chosen in Lemma~\ref{lem-Elower-prob} below, in a manner depending on $A$. Most implicit constants, errors, etc., in this section are allowed to depend on $K$.

Basically, the idea of the proof of Proposition~\ref{prop-E-prob} is that the probability of condition~\ref{item-onescale-thick} in the definition of $E_{z,j}$ is $(\ep_j / \ep_{j-1})^{\alpha^2/2 + o_j(1)} $ (by a basic Brownian motion estimate) and, conditional on this condition occurring, the rest of the conditions in the definition of $E_{z,j} \cap H_{z,j-1}^{\op{out}} \cap H_{z,j-1}^{\op{in}}$ occur with uniformly positive conditional probability if $A$ and $L$ are chosen to be sufficiently large. 
However, there is some subtlety involved since it is not immediately obvious that each of the conditions in the definition of $E_{z,j}$ occurs with positive probability for \emph{any} choices of $A$ and $L$. For example, condition~\ref{item-onescale-diam} involves internal diameters taken all the way up to the boundary of an annulus and condition~\ref{item-onescale-rn} involves a simultaneous bound for infinitely many Radon-Nikodym derivatives. 
As such, we will need to prove some estimates in order to establish Proposition~\ref{prop-E-prob}, which is why we defer the proof until Section~\ref{sec-onescale-prob}.

\subsection{One-point and two-point estimates for short-range events} 
\label{sec-short-twopoint}

The goal of this section is to prove the following one-point and two-point estimates for the events $\mcl E^n(z)$ of~\eqref{eqn-Ecap-def}.

\begin{lem}[One-point estimate for $\mcl E^n(z)$] \label{lem-thick-1pt}
For each $z\in B_{1/3}(0)$ and each $n\in\BB N$,
\eqb \label{eqn-thick-1pt}
\BB P\left[ \mcl E^n(z) \right] =  \ep_n^{\alpha^2/2 + o_n(1)}  ,
\eqe
where the rate of convergence of the $o_n(1)$ depends only on $K , \alpha,\gamma$. 
\end{lem}

\begin{lem}[Two-point estimate for $\mcl E^n(z)$] \label{lem-thick-2pt}
Suppose $z,w\in B_{1/3}(0)$, $m\in\BB N$ such that $|z-w| \in [\ep_{m+1} , \ep_m]$, and $n\in\BB N$ with $n\geq m$. Then
\allb \label{eqn-thick-2pt}
\BB P\left[\mcl E^n(z) \cap \mcl E^n(w) \right] 
&\leq \BB P\left[\mcl E_{m+3}^n(z) \cap \mcl E_3^{m-1}(z) \cap \mcl E_{m+3}^n(w)  \cap \mcl E_3^{m-1}(w) \right] \notag \\
&\leq \ep_m^{-\alpha^2/2 + o_m(1)}  \BB P\left[ \mcl E^n(z) \right] \BB P\left[ \mcl E^n(w) \right]  
\alle
where the rate of convergence of the $o_{m}(1)$ depend only on $K, \alpha,\gamma$.
\end{lem}

In the setting of Lemma~\ref{lem-thick-2pt}, we note that by definition $\mcl E_3^{m-1}(z)$ is the whole probability space if $m\leq 3$ and $\mcl E_{m+3}^n(z)$ is the whole probability space if $n\leq m+2$. The estimate~\eqref{eqn-thick-2pt} is still true in these degenerate cases. 

We will in fact only use the second inequality in~\eqref{eqn-thick-2pt} of Lemma~\ref{lem-thick-2pt}.
The reason why we need an upper bound for $\BB P[\mcl E_{m+3}^n(z) \cap \mcl E_3^{m-1}(z) \cap \mcl E_{m+3}^n(w)  \cap \mcl E_3^{m-1}(w)]$ instead of just for $\BB P[\mcl E^n(z) \cap\mcl E^n(w)]$ is that we will need to skip some scales in Section~\ref{sec-end-2pt}; see in particular Lemma~\ref{lem-Eupper-2pt}.

The key tool in the proof of Lemmas~\ref{lem-thick-1pt} and~\ref{lem-thick-2pt} is the following approximate independence across scales result, which is a consequence of~\eqref{eqn-Ecap-msrble} (measurability for $\mcl E_m^n(z)$) and conditions~\ref{item-onescale-rn} and~\ref{item-onescale-harmonic} in the definition of $E_{z,j}$.

\begin{lem} \label{lem-multiscale-cond}
For each $n,m\in\BB N_0$ with $n\geq m+1$ and each $z\in B_{1/3}(0)$, 
\eqb \label{eqn-multiscale-cond-upper}
\BB P\left[ \mcl E_{m+1}^n(z) \,\big|\, h|_{\BB C\setminus B_{\ep_m}(z) }   \right] \BB 1_{E_{z,m}}
\preceq  \BB P\left[ \mcl E_{m+2}^n(z)   \right] \BB 1_{E_{z,m}}
\eqe
and 
\eqb \label{eqn-multiscale-cond-lower}
\BB P\left[ \mcl E_{m+1}^n(z) \,\big|\, h|_{\BB C\setminus B_{\ep_{m-1}}(z) }   \right]  \BB 1_{E_{z,m}}
\succeq  \BB P\left[ \mcl E_{m+1}^n(z)   \right]  \BB 1_{E_{z,m}} ,
\eqe
with the implicit constants depending only on $ K, \alpha,\gamma$. 
\end{lem}
\begin{proof}
Let us first note that $h_{\ep_m}(z)$ is determined by each of $h|_{\BB C\setminus B_{\ep_m}(z)}$ and $(h-h_{\ep_m}(z))|_{\BB C\setminus B_{\ep_m}(z)}$ ($h_{\ep_m}(z)$ is equal to $-1$ times the average of the latter field over $\bdy \BB D$). Therefore,
\eqb
\sigma\left( h|_{\BB C\setminus B_{\ep_m}(z)}  \right) = \sigma\left( (h-h_{\ep_m}(z))|_{\BB C\setminus B_{\ep_m}(z)}  \right) .
\eqe
We will use this fact without comment throughout the rest of the proof. 

Throughout the rest of the proof, we condition on $h|_{\BB C\setminus B_{\ep_m}(z)}$ (equivalently, on $ (h-h_{\ep_m}(z))|_{\BB C\setminus B_{\ep_m}(z)}$) and we assume that $E_{z,m}$ occurs. 
By the Markov property of the GFF, the conditional law of $(h-h_{\ep_m}(z))|_{B_{\ep_m}(z)}$ is that of a zero-boundary GFF on $B_{\ep_m}(z)$ plus the harmonic function $\frk h_{z,m }  $. 
By condition~\ref{item-onescale-harmonic} in the definition of $E_{z,m}$, we have $\frk h_{z,m }    \in \frk F_{z,m }$, with $\frk F_{z,m}$ as in condition~\ref{item-onescale-rn} in the definition of $E_{z,m+1}$. 
By condition~\ref{item-onescale-rn} in the definition of $  E_{z,m+1}$, it therefore follows that on $E_{z,m+1}$, the Radon-Nikodym derivative of the conditional law of $(h-h_{\ep_m}(z))|_{B_{\ep_m/3}(z)}$ given $h|_{\BB C\setminus B_{\ep_{m-1}}(z)}$ w.r.t.\ the marginal law of $(h-h_{\ep_m}(z))|_{B_{\ep_m/3}(z)}$ is bounded above by $L$ and below by $L^{-1}$. 

By Lemma~\ref{lem-E-msrble} and~\eqref{eqn-Ecap-msrble}, we have $\mcl E_{m+2}^n(z) \cap E_{z,m+1} \in \sigma\left( (h-h_{\ep_m}(z))|_{B_{\ep_m/3}(z)} \right)$. 
By this and the conclusion of the preceding paragraph,  
\eqb \label{eqn-multiscale-cond0}
\BB P\left[ \mcl E_{m+2}^n(z) \cap E_{z,m+1}  \,|\,  h|_{\BB C\setminus B_{\ep_m}(z) }   \right] \BB 1_{E_{z,m}}
\asymp  \BB P\left[ \mcl E_{m+2}^n(z) \cap E_{z,m+1}   \right] \BB 1_{E_{z,m}} .
\eqe
Since $ \mcl E_{m+1}^n(z) \subset  \mcl E_{m+2}^n(z) \cap E_{z,m+1} \subset  \mcl E_{m+2}^n(z)  $, we obtain~\eqref{eqn-multiscale-cond-upper} from~\eqref{eqn-multiscale-cond0}.

Recall from~\eqref{eqn-Ecap-def} that $\mcl E_{m+1}^n(z)  = \mcl E_{m+2}^n(z) \cap E_{z,m+1} \cap  H_{z,m }^{\op{out}} \cap H_{z,m }^{\op{in}}$.
To obtain~\eqref{eqn-multiscale-cond-lower}, it therefore remains to deal with the events $H_{z,m }^{\op{out}}$ and $H_{z,m }^{\op{in}}$.  

We first recall that the event $H_{z,m }^{\op{out}}$ is determined by the field $h_{z,m }^{\op{out}}$, which is conditionally independent from $(h-h_{\ep_m}(z))|_{\BB C\setminus B_{\ep_{m }}(z)}  $ given $(h-h_{\ep_m}(z))|_{B_{\ep_m/3}(z)}$ (Lemma~\ref{lem-in-out-ind}). By condition~\ref{item-onescale-cond} in the definition of $E_{z,m+1}$,  
\eqbn
\BB P\left[H_{z,m }^{\op{out}}  \,\big|\, (h-h_{\ep_m}(z))|_{B_{\ep_m/3}(z)} \right] \BB 1_{ E_{z,m+1} }  \geq \frac34 \BB 1_{  E_{z,m+1} }  ,
\eqen
so by the above conditional independence, 
\eqb \label{eqn-use-Hcond0}
\BB P\left[H_{z,m }^{\op{out}}  \,\big|\, (h-h_{\ep_m}(z))|_{\BB C\setminus \BB A_{\ep_m/3,\ep_m}(z)} \right] \BB 1_{E_{z,m+1}} \geq \frac34 \BB 1_{E_{z,m+1}}  .
\eqe
Recall from Lemma~\ref{lem-E-msrble} and~\eqref{eqn-Ecap-msrble} that $\mcl E_{m+2}^n(z) \cap E_{z,m+1} \cap  E_{z,m} \in \sigma\left( (h-h_{\ep_m}(z))|_{\BB C\setminus \BB A_{\ep_m/3,\ep_m}(z)}  \right)$.
By~\eqref{eqn-use-Hcond0}, 
\eqb \label{eqn-use-Hcond-out}
\BB P\left[H_{z,m }^{\op{out}}  \,\big|\, (h-h_{\ep_m}(z))|_{\BB C\setminus \BB A_{\ep_m/3,\ep_m}(z)} \right] \BB 1_{\mcl E_{m+2}^n(z) \cap E_{z,m+1} \cap  E_{z,m}} \geq \frac34  \BB 1_{\mcl E_{m+2}^n(z) \cap E_{z,m+1} \cap  E_{z,m} }   .
\eqe 
Using condition~\ref{item-onescale-cond} in the definition of $E_{z,m}$, we similarly obtain
\eqb \label{eqn-use-Hcond-in}
\BB P\left[H_{z,m }^{\op{in}}  \,\big|\,  (h-h_{\ep_m}(z))|_{\BB C\setminus \BB A_{\ep_m/3,\ep_m}(z)} \right] \BB 1_{\mcl E_{m+2}^n(z) \cap E_{z,m+1} \cap  E_{z,m}} \geq \frac34 \BB 1_{\mcl E_{m+2}^n(z) \cap E_{z,m+1} \cap  E_{z,m}}  .
\eqe
By~\eqref{eqn-use-Hcond-out} and~\eqref{eqn-use-Hcond-in}, 
\eqb \label{eqn-H-cond-pos}
\BB P\left[H_{z,m }^{\op{out}} \cap H_{z,m }^{\op{in}}  \,\big|\,  (h-h_{\ep_m}(z))|_{\BB C\setminus \BB A_{\ep_m/3,\ep_m}(z)} \right]  \BB 1_{\mcl E_{m+2}^n(z) \cap E_{z,m+1} \cap  E_{z,m}} \geq \frac14  \BB 1_{\mcl E_{m+2}^n(z) \cap E_{z,m+1} \cap  E_{z,m}} .
\eqe

Since $\mcl E_{m+2}^n(z) \cap E_{z,m+1} \cap  E_{z,m} \in \sigma\left( (h-h_{\ep_m}(z))|_{\BB C\setminus \BB A_{\ep_m/3,\ep_m}(z)}  \right)$, we can take the conditional expectation of both sides of~\eqref{eqn-H-cond-pos} given $h|_{\BB C\setminus B_{\ep_m}(z)}$ to obtain
\eqb \label{eqn-H-cap-compare}
\BB P\left[\mcl E_{m+1}^n(z)  \,\big|\, h|_{\BB C\setminus B_{\ep_m}(z)} \right] \BB 1_{E_{z,m}} 
\geq  \frac14 \BB P\left[\mcl E_{m+2}^n(z) \cap E_{z,m+1}  \,\big|\, h|_{\BB C\setminus B_{\ep_m}(z)} \right] \BB 1_{E_{z,m}} .
\eqe
Here we recall that $\mcl E_{m+1}^n(z)   =  \mcl E_{m+2}^n(z) \cap E_{z,m+1} \cap H_{z,m }^{\op{out}} \cap H_{z,m}^{\op{in}}$ by definition.
By combining~\eqref{eqn-H-cap-compare} and~\eqref{eqn-multiscale-cond0} we obtain 
\eqb
\BB P\left[\mcl E_{m+1}^n(z)  \,\big|\, h|_{\BB C\setminus B_{\ep_m}(z)} \right] \BB 1_{E_{z,m}}
\succeq  \BB P\left[ \mcl E_{m+2}^n(z) \cap E_{z,m+1}   \right] \BB 1_{E_{z,m}}
\succeq  \BB P\left[ \mcl E_{m+1}^n(z)   \right] \BB 1_{E_{z,m}} ,
\eqe
which is~\eqref{eqn-multiscale-cond-lower}. 
\end{proof}

\begin{lem} \label{lem-multiscale-split}
For each $n,m\in\BB N_0$ with $n\geq m+1$ and each $z\in B_{1/3}(0)$, 
\eqb \label{eqn-multiscale-split-upper}
\BB P\left[ \mcl E^n(z) \right] \preceq \BB P\left[ \mcl E_{m+2}^n (z)  \right] \BB P\left[ \mcl E^m(z) \right]  
\eqe
and
\eqb \label{eqn-multiscale-split-lower}
\BB P\left[ \mcl E^n(z) \right] \succeq \BB P\left[ \mcl E_{m+1}^n (z)  \right] \BB P\left[ \mcl E^m(z) \right]  
\eqe
with the implicit constant depending only on $K, \alpha,\gamma$. 
\end{lem}
\begin{proof}
Since $\mcl E^n(z) = \mcl E_{m+1}^n(z) \cap \mcl E^m(z)$, 
\eqb  \label{eqn-multiscale-split-ratio}
\frac{\BB P\left[ \mcl E^n(z) \right] }{ \BB P\left[\mcl E^m(z) \right]}  = \BB P\left[ \mcl E_{m+1}^n(z) \,|\,  \mcl E^m(z)  \right] .
\eqe  
By definition, we have $E_{z,m} \supset\mcl E^m(z)$ and by~\eqref{eqn-Ecap-msrble}, we have $\mcl E^m(z) \in \sigma\left( h|_{\BB C\setminus B_{\ep_{m}}(z) } \right)$. 
Therefore, the lemma follows immediately from Lemma~\ref{lem-multiscale-cond}. 
\end{proof}

\begin{proof}[Proof of Lemma~\ref{lem-thick-1pt}]
Recall that $t\mapsto h_{e^{-t}}(z) - h_1(z)$ is a standard linear Brownian motion. 
By the Gaussian tail bound, for each $z\in\BB C$ and $j\in\BB N$, the probability of condition~\ref{item-onescale-thick} in the definition of $E_{z,j}$ is at most $ (\ep_j / \ep_{j-1})^{\alpha^2/2  + o_j(1)} $. 
By the independent increments property of Brownian motion, the events described in condition~\ref{item-onescale-thick} for different values of $j \in \BB N$ ($z$ fixed) are independent, so their intersection has probability at most $\ep_n^{\alpha^2/2 + o_n(1)}$. This intersection contains $\mcl E^n(z)$, which gives the upper bound in~\eqref{eqn-thick-1pt}. 

We now prove the lower bound. 
By applying~\eqref{eqn-multiscale-split-lower} of Lemma~\ref{lem-multiscale-split} $n+1$ times, noting that $\mcl E_j^j(z) = E_{z,j} \cap H_{z,j-1}^{\op{out}} \cap H_{z,j-1}^{\op{in}}$ for $j\in\BB N$ and $\mcl E^0(z) = E_{z,0}$, we obtain
\eqb \label{eqn-onepoint-split}
\BB P\left[\mcl E^n(z) \right] \geq c^{n+1} \prod_{j=1}^n \BB P\left[ E_{z,j} \cap H_{z,j-1}^{\op{out}} \cap H_{z,j-1}^{\op{in}} \right] 
\times  \BB P\left[ E_{z,0}  \right] 
\eqe
where $c>0$ is a constant depending only on $ \alpha,\gamma$. 
By Proposition~\ref{prop-E-prob}, for each $\zeta > 0$ there exists $b_\zeta > 0$ such that for each $z\in\BB C$ and $j\in\BB N$, $\BB P[E_{z,j} \cap H_{z,j-1}^{\op{out}} \cap H_{z,j-1}^{\op{in}}] \geq b_\zeta (\ep_j/\ep_{j-1})^{\alpha^2/2 +\zeta}$ and furthermore $\BB P\left[ E_{z,0}   \right] \geq b_\zeta$. Plugging this into~\eqref{eqn-onepoint-split} gives
\eqbn
\BB P\left[\mcl E^n(z) \right] \geq (c b_\zeta)^{n+1} \ep_0^{-\alpha^2/2-\zeta} \ep_n^{\alpha^2/2 + \zeta} \geq \ep_n^{\alpha^2/2 + \zeta + o_n(1)} .
\eqen
Sending $\zeta \rta 0$ now gives the lower bound in~\eqref{eqn-thick-1pt}. 
\end{proof}

The following lemma is the main input in the proof of the two-point estimate, Lemma~\ref{lem-thick-2pt}.

\begin{figure}[t!]
 \begin{center}
\includegraphics[scale=1]{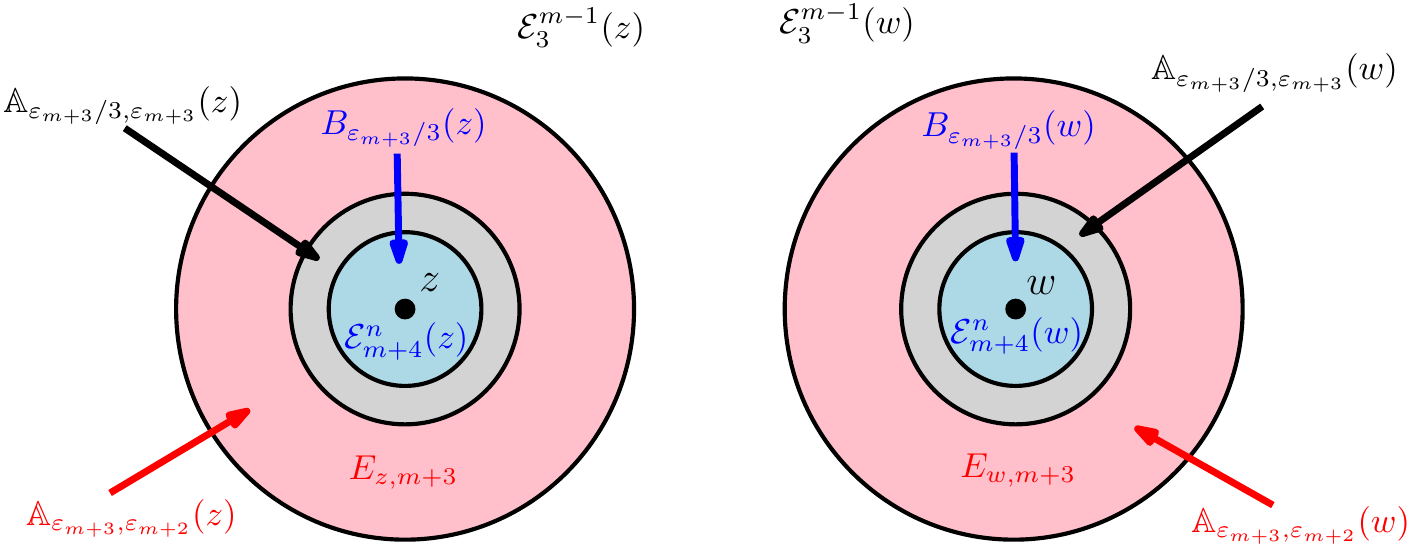}
\vspace{-0.01\textheight}
\caption{Illustration of the proof of Lemma~\ref{lem-twopt-event}. The symbol for each of the events involved in the proof is shown in the region such that the restriction of $h$ to that region determines the event. The events $\mcl E_3^{m-1}(z)$ and $\mcl E_3^{m-1}(w)$ are determined by the restrictions of $h$ to $\BB C\setminus B_{\ep_{m-1}}(z)$ and $\BB C\setminus B_{\ep_{m-1}(w)}$, respectively (not shown), which are subsets of $\BB C\setminus (B_{\ep_{m+2}}(z) \cup B_{\ep_{m+2}}(w))$. 
The key observations for the proof are that the events $\mcl E_{m+4}^n(z)$ and $\mcl E_{m+4}^n(w)$ are conditionally independent given $h|_{ \BB C\setminus (B_{\ep_{m+3}}(z) \cup B_{\ep_{m+3}}(w))}$; and if $E_{z,m+3} \cap E_{w,m+3}$ occurs, then the conditional probability of each of these events is comparable to its unconditional probability. 
}\label{fig-twopt-event}
\end{center}
\vspace{-1em}
\end{figure} 

\begin{lem} \label{lem-twopt-event}
Suppose $z,w\in B_{1/3}(0)$ and $m \in \BB N$ such that $|z-w| \in [  \ep_{m+1} , \ep_m]$. 
Then for $n \geq m$, 
\eqb \label{eqn-twopt-event}
\BB P\left[ \mcl E_{m+3}^n(z) \cap \mcl E_{m+3}^n(w) \,|\, \mcl E_3^{m-1}(z) \cap \mcl E_3^{m-1}(w) \right] 
\preceq \BB P\left[ \mcl E_{m+5}^n(z)  \right] \BB P\left[ \mcl E_{m+5}^n(w) \right] 
\eqe 
with the implicit constant depending only on $K,\alpha,\gamma$.
\end{lem} 
\begin{proof}
See Figure~\ref{fig-twopt-event} for an illustration of the setup. If $n  < m  +5$ the right side of~\eqref{eqn-twopt-event} is equal to 1, so we can assume without loss of generality that $n\geq m+5$. 

Since $|z-w| \in [  \ep_{m+1} , \ep_m]$, we have $B_{ \ep_{m+2}}(z) \cap B_{ \ep_{m+2}}(w) = \emptyset$ and $B_{ \ep_{m+2}}(z) \cup B_{ \ep_{m+2}}(w) \subset B_{\ep_{m-1}}(z)$. 
By the measurability statements from Lemma~\ref{lem-E-msrble} and~\eqref{eqn-Ecap-msrble}, this implies that 
\allb \label{eqn-twopt-msrble}
&\mcl E_3^{m-1}(z) ,  \mcl E_3^{m-1}(w) , E_{z,m+3} , E_{w,m+3} \in \sigma\left( h|_{ \BB C\setminus \left( B_{ \ep_{m+3}}(z) \cup B_{ \ep_{m+3}}(w)\right)} \right) , \notag \\   
&\qquad \qquad \mcl E_{m+4}^n(z)   \in \sigma\left( h|_{B_{ \ep_{m+3}}(z)} \right) , 
\quad \text{and} \quad \mcl E_{m+4}^n(w)   \in \sigma\left( h|_{B_{ \ep_{m+3}}(w)} \right) . 
\alle 
By the Markov property of the GFF, the restrictions of $h$ to $B_{ \ep_{m+3}}(z)$ and $B_{ \ep_{m+3}}(w)$ are conditionally independent given $h|_{ \BB C\setminus \left( B_{ \ep_{m+3}}(z) \cup B_{ \ep_{m+3}}(w)\right)}$, so by~\eqref{eqn-twopt-msrble} the same is true for the events $\mcl E_{m+4}^n(z)$ and $\mcl E_{m+4}^n(w)$. 
Therefore, a.s.\ 
\allb \label{eqn-twopt-split}
&\BB P\left[\mcl E_{m+4}^n(z) \cap \mcl E_{m+4}^n(w) \,|\,   h|_{ \BB C\setminus \left( B_{ \ep_{m+3}}(z) \cup B_{ \ep_{m+3}}(w)\right)}   \right] \notag \\ 
&\qquad\qquad   =\BB P\left[ \mcl E_{m+4}^n(z)  \,|\,   h|_{ \BB C\setminus \left( B_{ \ep_{m+3}}(z) \cup B_{ \ep_{m+3}}(w)\right)}   \right]  
 \BB P\left[ \mcl E_{m+4}^n(w) \,|\,   h|_{ \BB C\setminus \left( B_{ \ep_{m+3}}(z) \cup B_{ \ep_{m+3}}(w)\right)} \right]  .
\alle

By the conditional independence of $h|_{B_{ \ep_{m+3}}(z)} $ and $h|_{B_{ \ep_{m+3}}(w)} $ given $h|_{ \BB C\setminus \left( B_{ \ep_{m+3}}(z) \cap B_{ \ep_{m+3}}(w)\right)}$, a.s.\
\eqb
 \BB P\left[ \mcl E_{m+4}^n(z)  \,|\,   h|_{ \BB C\setminus \left( B_{ \ep_{m+3}}(z) \cup B_{ \ep_{m+3}}(w)\right)}    \right]
 =  \BB P\left[ \mcl E_{m+4}^n(z)  \,|\,   h|_{ \BB C\setminus  B_{ \ep_{m+3}}(z) }   \right] .
\eqe
By Lemma~\ref{lem-multiscale-cond} (applied with $m+1$ in place of $m$), we now obtain that a.s.\ 
\eqb \label{eqn-twopt-cond}
 \BB P\left[ \mcl E_{m+4}^n(z)  \,|\,   h|_{ \BB C\setminus \left( B_{ \ep_{m+3}}(z) \cup B_{ \ep_{m+3}}(w)\right)}    \right] \BB 1_{E_{z,m+3}}
 \preceq \BB P\left[\mcl E_{m+5}^n(z) \right]  \BB 1_{E_{z,m+3}}
\eqe
with the implicit constant depending only on $\alpha,\gamma$. 
By plugging~\eqref{eqn-twopt-cond} and the analogous bound with $w$ in place of $z$ into~\eqref{eqn-twopt-split} and multiplying both sides by $\BB 1_{\mcl E_3^{m-1}(z) \cap \mcl E_3^{m-1}(w)}$, we get that a.s.\ 
\allb \label{eqn-twopt-split'}
&\BB P\left[\mcl E_{m+4}^n(z) \cap \mcl E_{m+4}^n(w) \,|\,   h|_{ \BB C\setminus \left( B_{ \ep_{m+3}}(z) \cup B_{ \ep_{m+3}}(w)\right)}   \right] \BB 1_{\mcl E_3^{m-1}(z) \cap  \mcl E_3^{m-1}(w) \cap  E_{z,m+3} \cap E_{w,m+3}} \notag \\ 
&\qquad\qquad  \preceq 
\BB P\left[ \mcl E_{m+5}^n(z)   \right]  \BB P\left[ \mcl E_{m+5}^n(w)  \right] \BB 1_{\mcl E_3^{m-1}(z) \cap  \mcl E_3^{m-1}(w) \cap  E_{z,m+3} \cap E_{w,m+3}}  .
\alle

By the first measurability statement in~\eqref{eqn-twopt-msrble}, we can take unconditional expectations of both sides of~\eqref{eqn-twopt-split'} to get
\eqb \label{eqn-twopt-event'}
\BB P\left[ \mcl E_{m+4}^n(z) \cap \mcl E_{m+4}^n(w)  \,|\, \mcl E_3^{m-1}(z) \cap  \mcl E_3^{m-1}(w) \cap  E_{z,m+3} \cap E_{w,m+3}\right] 
\preceq \BB P\left[ \mcl E_{m+5}^n(z)  \right] \BB P\left[ \mcl E_{m+5}^n(w) \right] .
\eqe
Since $\mcl E_{m+3}^n(z) \subset \mcl E_{m+4}^n(z) \cap E_{z,m+3} $ and similarly with $w$ in place of $z$, we immediately get~\eqref{eqn-twopt-event} from~\eqref{eqn-twopt-event'}.
\end{proof}

\begin{proof}[Proof of Lemma~\ref{lem-thick-2pt}]
If $n < m+5$ the bound~\eqref{eqn-thick-2pt} is an easy consequence of Lemma~\ref{lem-thick-1pt} since we allow a multiplicative $\ep_m^{o_m(1)}$ error.
Hence we can assume without loss of generality that $n\geq m+5$. 
By Lemma~\ref{lem-twopt-event},
\allb \label{eqn-thick-2pt-start}
&\BB P\left[\mcl E_{m+3}^n(z) \cap \mcl E_3^{m-1}(z) \cap \mcl E_{m+3}^n(w)  \cap \mcl E_3^{m-1}(w) \right]  \notag \\
&\qquad \preceq  \BB P\left[ \mcl E_{m+5}^n(z)  \right] \BB P\left[ \mcl E_{m+5}^n(w)   \right] \BB P\left[ \mcl E_3^{m-1}(z) \cap \mcl E_3^{m-1}(w)  \right] \notag\\
&\qquad \leq  \BB P\left[ \mcl E_{m+5}^n(z)  \right] \BB P\left[ \mcl E_{m+5}^n(w)   \right] \BB P\left[ \mcl E_3^{m-1}(z)   \right]  .
\alle 
Note that in the last line we simply dropped $\mcl E_3^{m-1}(w)$ from the intersection. 
By~\eqref{eqn-multiscale-split-lower} of Lemma~\ref{lem-multiscale-split} (applied with $(n,m)$ replaced by each of $(m-1,1)$ and $(n,m+4)$),
\eqb
\BB P\left[ \mcl E_3^{m-1}(z) \right] \preceq \frac{ \BB P\left[ \mcl E^{m-1}(z) \right] }{\BB P\left[\mcl E^2(z) \right]   }
\quad \text{and} \quad
\BB P\left[ \mcl E_{m+5}^n(z) \right] \preceq \frac{ \BB P\left[ \mcl E^n(z) \right] }{\BB P\left[\mcl E^{m+4}(z) \right]   } .
\eqe
The same is true with $w$ in place of $z$. 
Plugging these estimates into~\eqref{eqn-thick-2pt-start} gives
\eqb
\BB P\left[\mcl E_{m+3}^n(z) \cap \mcl E_3^{m-1}(z) \cap \mcl E_{m+3}^n(w)  \cap \mcl E_3^{m-1}(w) \right]
\preceq   \frac{ \BB P\left[\mcl E^n(z) \right] \BB P\left[\mcl E^n(w) \right] \BB P\left[ \mcl E^{m-1}(z) \right]  }{\BB P\left[\mcl E^{m+4}(z) \right] \BB P\left[\mcl E^{m+4}(w) \right] \BB P\left[\mcl E^2(z) \right]   } .
\eqe
By applying Lemma~\ref{lem-thick-1pt} to lower-bound the probabilities of each of $\mcl E^{m+4}(z)$, $\mcl E^{m+4}(w)$, and $\mcl E^2(z)$ and to upper-bound $\BB P[\mcl E^{m-1}(z)]$, we now obtain~\eqref{eqn-thick-2pt}. 
\end{proof}

\subsection{Deterministic estimates truncated on the short-range event}
\label{sec-short-estimates}

In this subsection, we prove some estimates for the behavior of circle averages and $D_h$-distances on $\mcl E^n(z)$ which will be important for checking that the Frostman measure which we construct in Section~\ref{sec-dim-lower} is supported on $\mcl B_{\BB s}\cap \mcl T_h^\alpha \cap \wh{\mcl T}_h^\alpha$ (Lemma~\ref{lem-perfect-contain}) as well as for some of the probabilistic estimates in the next section. 
The first estimate ensures that $z$ is approximately an $\alpha$-thick point on $\mcl E^n(z)$. 

\begin{lem} \label{lem-E-thick}
There are deterministic functions $\chi: \BB N \rta (0,\infty)$ and $\wt\chi_K : (0,\infty) \rta (0,\infty)$ with $\lim_{j\rta\infty} \chi(j) = \lim_{r\rta0}  \wt\chi_K(r) = 0$ such that $\chi$ depends only on $ \alpha,\gamma$ (not on $K$), $\wt\chi_K$ depends only on $K , \alpha,\gamma$, and the following is true. 
Suppose $z\in B_{1/3}(0)$ and $n\in\BB N$.  
On $\mcl E^n(z)$,  
\eqb \label{eqn-E-thick}
(\alpha - \chi(j) ) \log \ep_j^{-1}  \leq h_{\ep_j}(z)  \leq (\alpha + \chi(j) ) \log \ep_j^{-1} ,\quad\forall j \in [0,n]_{\BB Z}
\eqe
and
\eqb \label{eqn-E-sup}
 (\alpha -  \wt\chi_K(r) ) \log r^{-1}  \leq  h_r(z) \leq   (\alpha + \wt\chi_K(r) ) \log r^{-1}  ,\quad \forall r \in [\ep_n , \ep_0]  .
\eqe  
\end{lem}
\begin{proof}
Throughout the proof we assume that $\mcl E^n(z)$ occurs. 
By condition~\ref{item-onescale-thick} in the definition of $E_{z,j}$, 
\eqb \label{eqn-E-thick-base}
 | h_{\ep_j}(z) - h_{\ep_{j-1}}(z)  -  \alpha \log(\ep_{j-1}/\ep_j)| \leq A[\log(\ep_{j-1}/\ep_j)]^{3/4}  ,\quad \forall j \in [1,n]_{\BB Z} .  
\eqe
Furthermore, the definition of $E_{z,0}$ implies that $|h_{\ep_0}(z)|\leq A$.
By summing these estimates, we obtain~\eqref{eqn-E-thick} with
\eqb 
\chi(j) = \frac{1}{\log(1/\ep_j)} \left( A + A \sum_{k=1}^j [\log(\ep_{k-1}/\ep_k)]^{3/4} \right) .
\eqe

We now establish~\eqref{eqn-E-sup}. By condition~\ref{item-onescale-sup} in the definition of $E_{z,j}$, 
\eqb \label{eqn-E-thick-sup}
|h_r(z) - h_{\ep_{j-1}}(z) - \alpha\log(\ep_{j-1}/r)| \leq K [\log(\ep_{j-1}/\ep_j)]^{3/4} ,\quad\forall r \in [\ep_j , \ep_{j-1}/3],\quad\forall j\in [1,n]_{\BB Z}. 
\eqe
For $j\in [0,n-1]_{\BB Z}$ and $r\in [\ep_j/3 , \ep_j ]$, the field $h_{z,j}^{\op{out}}$ of~\eqref{eqn-in-out-fields} satisfies $h - h_{\ep_j}(z) =  h_{z,j}^{\op{out}}  + \frk h_{z,j}$ and hence the circle average process of $h_{z,j}^{\op{out}}$ satisfies
\eqb
|h_r(z) - h_{\ep_j}(z)|   \leq |(h_{z,j}^{\op{out}})_r(z)|  + \sup_{u\in\bdy B_r(z)} |\frk h_{z,j}^{\op{out}}(u)| .
\eqe 
Therefore, condition~\ref{item-onescale-out} in the definition of $E_{z,j}$ together with the definition~\eqref{eqn-H-def} of $H_{z,j}^{\op{out}}$ yield 
\allb \label{eqn-E-thick-out} 
   \sup_{r \in [\ep_j/3,\ep_j /2]} |h_r(z) - h_{\ep_j }(z)| \leq 2 A    , \quad \forall j \in [0,n-1]_{\BB Z} .
\alle 
Similarly, using $h_{z,j}^{\op{in}}$ we get
\allb \label{eqn-E-thick-in} 
   \sup_{r \in [\ep_j/2,\ep_j ]} |h_r(z) - h_{\ep_j }(z)| \leq 2 A    , \quad \forall j \in [0,n-1]_{\BB Z} .
\alle  
Combining~\eqref{eqn-E-thick-sup}, \eqref{eqn-E-thick-out}, and~\eqref{eqn-E-thick-in} gives us an upper bound for $\sup_{r\in [\ep_j,\ep_{j-1}]} |h_r(z) - h_{\ep_{j-1}}(z)|$ for each $j\in [1,n]_{\BB Z}$. 
Combining this upper bound with~\eqref{eqn-E-thick-base} yields~\eqref{eqn-E-sup} for an appropriate choice of $\wt\chi_K$. 
Note that $\wt\chi_K$ has to depend on $K$ due to the $K$ in~\eqref{eqn-E-thick-sup}.  
\end{proof}

The second estimate of this subsection ensures that $z$ is an approximate \emph{metric} $\alpha$-thick point on $\mcl E^n(z)$ and also that $D_h(0,z)$ is of constant order on $\mcl E^n(z)$ (which will be important in Section~\ref{sec-end-1pt}). 

\begin{lem} \label{lem-E-dist}
Suppose $z\in \BB A_{1/4,1/3}(0)$ and $n,m\in\BB N_0$ with $n\geq m+1$. 
On $\mcl E^n(z)$, 
\eqb \label{eqn-E-dist-multiscale}
\sup_{u,v\in \BB A_{\ep_n,\ep_m}(z)} D_h(u,v) = \ep_m^{\xi(Q-\alpha) + o_m(1)},
\eqe
where the $o_m(1)$ is deterministic and its rate of convergence depends only on $K,\alpha,\gamma$. 
In particular, there is a deterministic constant $S_K = S_K(\alpha,\gamma) > 1$ depending only on $K  , \alpha,\gamma$ such that on $\mcl E^n(z)$,  
\eqb \label{eqn-E-dist}
D_h\left( 0 , B_{\ep_n}(z)    \right) \in \left[ S_K^{-1}  , S_K   \right].
\eqe
\end{lem}
\begin{proof}
Throughout the proof we assume that $\mcl E^n(z)$ occurs and we require all implicit constants and the rate of convergence of all $o(1)$ errors to be deterministic and depend only on $K,\alpha,\gamma$. 
The argument is purely deterministic.
 
To prove the lower bound in~\eqref{eqn-E-dist-multiscale}, we observe that
\eqb \label{eqn-E-dist-lower}
D_h\left( B_{\ep_n}(z) , \bdy B_{\ep_m}(z)   \right) \geq D_h\left( \bdy B_{2\ep_{m+1}}(z) , \bdy B_{3\ep_{m+1}}(z) \right) .
\eqe 
By condition~\ref{item-onescale-annulus} in the definition of $E_{z,m}$, the right side of this inequality is bounded below by $A^{-1} \ep_{m+1}^{\xi Q} e^{\xi h_{\ep_{m+1}}(z)}$, which in turn is bounded below by $\ep_{m+1}^{\xi (Q-\alpha) + o_m(1)}$ by Lemma~\ref{lem-E-thick}. Since $\ep_{m+1} =\ep_m^{1+o_m(1)}$, this gives the lower bound in~\eqref{eqn-E-dist-multiscale}.

To prove the upper bound in~\eqref{eqn-E-dist-multiscale}, we use the triangle inequality and the fact that $D_h(u,v) \leq D_h(u,v;W)$ for any $W\subset\BB C$ and $u,v\in W$ to get
\allb \label{eqn-E-dist-decomp}
\sup_{u,v\in \BB A_{\ep_n,\ep_m}(z)} D_h(u,v)   
&\leq \sum_{j=m+1}^n \sup_{u,v\in \BB A_{\ep_j,\ep_{j-1}/3}(z)} D_h\left( u , v ; \BB A_{\ep_j ,\ep_{j-1}/3}(z) \right)  \notag \\
& \qquad \qquad + \sum_{j=m}^{n-1} \sup_{u,v\in \BB A_{ \ep_j/2 , \ep_j}(z)} D_h\left( u , v ; \BB A_{\ep_{j}/2 , \ep_{j}}(z) \right)  \notag\\
&\qquad\qquad+ \sum_{j=m}^{n-1} \sup_{u,v\in \BB A_{ \ep_j/3 , \ep_j/2}(z)} D_h\left( u , v ; \BB A_{\ep_{j}/3 , \ep_{j}/2}(z) \right)  .
\alle
By condition~\ref{item-onescale-diam} in the definition of $E_{z,j}$ and then Lemma~\ref{lem-E-thick}, the terms in the first sum in~\eqref{eqn-E-dist-decomp} satisfy
\eqb
 \sup_{u,v\in \BB A_{\ep_j,\ep_{j-1}/3}(z)} D_h\left( u , v ; \BB A_{\ep_j ,\ep_{j-1}/3}(z) \right) 
 \preceq \ep_{j-1}^{\xi Q} e^{\xi h_{\ep_{j-1}}(z)}
 \leq  \ep_{j-1}^{\xi (Q-\alpha - \chi(j-1))}    ,
\eqe
which $\chi(j-1) = o_j(1)$ as in Lemma~\ref{lem-E-thick}. 
By summing this over all $j\in [m+1,n]_{\BB Z}$, we see that the first sum in~\eqref{eqn-E-dist-decomp} is bounded above by $\ep_m^{\xi(Q-\alpha) + o_m(1)}$. 

To deal with the second sum in~\eqref{eqn-E-dist-decomp}, we recall that $(h-h_{\ep_j}(z))|_{\BB A_{\ep_j/3,\ep_j}(z)} = h_{z,j}^{\op{in}} + \frk h_{z,j}^{\op{in}}$.
By Weyl scaling (Axiom~\ref{item-metric-f}), 
\allb \label{eqn-H-weyl}
&\sup_{u,v\in \BB A_{ \ep_j/2 , \ep_j}(z)} D_{h-h_{\ep_j}(z)}\left( u , v ; \BB A_{\ep_{j}/2 , \ep_{j}}(z) \right)  \notag\\
&\qquad\leq \left(\sup_{u \in \BB A_{ \ep_j/2 , \ep_j}(z)} e^{\xi \frk h_{z,j}^{\op{in}}(u)}  \right) 
\left(  \sup_{u,v\in \BB A_{ \ep_j/2 , \ep_j}(z)} D_{h_{z,j}^{\op{out}}}\left( u , v ; \BB A_{\ep_{j}/2 , \ep_{j}}(z) \right)    \right) .
\alle
By condition~\ref{item-onescale-out} in the definition of $E_{z,j}$, for $j\in [m,n-1]_{\BB Z}$, the first factor on the right side of~\eqref{eqn-H-weyl} is bounded above by $e^{\xi A}$.   
By the definition~\ref{eqn-H-def} of $H_{z,j}^{\op{in}}$, the second factor is bounded above by $A \ep_j^{\xi Q}$. 
Since $D_{h-h_{\ep_j}(z)} = e^{-\xi h_{\ep_j}(z)} D_h$, we can apply~\eqref{eqn-H-weyl} and then Lemma~\ref{lem-E-thick} to get that 
\eqb
\sup_{u,v\in \BB A_{ \ep_j/2 , \ep_j}(z)} D_h\left( u , v ; \BB A_{\ep_{j}/2 , \ep_{j}}(z) \right) 
\leq A e^{\xi A} \ep_j^{\xi Q} e^{\xi h_{\ep_j}(z)}
\preceq \ep_j^{\xi (Q-\alpha  -\chi(j)) }  . 
\eqe
By summing this over all $j\in [m,n-1]_{\BB Z}$ we get that the second sum in~\eqref{eqn-E-dist-decomp} is bounded above by $\ep_m^{\xi(Q-\alpha) + o_m(1)}$. 
We similarly treat the third sum in~\eqref{eqn-E-dist-decomp}, using $h_{z,j}^{\op{out}}$ and $\frk h_{z,j}^{\op{out}}$ instead of $h_{z,j}^{\op{in}}$ and $\frk h_{z,j}^{\op{in}}$.
This gives the upper bound in~\eqref{eqn-E-dist-multiscale}. 

To deduce~\eqref{eqn-E-dist}, we note that since $z\in \BB A_{1/4,1/3}(0)$ and $\ep_0 = 1/3$, we have $0\in \BB A_{\ep_n,\ep_0}(z)$ and $0 \notin B_{3\ep_1}(z)$. 
We may therefore combine~\eqref{eqn-E-dist-multiscale} with the discussion just after~\eqref{eqn-E-dist-lower} (both with $m=0$) to get~\eqref{eqn-E-dist} for an appropriate choice of $K$. 
\end{proof}

\section{Long-range events}
\label{sec-long-range}

The events $\mcl E^n(z)$ of Section~\ref{sec-short-range} provide all of the regularity events we will need, but there is nothing in the definitions of these events that ensures that $z$ is close to $\bdy\mcl B_{\BB s}$ when $\mcl E^n(z)$ occurs. 
To rectify this, fix $\beta \in \left(0 ,   \xi(Q-\alpha)\right) $ (which we will eventually take to be close to $\xi(Q-\alpha)$). For $z\in\BB C$ and $n\in\BB N$, we define
\eqb \label{eqn-onescale-dist}
G_{z,n}  :=   \left\{ D_h\left( 0 , B_{\ep_n}(z)  \right) \in \left[\BB s  , \BB s + \ep_n^\beta \right] \right\} .
\eqe  
With $\mcl E^n(z)$ as in~\eqref{eqn-Ecap-def}, we also define
\eqb \label{eqn-Gcap-def}
\mcl G^n(z) := \mcl E^n(z) \cap G_{z,n} .
\eqe
The goal of this section is to prove the lower bounds in Theorems~\ref{thm-bdy-dim} and~\ref{thm-thick-dim}.

Before stating our bounds, we recall that $\mcl E^n(z)$ depends on three parameters $A,L,K$, The parameters $A$ and $L$ were fixed in a manner depending only on $\alpha,\gamma$ in Proposition~\ref{prop-E-prob}. The parameter $K\geq A$ has not yet been chosen, and all of the estimates of Section~\ref{sec-short-range} are required to hold for any $K\geq A$ (although the constants involved are sometimes allowed to depend on $K$). In particular, Lemma~\ref{lem-E-dist} shows that for each $K\geq A$ there is an $S_K = S_K(\alpha,\gamma) >1$ such that on $\mcl E^n(z)$, we have $D_h\left( 0 , B_{\ep_n}(z)    \right) \in \left[ S_K^{-1}  , S_K   \right]$. This number $S_K$ in the special case when $K=A$ will play an important role in the results of this subsection. 
The main results of this section are the following one-point and two-point estimates for the events $\mcl G^n(z)$.

\begin{prop}[One-point estimate] \label{prop-end-1pt}
Let $S = S_A(\alpha,\gamma) >1$ be the constant from Lemma~\ref{lem-E-dist} for $K=A$ and suppose that $\BB s > S$.  
If the parameter $K \geq A$ from Section~\ref{sec-short-def} is chosen to be sufficiently large, in a manner depending only on $\BB s,\alpha,\gamma$, then
for each $z\in\BB A_{1/4,1/3}(0)$ and each $n\in\BB N$, 
\eqb \label{eqn-end1pt}
\BB P\left[\mcl G^n(z) \right] \geq \ep_n^{   \alpha^2/2   + \beta + o_n(1)}  
\eqe
with the rate of the $o_n(1)$ depending only on $\BB s , \alpha,\gamma$.
\end{prop}

\begin{prop}[Two-point estimate] \label{prop-end-2pt}
Let $S$ be as above, let $\BB s > S$, and let $K = K(\BB s,\alpha,\gamma) > A$ be chosen as in Proposition~\ref{prop-end-1pt}. 
Suppose $z,w\in\BB A_{1/4,1/3}(0)$, let $m\geq 3$ be such that $|z-w| \in [\ep_{m+1},\ep_m]$, and let $n\geq m+2$. 
Then
\eqb \label{eqn-end-2pt}
\BB P\left[\mcl G^n(z) \cap \mcl G^n(w) \right] \leq \ep_m^{ - \alpha^2/2  -  \xi(Q-\alpha) + o_m(1)} \BB P\left[\mcl G^n(z) \right] \BB P\left[ \mcl G^n(w) \right] 
\eqe
with the rate of the $o_m(1)$ depending only on $\BB s , \alpha,\gamma$.
\end{prop}

In Section~\ref{sec-dim-lower}, we will combine Propositions~\ref{prop-end-1pt} and~\ref{prop-end-2pt} via the usual Frostman measure argument in order to say that for $\BB s > S$, we have $\esssup \dim_{\mcl H}^0(\bdy\mcl B_{\BB s} \cap \mcl T_h^\alpha) \geq 2-\xi(Q-\alpha) - \alpha^2/2$. The following proposition will be used to transfer from the case when $\BB s  >S$ to the case when $\BB s \in (0,S]$.

\begin{prop}[Truncated one-point estimate] \label{prop-end-across}
Let $S$ be as above, let $\BB s > S$, and let $K = K(\BB s,\alpha,\gamma) > A$ be chosen as in Proposition~\ref{prop-end-1pt}. 
For each $C>0$, 
\eqb \label{eqn-end-across}
\BB P\left[\mcl G^n(z) \cap \left\{ D_{ h}\left(\bdy B_{5/6}(0) , \bdy B_{7/8}(0) \right) \geq C  \right\}  \right] \succeq \BB P\left[ \mcl G^n(z) \right] 
\eqe
with the implicit constant depending only on $C,\BB s , \alpha,\gamma$.
\end{prop}

Proposition~\ref{prop-end-across} (applied with a large choice of $C$) will eventually allow us to show that for $\BB s > S$, it holds with positive probability that $\dim_{\mcl H}^0(\bdy\mcl B_{\BB s} \cap \mcl T_h^\alpha)$ is bounded below and also $\mcl B_{\BB s} \subset B_{7/8}(0)$. This, in turn, will allow us to transfer from the case when $\BB s >S$ to the case when $\BB s \in (0,S]$ using Weyl scaling and local absolute continuity argument. See Proposition~\ref{prop-dim-contain}. The reason why we want $\mcl B_{\BB s}$ to be at positive distance from $\bdy\BB D$ for this absolute continuity argument is that for $a >0$, the laws of the restrictions of $h$ and $h+a$ to any compact subset of $\BB D$ are mutually absolutely continuous, but this is not true for the restrictions to $\BB D$ since $h$ is normalized so that $h_1(0) =0$.

\begin{notation} \label{notn-other-h}
Throughout this section, we will frequently work with a distribution $\wt h$ whose law is absolutely continuous with respect to the law of $h$. Typically, $\wt h$ will be obtained by adding a random smooth function to $h$. For an event $E$ depending on $h$, we denote the analogous event with $\wt h$ in place of $h$ by a parenthetical $\wt h$. So, e.g., $\mcl G^n(z;\wt h)$ is the event $\mcl G^n(z)$ defined with $\wt h$ in place of $h$. 
\end{notation}

As discussed in Section~\ref{sec-outline}, the key idea in the proofs of Propositions~\ref{prop-end-1pt}, \ref{prop-end-2pt}, and~\ref{prop-end-across} is similar to the idea for treating the ``long-range" event in Proposition~\ref{prop-event-upper}. The event $\mcl G^n(z)$ is the intersection of the short-range event $\mcl E^n(z)$ and the long-range event $G_{z,n}$. 
We already have the necessary estimates for $\mcl E^n(z)$ from Section~\ref{sec-short-range}. 
To deal with $G_{z,n}$, we will consider a field $\wt h$ obtained by adding a random smooth function to $h$. 
We will prove bounds for the conditional probability of $G_{z,n}(\wt h)$ (using Notation~\ref{notn-other-h}) given $h$, truncated on the event $\mcl E^n(z)$. 
Using the absolute continuity of the laws of $h$ and $\wt h$, we will then deduce the desired estimates for $h$.

Let us now describe the proofs in more detail. Section~\ref{sec-end-1pt} is devoted to the proof of Proposition~\ref{prop-end-1pt}. 
Here, we will take $\wt h = h + X\phi_{z,1}$ where $\phi_{z,1}$ is as in condition~\ref{item-onescale-dirichlet} in the definition of $E_{z,1}$ and $X$ is sampled uniformly from Lebesgue measure on $[0,R]$, independently from $h$, for a large but fixed constant $R>0$. 
If $\BB s > S$ then Lemma~\ref{lem-E-dist} shows that $D_h(0,B_{\ep_n}(z) ) \leq \BB s$ on $\mcl E^n(z)$. Furthermore, if $R$ is chosen to be sufficiently large, depending on $\BB s$, then it is not hard to see using condition~\ref{item-onescale-annulus} in the definition of $E_{z,1}$ that $D_{h+R\phi_{z,1}}(0,B_{\ep_n}(z)) \geq \BB s$ on this event (see Lemma~\ref{lem-dist-deriv}). 
Since $X$ is sampled uniformly from $[0,R]$, we can use Weyl scaling to say that on $\mcl E^n(z)$, the conditional probability of $G_{z,n}(\wt h)$ given $\mcl E^n(z)$ is at least a constant times $\ep_n^\beta$ (Lemma~\ref{lem-dist-1pt-lower'}). Taking unconditional probabilities and applying the estimate for $\BB P[\mcl E^n(z)]$ from Lemma~\ref{lem-thick-1pt} will then give us a lower bound for $\BB P[\mcl G^n(z;\wt h)]$. We then transfer from $\wt h$ back to $h$ via absolute continuity. 

The main technical difficulty in the above argument is that adding $X\phi_{z,1}$ to $h$ affects the occurrence of some of the conditions in the definition of $E_{z,1}$, so we do not necessarily know that $\mcl E^n(z) \subset \mcl E^n(z;\wt h)$. This means that the above argument, applied naively, only gives a lower bound for $\BB P[G_{z,n}(\wt h)\cap \mcl E^n(z)]$ instead of for $\BB P[\mcl G^n(z;\wt h)]$. To get around this, we will start by working with the event $\ul{\mcl E}^n(z)$ defined in the same manner as $\mcl E^n(z)$ but with $K = A$. We will then make the parameter $K$ sufficiently large (depending on $R$) so that $\ul{\mcl E}^n(z) \subset \mcl E^n(z;\wt h)$. This is the reason why we include $K$ in the definition of our events.

Section~\ref{sec-end-2pt} is devoted to the proof of Proposition~\ref{prop-end-2pt}. Here, we work with the field $\wt h = X_z \phi_{z,1} + X_w \phi_{w,m+2}$ where $X_z,X_w$ are sampled uniformly from $[0,1]$ independently from each other and from $h$. We first prove an upper bound for $\BB P[G_{z,n}( \wt h) \cap G_{w,n}( \wt h) | h] \BB 1_{\mcl E^n(z) \cap \mcl E^n(w)}$ (Lemma~\ref{lem-dist-2pt}). We then take unconditional probabilities and apply the two-point estimate for $\mcl E^n(z)$ (Lemma~\ref{lem-thick-2pt}) to obtain an upper bound for $\BB P[\mcl G^n(z;\wt h) \cap\mcl G^n(w;\wt h)]$ and finally transfer back to $h$ via absolute continuity. 
As in Section~\ref{sec-end-1pt}, there are technical difficulties with comparing $\mcl E^n(z)$ and $\mcl E^n(z;\wt h)$, but they are more easily dealt with than in Section~\ref{sec-end-1pt} since we only want an upper bound for probabilities and we allow an $\ep_m^{o_m(1)}$ error. So, we can simply ``skip" the scales which are affected by adding $X_z\phi_{z,1} + X_w\phi_{w,m+2}$; see Lemma~\ref{lem-Eupper-2pt}. 
 
Proposition~\ref{prop-end-across} is obtained by combining estimates in Sections~\ref{sec-end-1pt} and~\ref{sec-end-2pt}. Roughly speaking, the reason why the proposition is true is that $D_{ h}\left(\bdy B_{5/6}(0) , \bdy B_{7/8}(0) \right)$ is determined $h|_{\BB A_{5/6,7/8}(0)}$ whereas $\mcl E^n(z)$ depends on the restriction of $h$ to $B_{1/3}(z) \subset B_{2/3}(0)$, so the events $\mcl E^n(z)$ and $\left\{ D_{ h}\left(\bdy B_{5/6}(0) , \bdy B_{7/8}(0) \right) \geq C  \right\}$ are approximately independent. See Lemma~\ref{lem-Elower-prob}.

\subsection{One-point estimate} 
\label{sec-end-1pt}

In this subsection we will prove Proposition~\ref{prop-end-1pt} and most of Proposition~\ref{prop-end-across}. 
Throughout, to lighten notation we let
\eqb \label{eqn-dist-1pt-across}
F_C := \left\{D_{ h}\left(\bdy B_{5/6}(0) , \bdy B_{7/8}(0) \right) \geq C \right\} ,\quad\forall C > 0
\eqe
be the event appearing in Proposition~\ref{prop-end-across}. 
Most of this subsection is devoted to the proof of the following lemma, which is the key input in the proofs of Propositions~\ref{prop-end-1pt} and~\ref{prop-end-across}.

\begin{lem} \label{lem-dist-1pt-lower}
Let $S = S_A(\alpha,\gamma)$ be as in Proposition~\ref{prop-end-1pt} and let $\BB s > S$. 
Also let $C>1$ and let $F_C$ be as in~\eqref{eqn-dist-1pt-across}.
For each $z\in \BB A_{1/4,1/3}(0)$, and each $n\in\BB N$, 
\eqb \label{eqn-dist-1pt-lower}
\BB P\left[\mcl G^n(z) \cap F_C \right] \succeq \ep_n^\beta \BB P\left[ \mcl E^n(z) \right]  
\eqe
with the implicit constant depending only on $C , \BB s , \alpha,\gamma$.
\end{lem}

For $n\in\BB N$, let $\ul E_{z,1}$ and $\ul{\mcl E}^n(z)$ be the events $E_{z,1}$ and $\mcl E^n(z)$ of Section~\ref{sec-short-def} with $K = A$. 
As in the rest of the paper, we still write $E_{z,1}$ and $\mcl E^n(z)$ for the events of Section~\ref{sec-short-def} with a given choice of $K$. 
Since $K\geq A$ and $K$ appears only in the definition of $E_{z,1}$,  
\eqb \label{eqn-Elower-def}
\ul{\mcl E}^n(z) \subset \mcl E^n(z) \quad \text{and} \quad \ul{\mcl E}^n(z) = \ul{\mcl E}^1(z) \cap \mcl E_2^n(z) .
\eqe
The main reason for separating the events $\mcl E^n(z)$ and $\ul{\mcl E}^n(z)$ is the following lemma.

\begin{lem} \label{lem-Elower-compare}
Let $R > 0$. If we choose the parameter $K$ to be sufficiently large, in a manner depending only on $R,\alpha,\gamma$, then a.s.\ whenever $\ul{\mcl E}^n(z)$ occurs the event $\mcl E^n(z;h+x\phi_{z,1})$ occurs for each $x\in [0,R]$ (where here we use Notation~\ref{notn-other-h}). 
\end{lem}
\begin{proof}
Recall that $\ul{\mcl E}^1(z) = \ul E_{z,1} \cap E_{z,0} \cap H_{z,0}^{\op{out}} \cap H_{z,0}^{\op{in}}  $. 
The function $\phi_{z,1}$ is supported on $\BB A_{\ep_1,4\ep_1}(z)$, so adding a multiple of $\phi_{z,1}$ to $h$ does not affect the occurrence of any of the conditions in the definition of $E_{z,1}$ except possibly conditions~\ref{item-onescale-sup}, \ref{item-onescale-annulus}, \ref{item-onescale-dirichlet}, and \ref{item-onescale-diam}, i.e., each of the other conditions in the definition of $E_{z,1}$ occurs if and only if it occurs with $h$ replaced by $h+x\phi_{z,1}$ for any $x\in\BB R$ (note that for condition~\ref{item-onescale-rn}, we use Lemma~\ref{lem-rn-msrble}). 
Similarly, the occurrence of the event $E_{z,0} \cap H_{z,0}^{\op{out}} \cap H_{z,0}^{\op{in}}  $ is unaffected by adding a multiple of $\phi_{z,1}$ to $h$.
 
We will now deal with the four remaining conditions in the definition of $E_{z,1}$. 
Since $\phi_{z,1}$ takes values in $[0,1]$, Axiom~\ref{item-metric-f} implies that a.s.\ 
\eqb
D_h(u,v) \leq D_{h+x\phi_{z,1}}(u,v) \leq e^{\xi R} D_h(u,v) ,\quad \forall x\in [0,R] .
\eqe
From this, we see that condition~\ref{item-onescale-annulus} in the definition of $\ul E_{z,1}$ implies condition~\ref{item-onescale-annulus} in the definition of $E_{z,1}(h+x\phi_{z,1})$ (since this condition only involves a lower bound for $D_h$-distances). 
Similarly, the diameter upper bound condition~\ref{item-onescale-diam} for $\ul E_{z,1}$ implies condition~\ref{item-onescale-diam} for $E_{z,1}(h+x\phi_{z,1})$ provided $K \geq A e^{\xi R}$. 

For $x\in [0,R]$, adding $x \phi_{z,1}$ to $h$ increases the circle average process of $h$ by at most $R$ and increases the absolute value of the Dirichlet inner product $(h,\phi_{z,1})_\nabla$ by at most $R (\phi_{z,1},\phi_{z,1})_\nabla = R (\phi,\phi)_\nabla$. 
From this, we see that conditions~\ref{item-onescale-sup} and~\ref{item-onescale-diam} in the definition of $\ul E_{z,1}$ imply the corresponding conditions for $E_{z,1}(h+x\phi_{z,1})$ provided $K\geq A + R(1\vee (\phi,\phi)_\nabla)$. 
\end{proof}

Lemma~\ref{lem-dist-1pt-lower} will eventually be deduced from the following estimate together with a local absolute continuity argument. 

\begin{lem} \label{lem-dist-1pt-lower'} 
For each $\BB s > S$, there exists $R = R(\BB s , \alpha,\gamma) > 0$ such that the following is true. 
Let $X$ be sampled uniformly from $[0 , R]$, independently from $h$. 
Also let $z\in \BB A_{1/4,1/3}(0)$ and $n\in\BB N$.
On $\ul{\mcl E}^n(z)$, a.s.\
\eqb \label{eqn-dist-1pt-lower'}
\BB P\left[G_{z,n}(h + X\phi_{z,1} ) \,|\, h \right] \succeq \ep_n^\beta
\eqe
where the implicit constant is deterministic and depends only on $\BB s ,  \alpha,\gamma$.
\end{lem}
 
It is crucial for our purposes that the constant $R$ from Lemma~\ref{lem-dist-1pt-lower'} does \emph{not} depend on $K$, since $K$ needs to depend on $R$ due to Lemma~\ref{lem-Elower-compare}. 
This is why we truncate on $\ul{\mcl E}^n(z)$ (which does not depend on $K$) instead of on $\mcl E^n(z)$ in the lemma statement.

For the proof of Lemma~\ref{lem-dist-1pt-lower'}, we need the following lower bound for how much we perturb distances when we add a multiple of $\phi_{z,m}$ to $h$, which is a consequence of Lemma~\ref{lem-theta-deriv}.

\begin{lem} \label{lem-dist-deriv}
Let $z\in \BB A_{1/4,1/3}(0)$ and let $n,m \in\BB N$ with $m\leq n$. 
Recalling the smooth bump function $\phi_{z,m}$ from condition~\ref{item-onescale-dirichlet} in the definition of $E_{z,m}$, we define
\eqb \label{eqn-dist-f-def}
\theta(x) = \theta_{z,m}^n(x) := D_{h + x\phi_{z,m}}\left(  0 , B_{\ep_n}(z)   \right),\quad\forall x \geq 0. 
\eqe 
Then a.s.\ $\theta$ is strictly increasing and locally Lipschitz continuous and on $\mcl E^m(z)$, a.s.\ 
\eqb \label{eqn-dist-deriv}
\theta'(x) \succeq e^{\xi x} \ep_m^{\xi Q} e^{\xi h_{ \ep_m}(z)}   ,\quad \text{for Lebesgue-a.e.\ $x\geq 0$}
\eqe 
with a deterministic implicit constant depending only on $\alpha,\gamma$ (not on $K$).  
\end{lem}
\begin{proof}
By Lemma~\ref{lem-theta-deriv} applied with $K_1 = \ol{B_{\ep_n}(z)} $, $K_2 =\{0\}$, and $\mcl A = \BB A_{2\ep_m , 3\ep_m}(z)$, $\theta$ is strictly increasing and locally Lipschitz continuous and 
\eqb\label{eqn-dist-deriv0}
\theta'(x) \geq \xi e^{\xi x}  D_h\left(\bdy B_{2\ep_m}(z) , \bdy B_{3\ep_m}(z) \right)  ,\quad  \text{for Lebesgue-a.e.\ $x\geq 0$} .
\eqe
Condition~\ref{item-onescale-annulus} in the definition of $E_{z,m}$ shows that the right side of~\eqref{eqn-dist-deriv0} is at least $ A^{-1} \xi   e^{\xi x}  \ep_m^{\xi Q} e^{\xi h_{ \ep_m}(z)}$, which concludes the proof.
\end{proof}

\begin{proof}[Proof of Lemma~\ref{lem-dist-1pt-lower'}]
Throughout the proof we condition on $h$ and assume that $\ul{\mcl E}^n(z)$ occurs. 
As in Lemma~\ref{lem-dist-deriv} (with $m=1$), for $x\geq 0$ we define
$\theta(x) := D_{h + x\phi_{z,1}}\left(  0 , B_{\ep_n}(z)   \right)$. 
By Lemma~\ref{lem-E-dist} (applied with $K=A$),  
\eqb \label{eqn-use-E-dist}
\theta(0) = D_h(0,B_{\ep_n}(z) ) \in [S^{-1}    ,S   ] .
\eqe

Lemma~\ref{lem-dist-deriv} provides a lower bound for $\theta'(x)$. We will also need a corresponding upper bound, whose proof is much easier. 
Since $\phi_{z,1} \leq 1$, we have $\theta(y) \leq e^{\xi(y-x)} \theta(x) $ for each $x,y\geq 0$ with $x < y$.
This implies that that $\theta(y) - \theta(x) \leq   (e^{\xi(y-x)} - 1) \theta(x)$ and hence that $\theta'(x) \leq \xi \theta(x)  \leq \xi e^{\xi x} \theta(0)$ for Lebesgue-a.e.\ $x\geq 0$.  
By combining this with~\eqref{eqn-use-E-dist}, we get
\eqb  \label{eqn-dist-deriv-upper}
\theta'(x) \leq \xi  S  e^{\xi x}   ,\quad \text{for Lebesgue-a.e.\ $x\geq 0$} .
\eqe 
 
The bound~\eqref{eqn-dist-deriv-upper} together with the fundamental theorem of calculus implies that for any $R>0$ and any $\BB s \in [\theta(0) , \theta(R) - \ep_n^\beta]$, the Lebesgue measure of the set of $x\in [0,R]$ for which $\theta(x) \in [\BB s , \BB s + \ep_n^\beta ]$ is at least $ \xi^{-1} S^{-1} e^{-\xi R} \ep_n^\beta $. 
Hence, if $X$ is uniformly distributed on $[0,R]$, then for any $\BB s \in [\theta(0) , \theta(R) - \ep_n^\beta]$,
\eqb
\BB P\left[G_{z,n}(h + X\phi_{z,1} ) \,|\, h \right]
= \BB P\left[ \theta(X) \in [\BB s , \BB s + \ep_n^\beta ] \,|\, h \right] 
\geq  \frac{ \ep_n^\beta}{ \xi S R e^{\xi R} } .
\eqe
 
To prove~\eqref{eqn-dist-1pt-lower'}, it remains to show that if $\BB s >S$, then for a large enough choice of $R  =R(\BB s , \alpha , \gamma)$, we have $\BB s \in [\theta(0) , \theta(R) - \ep_n^\beta]$. 
By integrating the bound~\eqref{eqn-dist-deriv} from Lemma~\ref{lem-dist-interval}, we obtain that for some $\alpha,\gamma$-dependent constant $c>0$,  
\eqb  \label{eqn-dist-1pt-int}
\theta(R) \geq  \theta(0) +  c \ep_1^{\xi Q} e^{\xi h_{ \ep_1}(z)} \int_0^R  e^{\xi x} \,dx  = \theta(0) + c \ep_1^{\xi Q} e^{\xi h_{ \ep_1}(z)} (e^{\xi R} - 1)   .
\eqe 
By Lemma~\ref{lem-E-thick}, $\ep_1^{\xi Q} e^{\xi h_{\ep_1}(z)}$ is bounded below by some deterministic $\alpha,\gamma$-dependent constant $c' > 0$.
Hence~\eqref{eqn-use-E-dist} and~\eqref{eqn-dist-1pt-int} together imply that $\theta(R) \geq S^{-1} + c c' (e^{\xi R}-1) $. 
Hence, by choosing $R$ sufficiently large, in a manner depending only on $\BB s ,\alpha,\gamma$, we can arrange that $\theta(R) \geq \BB s + 1$, as required.
\end{proof}

In the rest of this section we will fix $\BB s  > S$ and let $R = R(\BB s,\alpha,\gamma)$ be as in Lemma~\ref{lem-dist-1pt-lower'}. 
We also assume that $K = K(R,\alpha,\gamma ) > A$ is chosen as in Lemma~\ref{lem-Elower-compare} for this choice of $R$, so that $K$ depends only on $\BB s , \alpha,\gamma$.
The following lemma is needed to compare the probabilities of $\ul{\mcl E}^n(z)$ and $\mcl E^n(z)$. We also include the event $F_C$, with a view toward Proposition~\ref{prop-end-across}. 
 
\begin{lem} \label{lem-Elower-prob}
Let $C>0$ and recall the event $F_C$ from~\eqref{eqn-dist-1pt-across}. 
For each $z\in\BB A_{1/4,1/3}(0)$ and each $n\in\BB N$,
\eqb
\BB P\left[ \ul{\mcl E}^n(z) \cap F_C \right] \succeq \BB P\left[\mcl E^n(z) \right]
\eqe
with the implicit constant depending only on $C,\BB s, \alpha,\gamma$.
\end{lem}

We first treat the special case when $n =1$. 

\begin{lem} \label{lem-Elower-prob-across}   
For each $C>0$ and each $z\in\BB A_{1/4,1/3}(0)$, 
\eqb
\BB P\left[ \ul{\mcl E}^1(z) \cap F_C \right] \succeq 1 
\eqe
with the implicit constant depending only on $C,  \alpha,\gamma$.
\end{lem}
\begin{proof}
By~\eqref{eqn-Ecap-msrble}, the event $\ul{\mcl E}^1(z)$ is a.s.\ determined by the restriction of $h$ to $B_{1/3}(z) \subset B_{2/3}(0)$ and by Lemma~\ref{lem-thick-1pt} (applied with $K=A$) the probability of this event is bounded below by a positive constant depending only on $\alpha,\gamma$. 
By the locality of the metric (Axiom~\ref{item-metric-local}), the event $F_C$ is a.s.\ determined by $h|_{\BB A_{5/6,7/8}(0)}$. Furthermore, by adding a large smooth bump function supported in $\BB A_{5/6,7/8}(0)$ to $h$ and noting that this affects the law of $h$ in an absolutely continuous way, one easily sees that $\BB P[F_C] > 0$ for any choice of $C>0$. 

By the Markov property of the GFF, the conditional law of $h|_{B_{2/3}(0)}$ given $h|_{\BB A_{5/6,7/8}(0)}$ is a.s.\ mutually absolutely continuous with respect to the marginal law of $h|_{B_{2/3}(0)}$. Combining this with the preceding paragraph concludes the proof.
\end{proof}

\begin{proof}[Proof of Lemma~\ref{lem-Elower-prob}]
Due to Lemma~\ref{lem-E-msrble}, the event $\ul{\mcl E}^1(z)  \cap F_C$ is a.s.\ determined by $h|_{\BB C\setminus B_{\ep_1}(z)}$.
Furthermore, this event is contained in $ E_{z,1}$.   
Consequently, Lemma~\ref{lem-multiscale-cond} (applied with $m=1$) implies that a.s.\ 
\eqbn
\BB P\left[ \mcl E_2^n(z) \,|\, h|_{\BB C\setminus B_{\ep_1}(z) }   \right] \BB 1_{\ul{\mcl E}^1(z)  \cap F_C}
\succeq  \BB P\left[ \mcl E_2^n(z)   \right] \BB 1_{\ul{\mcl E}^1(z)  \cap F_C} .
\eqen
Taking unconditional expectations of both sides and using Lemma~\ref{lem-Elower-prob-across} now gives
\eqbn
\BB P\left[ \ul{\mcl E}^n(z) \cap F_C \right] \succeq \BB P\left[ \mcl E_2^n(z) \right]  \geq \BB P\left[\mcl E^n(z) \right] .
\eqen
\end{proof}

\begin{proof}[Proof of Lemma~\ref{lem-dist-1pt-lower}]
Let $R = R(\BB s,\alpha,\gamma)$ be as in Lemma~\ref{lem-dist-1pt-lower'} and let $X$ be sampled uniformly from $[0,R]$. 
Define the field 
\eqbn
\wt h := h + X \phi_{z,1} .
\eqen
By Lemma~\ref{lem-dist-1pt-lower'}, 
\eqb  \label{eqn-dist-1pt-cond}
\BB P\left[  G_{z,n}(\wt h ) \,|\, h \right] \BB 1_{\ul{\mcl E}^n(z ) \cap F_C }  \succeq \ep_n^\beta \BB 1_{\ul{\mcl E}^n(z ) \cap F_C }  .
\eqe 
Since the support of $\phi_{z,1}$ is disjoint from $\BB A_{5/6,7/8}(0)$, we have $D_h(\bdy B_{5/6}(0) , \bdy B_{7/8}(0) ) = D_{\wt h}(\bdy B_{5/6}(0) , \bdy B_{7/8}(0) )$ and hence $F_C = F_C(\wt h)$.
By Lemma~\ref{lem-Elower-compare}, if $\ul{\mcl E}^n(z )$ occurs then also $\mcl E^n(z;\wt h)$ occurs.
Therefore, we can take unconditional expectations of both sides of~\eqref{eqn-dist-1pt-cond} and use Lemma~\ref{lem-Elower-prob} to lower-bound the right side to get
\allb  \label{eqn-shifted-1pt}
\BB P\left[ \mcl G^n(z;\wt h) \cap F_C(\wt h)   \right] 
\geq \BB P\left[G_{z,n}(\wt h ) \cap \ul{\mcl E}^n(z)   \cap F_C    \right] 
\succeq \ep_n^\beta \BB P\left[\ul{\mcl E}^n(z ) \cap F_C  \right] 
\succeq \ep_n^\beta  \BB P\left[ \mcl E^n(z) \right] .
\alle
 
By Lemma~\ref{lem-gff-abs-cont}, if we condition on $X$, then the conditional law of $\wt h$ is mutually absolutely continuous with respect to the law of $h$. 
The Radon-Nikodym derivative of the latter law w.r.t.\ the law of the former law is
\eqb \label{eqn-1pt-rn}
M_X = M_X(\wt h)  = \exp\left( - X  (\wt h,  \phi_{z,1})_\nabla + \frac12 X^2 (\phi_{z,1} , \phi_{z,1})_\nabla \right) .
\eqe
By condition~\ref{item-onescale-dirichlet} in the definition of $E_{z,1}(\wt h)$, on $\mcl E^n(z;\wt h)$ we have $|(\wt h,\phi_{z,1})_\nabla| \leq K$. 
Since $(\phi_{z,1} , \phi_{z,1})_\nabla = (\phi,\phi)_\nabla $ and $X\in [0,R]$, it follows that on $\mcl E^n(z;\wt h)$, the Radon-Nikodym derivative~\eqref{eqn-2pt-rn} is bounded above and below by deterministic constants which depend only on $\BB s , \alpha,\gamma$. 

We may therefore compute 
\alb
\BB P\left[\mcl G^n(z  ) \cap F_C   \right]  
&= \BB P\left[\mcl G^n(z   ) \cap F_C  \,|\, X \right] \quad \text{($X, h$ are independent)} \notag\\ 
&= \BB E\left[  M_X \BB 1_{ \mcl G^n(z ;\wt h ) \cap F_C(\wt h) } \,|\, X \right]  \notag \\ 
&\succeq \BB P\left[   \mcl G^n(z ;\wt h )    \cap F_C(\wt h)  \,|\,X \right]  .
\ale
Taking unconditional expectations of both sides of this last inequality and using~\eqref{eqn-shifted-field-2pt} now gives
\eqbn
\BB P\left[   \mcl G^n(z  )   \cap F_C  \right]
\succeq   \ep_n^{ \beta} \BB P\left[ \mcl E^n(z  ) \right]  ,
\eqen 
as required.
\end{proof}

\begin{proof}[Proof of Proposition~\ref{prop-end-1pt}]
This is immediate from the lower bound for $\BB P[\mcl E^n(z)]$ from Lemma~\ref{lem-thick-1pt} combined with Lemma~\ref{lem-dist-1pt-lower}. 
\end{proof}

\subsection{Two-point estimate} 
\label{sec-end-2pt}

The goal of this subsection is to finish the proofs of Propositions~\ref{prop-end-2pt} and~\ref{prop-end-across}. 
Throughout, we fix $\BB s  > S$ and we assume that $K = K(\BB s , \alpha,\gamma)$ is chosen as in Proposition~\ref{prop-end-1pt}. 
Most of the subsection is devoted to the proof of the following lemma, which combines with Lemma~\ref{lem-dist-1pt-lower} to yield Proposition~\ref{prop-end-2pt}. 

\begin{lem} \label{lem-full-2pt}
Suppose $z,w\in \BB A_{1/4,1/3}(0)$, let $m\geq 3$ be such that $|z-w| \in [\ep_{m+1} , \ep_m]$, and let $n\geq m+2 $.
Then
\eqb \label{eqn-full-2pt}
\BB P\left[   \mcl G^n(z  ) \cap \mcl G^n(w )   \right]
\leq \ep_m^{ - \alpha^2/2 -\xi(Q-\alpha) - o_m(1)}  \ep_n^{2\beta} \BB P\left[ \mcl E^n(z ) \right] \BB P\left[ \mcl E^n(w ) \right] ,
\eqe 
with the rate of the $o_m(1)$ depending only on $\BB s , \alpha,\gamma$ (not on $n$ or the particular choices of $z$ and $w$). 
\end{lem}
 
As in Section~\ref{sec-end-1pt}, we need a variant of $\mcl E^n(z)$ to deal with the fact that adding a smooth function to $h$ can affect the occurrence of $\mcl E^n(z)$. 
For $m \in \BB N$, let $\ol{\mcl E}^{n,m}(z)$ be the event that there exists a continuous function $f : \BB C\rta [-1,1]$ which is supported on 
\eqb \label{eqn-Eupper-support}
\left( \BB C\setminus B_{\ep_2}(z)  \right) \cup \BB A_{\ep_{m+2} , \ep_{m-1}}(z)
\eqe
such that $\mcl E^n(z)$ occurs with $h+f$ in place of $h$. Obviously, $\ol{\mcl E}^{n,m}(z) \supset \mcl E^n(z)$ (take $f = 0$). 
We have the following two-point estimate for $\ol{\mcl E}^{n,m}(z)$. 

\begin{lem} \label{lem-Eupper-2pt}
Let $z,w\in \BB A_{1/4,1/3}(0)$, let $m\in\BB N$ such that $|z-w| \in [\ep_{m+1} , \ep_m]$, and let $n\in\BB N$ with $n\geq m $. Then
\eqb \label{eqn-Eupper-2pt}
\BB P\left[\ol{\mcl E}^{n,m}(z) \cap \ol{\mcl E}^{n,m}(w) \right] \leq \ep_m^{-\alpha^2/2 + o_m(1)}  \BB P\left[ \mcl E^n(z) \right] \BB P\left[ \mcl E^n(w) \right] 
\eqe
 with the rate of the $o_m(1)$ depending only on $\alpha,\gamma$. 
\end{lem}
\begin{proof}
Recall that from~\eqref{eqn-Ecap-msrble} that $\mcl E_{n_1}^{n_2}(z) \in \sigma\left( h|_{\BB A_{\ep_{n_1},\ep_{n_2-1}}(z)} \right)$ for each $n_1,n_2\in\BB N$.  
Since the function $f$ in the definition of $\ol{\mcl E}^{n,m}(z)$ is supported on the set~\eqref{eqn-Eupper-support}, it follows that 
\eqb \label{eqn-Eupper-contain}
\ol{\mcl E}^{n,m}(z) \subset \mcl E_3^{m-1}(z) \cap \mcl E_{m+3}^n(z) .
\eqe
The same is true with $w$ in place of $z$. Hence~\eqref{eqn-Eupper-2pt} is an immediate consequence of the second inequality in Lemma~\ref{lem-thick-2pt}.  
\end{proof}

The following lemma will be combined with an absolute continuity argument to obtain Lemma~\ref{lem-full-2pt}. 

\begin{lem}  \label{lem-dist-2pt}
Let $z,w\in \BB A_{1/4,1/3}(0)$ and $m\in\BB N$ such that $|z-w| \in [\ep_{m+1} ,\ep_{m}]$.
Let $X_z$ and $X_w$ be sampled from Lebesgue measure on $[0,1]$, independently from each other and from $h$, and let 
\eqbn
\wt h := h + X_z \phi_{z,1} + X_w \phi_{w,m+2} .
\eqen
If $n\geq m+2$, then on $\ol{\mcl E}^{n,m}(z) \cap \ol{\mcl E}^{n,m}(w)$, 
\eqb \label{eqn-dist-2pt}
\BB P\left[ G_{z,n}(\wt h) \cap G_{w,n}(\wt h) \, \big| \, h \right] 
\preceq   \ep_m^{-\xi(Q-\alpha) - o_m(1)} \ep_n^{2\beta} 
\eqe
with the implicit constant and the rate of the $o_m(1)$ depending only on $ \alpha,\gamma$ (not on $n$ or the particular choices of $z$ and $w$).  
\end{lem}

To prove Lemma~\ref{lem-dist-2pt}, we first consider the simpler situation where we add just one random smooth function to $h$. 

\begin{lem} \label{lem-dist-interval}
Let $z \in \BB A_{1/4,1/3}(0)$ and $  m,n\in\BB N$ with $m\leq n$. 
Let $X$ be sampled from Lebesgue measure on $[0,1]$, independently from $h$.  
If $\ol{\mcl E}^{m,m}(z)$ occurs, 
then
\eqb \label{eqn-dist-interval} 
\BB P\left[ G_{z,n}\left(h + X \phi_{z,m} \right)   \, \big| \, h \right] 
\leq \ep_m^{-\xi(Q-\alpha) - o_m(1)}   \ep_n^\beta  
\eqe
with the rate of convergence of the $o_m(1)$ depending only on $ \BB s , \alpha,\gamma$. 
\end{lem}
\begin{proof}
Throughout the proof we condition on $h$ and assume that $\ol{\mcl E}^{m,m}(z)$ occurs.

For $x \geq 0$, define $\theta(x) =  D_{h + x\phi_{z,m}}\left(  0 , B_{\ep_n}(z)   \right)$ as in Lemma~\ref{lem-dist-deriv}. 
Recall that by the definition of $\ol{\mcl E}^{m,m}(z)$, the event $E_{z,m}(h+f)$ occurs for some function $f : \BB C\rta [-1,1]$. Since $D_{h+f} \leq e^{ \xi} D_h$, condition~\ref{item-onescale-annulus} in the definition of $E_{z,m}$ implies that on $\ol{\mcl E}^{m,m}(z)$, we have $ D_h\left( \bdy  B_{2\ep_j}(z)  , \bdy  B_{3\ep_j}(z) \right) \geq A^{-1} e^{-\xi} \ep_j^{\xi Q} e^{\xi h_{\ep_j}(z)}$. 
Using this bound in the proof of Lemma~\ref{lem-dist-deriv}, we see that the conclusion of Lemma~\ref{lem-dist-deriv} holds with $\ol{\mcl E}^{m,m}(z)$ in place of $\mcl E^m(z)$, i.e., on $\ol{\mcl E}^{m,m}(z)$,
\eqb \label{eqn-dist-deriv'}
\theta'(x) \succeq \ep_m^{\xi Q} e^{\xi h_{ \ep_m}(z)}   ,\quad \text{for Lebesgue-a.e.\ $x\geq 0$}
\eqe 
with the implicit constant depending only on $\BB s  , \alpha,\gamma$. Note that we have dropped a factor of $e^{\xi x}$, which is at least 1. 
 
The bound~\eqref{eqn-dist-deriv'} together with the fundamental theorem of calculus implies that the Lebesgue measure of the set of $x\in [0 , 1]$ for which $\theta(x) \in [\BB s , \BB s + \ep_n^\beta ]$ is at most a $\BB s , \alpha,\gamma$-dependent constant times $ \ep_n^\beta / ( \ep_m^{\xi Q} e^{\xi h_{ \ep_m}(z)})$. 
Since $X$ is sampled uniformly from Lebesgue measure on $[0 , 1]$, independently from $h$, this shows that
\eqb
\BB P\left[ G_{z,n}\left(h + X \phi_{z,m} \right)   \, \big| \, h \right] 
= \BB P\left[ \theta(X) \in [\BB s , \BB s  +\ep_n^\beta ] \,\big|\, h \right]  
\preceq \frac{ \ep_n^\beta  }{ \ep_m^{\xi Q} e^{\xi h_{ \ep_m}(z)} } . 
\eqe
By Lemma~\ref{lem-E-thick}, on $\mcl E^m(z)$ we have $e^{\xi h_{ \ep_m}(z)}  = \ep_m^{-\xi \alpha + o_m(1)}$. On $\ol{\mcl E}^{m,m}(z)$, the event $\mcl E^m(z ; h+f)$ occurs for some smooth function $f$ taking values in $[-1,1]$, so also $e^{\xi h_{ \ep_m}(z)}  = \ep_m^{-\xi \alpha + o_m(1)}$ on $\ol{\mcl E}^{m,m}(z)$. This gives~\eqref{eqn-dist-interval}. 
\end{proof}

The following lemma will tell us that in the setting of Lemma~\ref{lem-dist-2pt}, $D_{\wt h}(0,B_{\ep_n}(z))$ does not depend on the value of $X_w$. 
This will allow us to reduce Lemma~\ref{lem-dist-2pt} to Lemma~\ref{lem-dist-interval}.

\begin{figure}[t!]
 \begin{center}
\includegraphics[scale=.75]{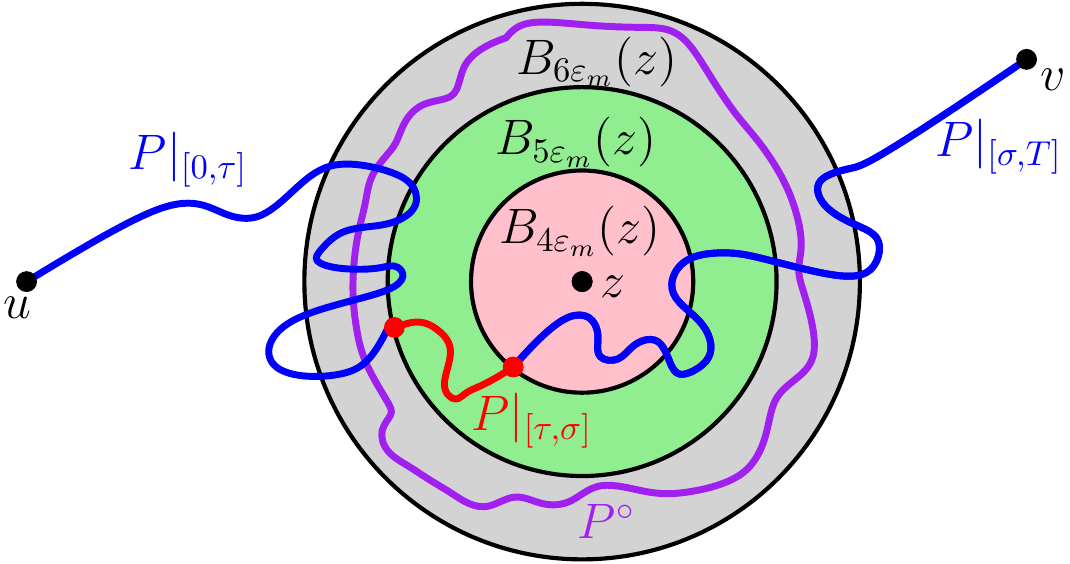}
\vspace{-0.01\textheight}
\caption{Illustration of the proof of Lemma~\ref{lem-geo-away}. The purple path $P^\circ$ is provided by condition~\ref{item-onescale-around} in the definition of $E_{z,m}$. Its $D_{h+g}$-length is much shorter than the $D_{h+g}$-distance between the inner and outer boundaries of the green annulus, and hence much shorter than $\sigma-\tau$. If $P$ enters $B_{4\ep_m}(z)$, then we get a path from $u$ to $v$ which is shorter than $P$ by concatenating a segment of $P$ with a segment of $P^\circ$, so $P$ cannot have near-minimal $D_{h+g}$-length among paths from $u$ to $v$. 
}\label{fig-geo-away}
\end{center}
\vspace{-1em}
\end{figure}

\begin{lem} \label{lem-geo-away}
Let $z\in \BB A_{1/4,1/3}(0)$ and $m\in\BB N$. 
On $\ol{\mcl E}^{m,m}(z)$, there a.s.\ exists a random $\delta > 0$ such that the following is true.
Let $u,v\in\BB C\setminus B_{6\ep_{m}}(z)$ and let $g : \BB C\rta [-1,1]$ be a continuous function.
Then no path of $D_{h+g}$-length smaller than $D_{h+g}(u,v ) + \delta$ can enter $B_{4\ep_{m}}(z)$.
\end{lem}
\begin{proof} 
See Figure~\ref{fig-geo-away} for an illustration of the proof.
On $\ol{\mcl E}^{m,m}(z)$, the event $E_{z,m}(h+f)$ occurs for some function $f : \BB C\rta [-1,1]$. Since $e^{-2\xi} D_{h+f} \leq D_{h+g} \leq e^{ 2\xi} D_{h+f}$ (by Axiom~\ref{item-metric-f}), condition~\ref{item-onescale-annulus} in the definition of $E_{z,m}$ implies that 
\begin{enumerate}[($*$)] \label{item-use-onescale-annulus}
\item There is a path $P^\circ$ in $\BB A_{5\ep_m,6\ep_m} (z)$ which disconnects the inner and outer boundaries of $\BB A_{5\ep_m , 6\ep_m} (z)$ and whose $D_{h+g}$-length is at most $\frac{e^{4\xi}}{100}  D_{h+g}\left( \bdy B_{4\ep_m}(z) , \bdy B_{5\ep_m}(z) \right)$. 
\end{enumerate}
Note that $\xi < 1/2$ so $e^{4\xi}/100 < 1$. 

Now let $P : [0,T] \rta \BB C$ be a path from $u$ to $v$ which enters $B_{4\ep_m}(z)$, parameterized by $D_{h+f}$-length. 
Let $\sigma$ be the first time that $P$ enters $B_{4\ep_m}(z)$ and let $\tau$ be the last time before $\sigma$ at which $P$ enters $B_{5\ep_m}(z)$. 
Then $P([\tau,\sigma])\subset \ol{\BB A_{4\ep_m,5\ep_m}(z)}$ and 
\eqb \label{eqn-crossing-compare}
\sigma-\tau \geq   D_{h+g}\left( \bdy B_{4\ep_m}(z) , \bdy B_{5\ep_m}(z) \right) .
\eqe 
Since $u,v\notin  B_{6\ep_m}(z)$, the path $P$ must cross between the inner and outer boundaries of $\BB A_{5\ep_m , 6\ep_m} (z)$ before time $\tau$ and after time $\sigma$. 
Consequently, there are times $t \in [0,\tau]$ and $s \in [\sigma,T]$ for which $P(t)$ and $P(s)$ lie in the range of the path $P^\circ$ from~\eqref{item-use-onescale-annulus}.
By concatenating $P|_{[0,t]}$, a segment of $P^\circ$, and $P|_{[s,T]}$, we get a path from $u$ to $v$ with $D_{h+g}$-length at most
\alb
&t + (T-s) + \frac{e^{4\xi}}{100}   D_{h+g}\left( \bdy B_{4\ep_m}(z) , \bdy B_{5\ep_m}(z) \right) \notag \\
&\qquad \leq t + (T-s) + \sigma-\tau  - \left( 1 - \frac{e^{4\xi}}{100} \right)   D_{h+g}\left( \bdy B_{4\ep_m}(z) , \bdy B_{5\ep_m}(z) \right)  \notag\\
&\qquad \leq T - \left( 1 - \frac{e^{4\xi} }{100} \right)   D_{h+g}\left( \bdy B_{4\ep_m}(z) , \bdy B_{5\ep_m}(z) \right)
\ale
where in the first inequality we use~\eqref{eqn-crossing-compare}.
Since $D_{h+g} \geq e^{- \xi} D_h$, we therefore have
\eqbn
T \geq D_{h+g}(u,v) + \delta \quad \text{for} \quad \delta = \left( 1 - \frac{ e^{4\xi} }{100} \right)  e^{-\xi } D_h\left( \bdy B_{4\ep_m}(z) , \bdy B_{5\ep_m}(z) \right) .
\eqen
\end{proof}

\begin{proof}[Proof of Lemma~\ref{lem-dist-2pt}]
Throughout the proof, we condition on $h$ and we assume that $\ol{\mcl E}^{n,m}(z) \cap \ol{\mcl E}^{n,m}(w)$ occurs.

We first argue using Lemma~\ref{lem-geo-away} that $D_{\wt h}\left( 0 , B_{\ep_n}(z)  \right)$ is determined by $h$ and $X_z$ (it does not depend on $X_w$). 
Indeed, since $n\geq m+2$ and $|z-w| \geq \ep_{m+1}$, we have $ \ol{B_{\ep_n}(z) } \cap \ol{B_{6\ep_{m+2}}(w)} = \emptyset$.
Since $w\in\BB A_{1/4,1/3}(0)$, also $0\notin B_{6\ep_{m+2}}(w)$. 
Since $X_z \phi_{z,1}  + X_w \phi_{w,m+2}$ and $X_z\phi_{z,1}$ each take values in $[-1,1]$, 
Lemma~\ref{lem-geo-away} applied with $m+2$ in place of $m$ and $w$ in place of $z$ implies that the following is true.
No $D_{\wt h}$-geodesic from 0 to a point $u\in \ol{B_{\ep_n}(z)}$ can enter $B_{4\ep_{m+2}}(w)$ and the same is true with $h + X_z \phi_{z,1}$ in place of $\wt h$. 
Since $\phi_{w,m+2}$ is supported on $B_{4\ep_{m+2}}(w)$, this implies that 
\allb \label{eqn-remove-one-var}
D_{\wt h}\left(0 , B_{\ep_n}(z)  \right) 
= D_{\wt h}\left(0,B_{\ep_n}(z) ; \BB C \setminus \ol{B_{4\ep_{m+2}}(w)} \right) 
&= D_{h + X_z \phi_{z,1}}\left(0,B_{\ep_n}(z ) ; \BB C\setminus \ol{B_{4\ep_{m+2}}(w)} \right) \notag\\ 
&=  D_{h + X_z \phi_{z,1}}\left(0 , B_{\ep_n}(z)  \right)  . 
\alle
 
By Lemma~\ref{lem-dist-interval} (applied with $m=1$) and~\eqref{eqn-remove-one-var}, on $\ol{\mcl E}^{n,m}(z)  $, a.s.\ 
\eqb \label{eqn-dist-2pt-1}
\BB P\left[   G_{z,n}\left( \wt h \right) \, \big| \, h \right] 
= \BB P\left[   G_{z,n}\left( h + X_{z,1} \phi_{z,1} \right) \, \big| \, h \right] 
\preceq \ep_1^{-\xi(Q-\alpha) - o (1)}   \ep_n^\beta
\preceq \ep_n^\beta , 
\eqe
where in the last line we used that $\ep_1$ is a universal constant. 
On the other hand, exactly the same argument as in the proof of Lemma~\ref{lem-dist-interval} shows that on $\ol{\mcl E}^{n,m+2}(w)$, a.s.\ 
\eqb \label{eqn-dist-2pt-2}
\BB P\left[  G_{w,n}\left(\wt h \right)  \, \big| \, h , X_z \right] 
\preceq  \ep_m^{-\xi(Q-\alpha) - o_m(1)} \ep_n^{ \beta}  .
\eqe
Since $G_{z,n}(\wt h)$ is determined by $h$ and $X_z$ (due to~\eqref{eqn-remove-one-var}) we can combine~\eqref{eqn-dist-2pt-1} and~\eqref{eqn-dist-2pt-2} to get~\eqref{eqn-dist-2pt}.
\end{proof}

We now conclude the proof of Lemma~\ref{lem-full-2pt} by comparing the laws of  $h$ and $\wt h$.

\begin{proof}[Proof of Lemma~\ref{lem-full-2pt}] 
As in Lemma~\ref{lem-dist-2pt}, let $X_z$ and $X_w$ be sampled from Lebesgue measure on $ [0,1]$, independently from each other and from $h$, and let $\wt h := h + X_z \phi_{z,1} + X_w \phi_{w,m+2}$.  

By Lemma~\ref{lem-Eupper-2pt}, 
\eqb \label{eqn-use-Eupper-2pt}
\BB P\left[ \ol{\mcl E}^{n,m}(z ) \cap \ol{\mcl E}^{n,m}(w ) \right] 
\leq \ep_m^{-\alpha^2/2 + o_m(1)}  \BB P\left[ \mcl E^n(z) \right] \BB P\left[ \mcl E^n(w) \right]   .
\eqe
By Lemma~\ref{lem-dist-2pt},
\eqb \label{eqn-use-dist-2pt}
\BB P\left[ G_{z,n}(\wt h) \cap G_{w,n}(\wt h) \,|\, \ol{\mcl E}^{m,n}(z ) \cap \ol{\mcl E}^{m,n}(w )  \right]
\leq \ep_m^{-\xi(Q-\alpha) - o_m(1)} \ep_n^{2\beta}  .
\eqe
By combining~\eqref{eqn-use-Eupper-2pt} and~\eqref{eqn-use-dist-2pt}, we obtain
\eqb \label{eqn-add-dist-event}
\BB P\left[ G_{z,n}(\wt h) \cap \ol{\mcl E}^{n,m}(z ) \cap G_{w,n}(\wt h) \cap  \ol{\mcl E}^{n,m}(w )\right] 
\leq \ep_m^{-\alpha^2/2 -\xi(Q-\alpha) - o_m(1)}  \ep_n^{2\beta} \BB P\left[ \mcl E^n(z ) \right] \BB P\left[ \mcl E^n(w ) \right]  .
\eqe
 
For each possible realization of $X_z$ and $X_w$, the function $h - \wt h =- X_z \phi_{z,1} - X_w \phi_{w,m+2}$ takes values in $[-1,1]$ and is supported on $\BB A_{\ep_1,4\ep_1}(z) \cup \BB A_{\ep_{m+2},4\ep_{m+2}}(w)$.  
Since $|z-w| \in [\ep_{m+1},\ep_m]$ and $m\geq 3$, this support is contained in
\eqbn
\left[ \left( \BB C\setminus B_{\ep_2}(z)  \right) \cup \BB A_{\ep_{m+2} , \ep_{m-1}}(z) \right] 
\cap \left[ \left( \BB C\setminus B_{\ep_2}(w)  \right) \cup \BB A_{\ep_{m+2} , \ep_{m-1}}(w) \right] . 
\eqen
Consequently, the definitions of $\ol{\mcl E}^{n,m}(z)$ and $\ol{\mcl E}^{n,m}(w)$ imply that if $\mcl E^n(z;\wt h)  \cap \mcl E^n(w;\wt h)$ occurs, then also $\ol{\mcl E}^{n,m}(z) \cap \ol{\mcl E}^{n,m}(w)$ occurs.
Recalling the definition~\eqref{eqn-Gcap-def} of $\mcl G^n(z)$, we therefore get from~\eqref{eqn-add-dist-event} that
\eqb \label{eqn-shifted-field-2pt}
\BB P\left[\mcl G^n(z;\wt h) \cap \mcl G^n(w;\wt h) \right] 
\leq \ep_m^{-\alpha^2/2 -\xi(Q-\alpha) - o_m(1)}  \ep_n^{2\beta} \BB P\left[ \mcl E^n(z ) \right] \BB P\left[ \mcl E^n(w ) \right]  .
\eqe
 
Recall that the supports of $\phi_{z,1}$ and $\phi_{w,m+2}$ are disjoint, so in particular $(\phi_{z,1} ,\phi_{w,m+2})_\nabla = 0$. 
By Lemma~\ref{lem-gff-abs-cont}, if we condition on $X_z$ and $X_w$, then the conditional law of $\wt h$ is mutually absolutely continuous with respect to the law of $h$. 
The Radon-Nikodym derivative of the latter law w.r.t.\ the law of the former law is
\eqb \label{eqn-2pt-rn}
M  = \exp\left(  - X_z (\wt h,  \phi_{z,1})_\nabla - X_w (\wt h,  \phi_{w , m+2})_\nabla 
+\frac12 X_z^2 (\phi_{z,1},\phi_{z,1})_\nabla + \frac12 X_w^2 (\phi_{w,m+2},\phi_{w,m+2})_\nabla \right) .
\eqe
By condition~\ref{item-onescale-dirichlet} in the definitions of each of $E_{z,1}(\wt h)$ and $E_{w,m+2}(\wt h)$, on $\mcl E^n(z;\wt h) \cap \mcl E^n(w;\wt h)$ we have
\eqbn
(\wt h,  \phi_{z,1})_\nabla \leq K \quad \text{and} \quad (\wt h,\phi_{w,m+2})_\nabla \leq K .
\eqen
Since $(\phi_{z,1} , \phi_{z,1})_\nabla =  (\phi_{w,m+2},\phi_{w,m+2})_\nabla  = (\phi,\phi)_\nabla$ is a universal constant and $X_z,X_w\in[0,1]$, we get that whenever 
$\mcl E^n(z;\wt h) \cap \mcl E^n(w;\wt h)$ occurs, the Radon-Nikodym derivative~\eqref{eqn-2pt-rn} is bounded above and below by deterministic constants which depend only on $\BB s , \alpha,\gamma$. 

We may therefore compute 
\alb
\BB P\left[\mcl G^n(z  ) \cap \mcl G^n(w  ) \right]  
&= \BB P\left[\mcl G^n(z    ) \cap \mcl G^n(w  ) \,|\, X_z,X_w \right] \quad \text{($X_z,X_w, h$ are independent)} \notag\\ 
&= \BB E\left[  M \BB 1_{ \mcl G^n(z ;\wt h ) \cap \mcl G^n(w ;\wt h) } \,|\, X_z, X_w \right]  \notag \\ 
&\preceq \BB P\left[   \mcl G^n(z ;\wt h ) \cap \mcl G^n(w; \wt h )   \,|\, X_z, X_w \right]  .
\ale
Taking unconditional expectations of both sides of this last inequality and using~\eqref{eqn-shifted-field-2pt} now gives~\eqref{eqn-full-2pt}.
\end{proof}

\begin{proof}[Proof of Proposition~\ref{prop-end-2pt}]
By Lemma~\ref{lem-dist-1pt-lower}, we have $ \ep_n^{2\beta} \BB P\left[ \mcl E^n(z ) \right] \BB P\left[ \mcl E^n(w ) \right] \preceq \BB P\left[\mcl G^n(z) \right] \BB P\left[ \mcl G^n(w) \right]$. We now apply this to upper-bound the right side of the estimate of Lemma~\ref{lem-full-2pt} to obtain~\eqref{eqn-end-2pt}.  
\end{proof}

For the proof of Proposition~\ref{prop-end-across}, we need an upper bound for $\BB P[\mcl G^n(z)] $ in terms of $\BB P[\mcl E^n(z)]$ which complements the lower bound of Lemma~\ref{lem-dist-1pt-lower}.  This upper bound comes from the same argument as in the proof of Lemma~\ref{lem-full-2pt}.

\begin{lem} \label{lem-dist-1pt-upper}
We have
\eqb \label{eqn-dist-1pt-upper}
\BB P\left[\mcl G^n(z) \right] \preceq \ep_n^\beta \BB P\left[ \mcl E^n(z) \right] ,
\eqe
with the implicit constant depending only on $\BB s , \alpha,\gamma$.
\end{lem}
\begin{proof}
Let $X$ be sampled uniformly from $[0,1]$ and let $\wt h := h + X\phi_{z,1}$. 
By Lemma~\ref{lem-dist-interval} applied with $m =1$, 
\eqb \label{eqn-dist-1pt-upper-cond}
\BB P\left[G_{z,n}(\wt h) \,|\, h \right] \BB 1_{\ol{\mcl E}^{n,1}(z)} \preceq \ep_n^\beta  \BB 1_{\ol{\mcl E}^{n,1}(z)}  .
\eqe
If $\mcl E^n(z;\wt h)$ occurs, then $\ol{\mcl E}^{n,1}(z)$ occurs. Hence taking unconditional expectations of both sides of~\eqref{eqn-dist-1pt-upper-cond} shows that
\eqb \label{eqn-dist-1pt-upper0}
\BB P\left[ \mcl G^n(z;\wt h) \right] \preceq \ep_n^\beta \BB P\left[ \ol{\mcl E}^{n,1}(z) \right] \leq \ep_n^\beta \BB P\left[ \mcl E_3^n(z) \right] .
\eqe
By Lemma~\ref{lem-multiscale-split}, we have $\BB P\left[ \mcl E_3^n(z) \right] \preceq \BB P\left[\mcl E^n(z)\right]$. 
By a Radon-Nikodym derivative calculation as in the proof of Lemma~\ref{lem-full-2pt}, we also have $\BB P\left[\mcl G^n(z)\right] \preceq \BB P\left[\mcl G^n(z;\wt h)\right]$. 
Plugging these bounds into~\eqref{eqn-dist-1pt-upper0} yields~\eqref{eqn-dist-1pt-upper}. 
\end{proof}

\begin{proof}[Proof of Proposition~\ref{prop-end-across}]
Combine Lemmas~\ref{lem-dist-1pt-lower} and~\ref{lem-dist-1pt-upper}.
\end{proof}

\subsection{Lower bounds for Hausdorff dimension}
\label{sec-dim-lower}

Fix a point $z_0 \in \BB A_{1/3,1/4}(0)$ which lies at Euclidean distance at least $1/100$ from $\bdy \BB A_{1/3,1/4}(0)$. 
For $n\geq 3$, we define
\eqb \label{eqn-Z_n-def}
\mcl Z_n := (\ep_{n-3} \BB Z^2) \cap B_{\ep_4}(z_0) .
\eqe
Also let $\mcl Z_n'$ be the set of $z\in\mcl Z_n$ for which $\mcl G^n(z)$ occurs. 
We define the set of \emph{perfect points}
\eqb \label{eqn-perfect-def}
\mcl P :=  \bigcap_{N \geq 1}  \ol{\bigcup_{n\geq N} \bigcup_{z\in\mcl Z_n'} B_{\ep_{n-3} }(z) } .
\eqe
Equivalently, $\mcl P$ is the set of $u\in \ol{B_{\ep_4}(z_0)}$ for which there is a sequence $n_j\rta\infty$ and points $z_j \in\mcl Z_{n_j}'$ with $z_j \rta u$.

\begin{remark} \label{remark-Z_n-choice}
The reason for the somewhat strange choice of $\mcl Z_n$ in~\eqref{eqn-Z_n-def} is that the statement of the two-point estimate Proposition~\ref{prop-end-2pt} requires $n\geq m+2$ and $m\geq 3$. 
Indeed, we use $\ep_{n-3}$ instead of $\ep_n$ in~\eqref{eqn-Z_n-def} so that $|z-w| \geq \ep_{n-3}/2 \geq \ep_{n-2}$ for each distinct $z,w\in \mcl Z_n$, which implies that $|z-w| \in [\ep_{m+1} , \ep_m]$ for some $m \leq n-2$. 
We use $B_{\ep_4}(z_0)$ instead of $\BB A_{1/3,1/4}(0)$ in~\eqref{eqn-Z_n-def} so that $|z-w| \leq \ep_3$ for each $z,w\in \mcl Z_n$, which implies that if $z\not=w$ then $|z-w| \in [\ep_{m+1},\ep_m]$ for some $m\geq 3$. 
Consequently, for any distinct $z,w\in \mcl Z_n$ we can apply Proposition~\ref{prop-end-2pt} to $z$ and $w$ to get that if $m$ is chosen so that $|z-w| \in [\ep_{m+1},\ep_m]$, then 
\allb
\BB P\left[ \mcl G^n(z) \cap \mcl G^n(w) \right] 
&\leq \ep_m^{-\alpha^2/2 - \xi(Q-\alpha) + o_m(1)} \BB P\left[\mcl G^n(z) \right] \BB P\left[\mcl G^n(w) \right] \notag\\
&= |z-w|^{-\alpha^2/2 - \xi(Q-\alpha) + o_{|z-w|}(1)}  \BB P\left[\mcl G^n(z) \right] \BB P\left[\mcl G^n(w) \right] 
\alle
where the $o_{|z-w|}(1)$ tends to zero as $|z-w| \rta 0$ at a rate depending only on $\BB s , \alpha,\gamma$ (not on $n$ or the particular choice of $z,w\in\mcl Z_n$).  
\end{remark}

Let us now check that the set of perfect points is contained in the set whose dimension we seek to lower-bound. 

\begin{lem}  \label{lem-perfect-contain}
We have
\eqb \label{eqn-perfect-contain}
\mcl P \subset \bdy\mcl B_{\BB s} \cap \mcl T_h^\alpha \cap \wh{\mcl T}_h^\alpha  .
\eqe
\end{lem}
\begin{proof}
Let $w\in \mcl P$ and let $n_j\rta\infty$ and $z_j \in\mcl Z_{n_j}'$ with $z_j \rta w$.
For each $j\in\BB N$, the event $\mcl G^{n_j}(z)$ occurs, so in particular $D_h( 0 , B_{\ep_{n_j}}( z_{n_j} )  ) \in [\BB s , \BB s + \ep_{n_j}^\beta ]$. 
As $j\rta\infty$, the $D_h$-diameter of $B_{\ep_{n_j}}(z_{n_j})$ converges to zero, so $D_h(0,z_j ) \rta \BB s$ and hence $D_h(0,w ) = \BB s$. 
Therefore, a.s.\ $\mcl P \subset \bdy \mcl B_{\BB s}$. 

We next show that a.s.\ $\mcl P \subset \mcl T_h^\alpha$. 
By~\cite[Proposition 2.1]{hmp-thick-pts} and local absolute continuity (to transfer from a zero-boundary GFF to a whole-plane GFF), after possibly replacing the circle average process by a continuous modification, we can arrange that the following is true.
For each $\zeta \in (0,1)$ there a.s.\ exists a random $C > 0$ such that for each $z,z'\in\BB D$ and each $r,r' \in (0,1]$ such that $1/2 \leq r/r' \leq 2$, 
\eqb \label{eqn-circle-avg-cont}
|h_r(z) - h_{r'}(z')| \leq C r^{-1/2} |(z,r) - (z',r')|^{(1-\zeta)/2}   .
\eqe 
 
Now let $w\in\mcl P$ and $\delta \in (0,1)$. By the definition of $\mcl P$, we can find $n\in\BB N$ with $\ep_n < \delta$ and a $z\in \mcl Z_n$ such that $|w - z| \leq \delta^{100}$. 
By~\eqref{eqn-circle-avg-cont}, for each $r\in [\delta , 1]$, 
\eqb \label{eqn-circle-avg-compare}
|h_r(w) - h_r(z )| \leq C r^{-1/2} |w-z|^{(1-\zeta)/2} \leq C r^{-1/2  + (1-\zeta)/200}  = o_r(\log r^{-1}) .
\eqe
By Lemma~\ref{lem-E-thick}, for $z\in\mcl Z_n'$ we have $h_r(z) = (\alpha  + o_r(1)) \log r^{-1}$ for each $r \in [\ep_n ,\ep_0 ]$. 
By combining this with~\eqref{eqn-circle-avg-compare}, we obtain
\eqb
h_r(w) = (\alpha  + o_r(1)) \log r^{-1} ,\quad\forall r \in [\delta ,1] .
\eqe
Sending $\delta \rta 0$ shows that $w\in \mcl T_h^\alpha$. Hence a.s.\ $\mcl P\subset\mcl T_h^\alpha$. 

Finally, we show that a.s.\ $\mcl P \subset \wh{\mcl T}_h^\alpha$. 
To this end, let $w\in\mcl P$ and let $r \in (0,\ep_1]$.
Choose $m\in\BB N$ such that $r\in [\ep_{m+1},\ep_m]$ and let $N \geq m + 3$ be large enough that 
\eqb \label{eqn-perfect-contain-diam}
\sup_{u,v\in B_{\ep_N}(z)} D_h(u,v) \leq  \ep_m^{ \xi(Q-\alpha)} , \quad \forall z\in \BB A_{1/4,1/3}(0) .
\eqe
(such an $N$ exists since $D_h$ induces the Euclidean topology). 

By the definition of $\mcl P$, we can find $n \geq N+1$ and $z\in\mcl Z_n'$ such that $D_h(z,w) \leq \ep_N/2$. 
Since $r\in [\ep_{m+1},\ep_m]$ and $\ep_N <  \ep_{m+2}/100$, we have $B_{\ep_{m+2}}(z) \subset  B_r(w) \subset B_{\ep_{m-1}}(z)$. 
Since $\mcl E^N(z)$ occurs, we can apply Lemma~\ref{lem-E-dist} to get that
\eqb
\sup_{u,v\in B_r(w)} D_h(u,v) 
\geq \sup_{u,v\in \BB A_{\ep_N , \ep_{m+2}}(z)} D_h(u,v) 
\geq \ep_m^{\xi(Q-\alpha) + o_m(1)}
= r^{\xi(Q-\alpha) + o_r(1)} 
\eqe
and
\allb
\sup_{u,v\in B_r(w)} D_h(u,v)
&\leq \sup_{u,v\in B_{\ep_{m-1}}(z)} D_h(u,v) \notag\\
&\leq \sup_{u,v\in \BB A_{\ep_N , \ep_{m-1}}(z)} D_h(u,v)  + \sup_{u,v\in B_{\ep_N  }(z)} D_h(u,v) \notag \\
&\leq \ep_m^{\xi(Q-\alpha) + o_m(1)} + \ep_m^{\xi(Q-\alpha)} \quad \text{(by Lemma~\ref{lem-E-dist} and~\eqref{eqn-perfect-contain-diam})} \notag \\
&= r^{\xi (Q-\alpha) + o_r(1)}  .
\alle 
Therefore, $w\in \wh{\mcl T}_h^\alpha$ and hence $\mcl P\subset\wh{\mcl T}_h^\alpha$. 
\end{proof}

\begin{prop} \label{prop-perfect-dim}
Assume that $\alpha \in (\xi -\sqrt{4-2\xi Q +\xi^2} ,\xi + \sqrt{4-2\xi Q +\xi^2})$. 
For each $\Delta < 2-\alpha^2/2 - \xi(Q-\alpha)$ and each $C>0$, 
\eqb
\BB P\left[ \dim_{\mcl H} \mcl P \geq \Delta ,\, D_h\left(\bdy B_{5/6}(0) , \bdy B_{7/8}(0) \right)   \geq C \right] >  0. 
\eqe
\end{prop}
\begin{proof}
For a Borel measure $\nu$ on $\BB C$ and $\Delta > 0$, we define the \emph{$\Delta$-energy} of $\nu$ by
\eqb
I_\Delta(\nu) = \iint_{\BB C\times\BB C} \frac{1}{|u-v|^\Delta} \,d\nu(u) \, d\nu(v) . 
\eqe
By Frostman's lemma~\cite[Theorem 4.27]{peres-bm}, if $X\subset \BB C$ is a closed set and there is a non-trivial Borel measure $\nu$ on $X$ with $I_\Delta(\nu) < \infty$, then $\dim_{\mcl H} X \geq \Delta$. 
Hence, to prove our lower bound for $\dim_{\mcl H} \mcl P$, we only need to show that for each $\Delta < 2 - \alpha^2/2 - \xi(Q-\alpha)$, it holds with positive probability there is a non-trivial Borel measure $\nu$ on $\mcl P$ such that $I_\Delta(\nu)< \infty$.
We will construct such a measure in the usual manner, as is done, e.g., in~\cite{beffara-dim,hmp-thick-pts,mww-nesting}.
Throughout the proof, all implicit constants and the rates of convergence of all $o(1)$ errors are required to depend only on $C,\BB s,\alpha,\gamma$. 

For $n\in\BB N$, we define a measure on $ B_{\ep_4}(z_0) $ by
\eqb
d\nu_n(u) = \sum_{z\in\mcl Z_n} \frac{\BB 1_{\mcl G^n(z)}}{\BB P[\mcl G^n(z)]} \BB 1_{u\in B_{\ep_{n-3}}(z)} \,du .
\eqe
For $C > 0$, we write $F_C := \left\{D_h\left(\bdy B_{5/6}(0) , \bdy B_{7/8}(0) \right)   \geq C\right\}$, as in~\eqref{eqn-dist-1pt-across}. 
By Proposition~\ref{prop-end-across}, 
\allb \label{eqn-first-moment}
\BB E\left[ \nu_n\left(B_{\ep_4}(z_0) \right)   \BB 1_{F_C}   \right] 
 = \sum_{z\in\mcl Z_n} \frac{ \BB P\left[ \mcl G^n(z) \cap F_C \right]   }{ \BB P\left[\mcl G^n(z) \right]   } \times \pi \ep_{n-3}^2   
 \succeq 1 .
\alle
 
By Propositions~\ref{prop-end-1pt} and~\ref{prop-end-2pt} (see Remark~\ref{remark-Z_n-choice}), 
\allb \label{eqn-second-moment}
\BB E\left[ \nu_n\left( B_{\ep_4}(z_0) \right)^2 \right] 
&= \sum_{z,w\in\mcl Z_n'} \iint_{B_{\ep_{n-3} }(z) \times B_{\ep_{n-3}(w) }}   \frac{\BB P\left[ \mcl G^n(z) \cap \mcl G^n(w) \right] }{\BB P[\mcl G^n(z)] \BB P[\mcl G^n(w)]}     \,du  \, dv\notag \\ 
&\asymp \ep_{n-3}^4 \sum_{z \in \mcl Z_n} \frac{1}{\BB P[\mcl G^n(z)]} 
 + \ep_{n-3}^4 \sum_{\substack{z,w\in \mcl Z_n \\ z\not=w}} \frac{\BB P\left[ \mcl G^n(z) \cap \mcl G^n(w) \right]}{\BB P\left[\mcl G^n(z) \right] \BB P\left[\mcl G^n(w)\right]} \notag\\
&\preceq  \ep_{n-3}^4 \sum_{z \in \mcl Z_n} \ep_n^{-\alpha^2/2 - \beta + o_n(1)} 
 + \ep_{n-3}^4 \sum_{\substack{z,w\in \mcl Z_n \\ z\not=w}} |z-w|^{-\alpha^2/2 - \xi(Q-\alpha) - o_{|z-w|}(1)}  \notag \\
&\preceq \ep_{n-3}^4 \sum_{z \in \mcl Z_n} \ep_n^{-\alpha^2/2 - \beta + o_n(1)} 
 +  \iint_{\BB D \times \BB D}  |u-v|^{-\alpha^2/2 - \xi(Q-\alpha) - o_{|u-v|}(1)} \, du\, dv .
\alle 
Since $\ep_{n-3} = \ep_n^{1+o_n(1)}$ and $\alpha\in (\xi -\sqrt{4-2\xi Q +\xi^2} ,\xi + \sqrt{4-2\xi Q +\xi^2})$, we have $\alpha^2/2 + \beta   < \alpha^2/2  + \xi(Q-\alpha)  < 2$. Hence the right side of~\eqref{eqn-second-moment} is bounded above by a constant depending only on $\BB s,\alpha,\gamma$. 

For $\Delta < 2 - \alpha^2/2  -  \xi(Q-\alpha)$, we can again use Propositions~\ref{prop-end-1pt} and~\ref{prop-end-2pt} to obtain that the expected $\Delta$-energy of $\nu_n$ satisfies
\allb \label{eqn-energy}
\BB E\left[ I_\Delta(\nu_n) \right] 
&= \sum_{z,w \in \mcl Z_n}  \frac{\BB P\left[ \mcl G^n(z) \cap \mcl G^n(w) \right]}{\BB P\left[\mcl G^n(z) \right] \BB P\left[\mcl G^n(w)\right]} \iint_{B_{\ep_{n-3}}(z) \times B_{\ep_{n-3}}(w)} \frac{1}{|u-v|^\Delta} \, du \, dv \notag\\
&= \sum_{z\in\mcl Z_n}  \frac{1}{\BB P\left[\mcl G^n(z) \right]  } \iint_{B_{\ep_{n-3}}(z) \times B_{\ep_{n-3}}(z)} \frac{1}{|u-v|^\Delta} \, du \, dv \notag\\
&\qquad +  \sum_{\substack{z,w \in \mcl Z_n \\ z\not= w}}  \frac{\BB P\left[ \mcl G^n(z) \cap \mcl G^n(w) \right]}{\BB P\left[\mcl G^n(z) \right] \BB P\left[\mcl G^n(w)\right]} \iint_{B_{\ep_{n-3}}(z) \times B_{\ep_{n-3}}(w)} \frac{1}{|u-v|^\Delta} \, du \, dv \notag\\ 
&\preceq \sum_{z\in\mcl Z_n} \ep_n^{-\alpha^2/2 - \beta - o_n(1) } \times \ep_{n-3}^{4-\Delta}    \notag\\
&\qquad +  \sum_{\substack{z,w \in \mcl Z_n \\ z\not= w}} |z-w|^{-\alpha^2/2 - \xi(Q-\alpha) - o_{|z-w|}(1)} \times \frac{\ep_{n-3}^4}{|z-w|^\Delta} \notag\\ 
&\preceq   \ep_n^{2 -\alpha^2/2 - \beta - \Delta - o_n(1)}    
 +  \ep_{n-3}^4 \sum_{\substack{z,w \in \mcl Z_n \\ z\not= w}} |z-w|^{-\Delta -\alpha^2/2 - \xi(Q-\alpha) - o_{|z-w|}(1)}  \notag\\ 
&\preceq   \ep_n^{2 -\alpha^2/2 - \beta - \Delta - o_n(1)}    
 +  \iint_{\BB D \times \BB D} |u-v|^{-\Delta -\alpha^2/2 - \xi(Q-\alpha) - o_{|u-v|}(1)} \, du\,dv   . 
\alle
Since $\alpha^2/2 + \beta + \Delta  < \alpha^2/2  + \xi(Q-\alpha)   + \Delta < 2$, this last quantity is bounded above by a constant depending only on $\BB s,\alpha,\gamma$. 
 
Due to~\eqref{eqn-first-moment},~\eqref{eqn-second-moment}, and~\eqref{eqn-energy}, we can now conclude via the usual argument (see, e.g., the proof of~\cite[Proposition 4.8]{mww-nesting}) that with positive probability, $F_C$ occurs and there is a weak subsequential limit $\nu$ of the $\nu_n$'s which is non-trivial and has finite $\Delta$-energy. 
In order to make our paper more self-contained, we will give some details of this argument.

By~\eqref{eqn-first-moment} the mean of the random variable $\BB 1_{F_C} \nu_n(B_{\ep_4}(z_0))$ is bounded below independently of $n$ and by~\eqref{eqn-second-moment} its variance is bounded above independently of $n$. 
By the Chebyshev inequality and the Paley-Zygmund inequality there is a deterministic constant $C > 1 $ such that for every $n\in\BB N$, 
\eqbn
\BB P\left[ C^{-1} \leq \BB 1_{F_C} \nu_n(B_{\ep_4}(z_0)) \leq C \right] \geq 1/C . 
\eqen
By~\eqref{eqn-energy}, we have $\BB E\left[ I_\Delta(\nu_n) \right] \preceq 1$ so we can apply Markov's inequality to get that after possibly increasing $C$,
\eqb
 \BB P\left[ \Gamma_n \right] \geq 1/C \quad \text{where} \quad \Gamma_n := \left\{ C^{-1} \leq \BB 1_{F_C} \nu_n(B_{\ep_4}(z_0)) \leq C ,\, I_\Delta(\nu_n) \leq C \right\} .
\eqe
Hence with positive probability, $\Gamma_n$ occurs for infinitely many values of $n$. Let $\Gamma$ be the event that this is the case.
On $\Gamma$, the event $F_C$ necessarily occurs.
Furthermore, by Prokhorov's theorem, on $\Gamma$ there is a (random) subsequence $(\nu_{n_k})$ of the measures $\nu_n$ which converges weakly to a non-trivial limiting measure $\nu$. 
By the Portmanteau theorem, $\int f \, d\mu \leq \liminf_{k\rta\infty} \int f d\mu_k$ whenever $\mu_k\rta \mu$ weakly and $f$ is a lower semicontinuous function which is bounded below.
Taking $f(z,w) = |z-w|^{-\Delta}$ (note that this function is bounded below on $\BB D\times \BB D$) and $\mu_k = \nu_{n_k } \otimes \nu_{n_k}$ shows that on $\Gamma$, the subsequential limiting measure $\nu$ has finite $\Delta$-energy. 
 
Since each $\nu_n$ is supported on the set $\bigcup_{z\in\mcl Z_n'} B_{\ep_{n-3}}(z)  $, it follows that each point of the support of $\nu$ is a limit as $k\rta\infty$ of points $z_{n_k} \in \mcl Z_{n_k}'$. By the definition~\eqref{eqn-perfect-def} of $\mcl P$, this means that $\nu$ is supported on $\mcl P$. 
As explained at the beginning of the proof, this shows via Frostman's lemma that with positive probability, $\dim_{\mcl H} \mcl P \geq \Delta$ and $F_C$ occurs, i.e., $D_h\left(\bdy B_{5/6}(0) , \bdy B_{7/8}(0) \right)   \geq C$.
\end{proof}

\begin{prop} \label{prop-dim-contain} 
For each $\BB s  >0$ and each $\Delta < 2-\xi(Q-\alpha) -\alpha^2/2$, it holds with positive probability that 
\eqb
\dim_{\mcl H}^0\left( \bdy\mcl B_{\BB s} \cap \mcl T_h^\alpha \cap \wh{\mcl T}_h^\alpha \right) \geq \Delta \quad \text{and} \quad \mcl B_{\BB s} \subset B_{7/8}(0) .
\eqe
\end{prop}
\begin{proof}
We first consider the case when $\BB s > S$, with $S$ as in Lemma~\ref{lem-E-dist}. In this case, Lemma~\ref{lem-perfect-contain} and Proposition~\ref{prop-perfect-dim} (applied with $C=\BB s$) imply that with positive probability,
\eqbn
\dim_{\mcl H}^0\left( \bdy\mcl B_{\BB s} \cap \mcl T_h^\alpha \cap \wh{\mcl T}_h^\alpha \right) \geq \Delta \quad \text{and} \quad D_{ h}(\bdy B_{5/6}(0) , \bdy B_{7/8}(0) ) \geq \BB s .
\eqen
If $D_h(\bdy B_{5/6}(0) , \bdy B_{7/8}(0) ) \geq \BB s$, then also $D_h(0, \bdy B_{7/8}(0) ) \geq \BB s$ and hence $\mcl B_{\BB s} \subset B_{7/8}(0)$. Hence the lemma holds when $\BB s > S$. 

Now assume that $\BB s \in [0,S]$ and let $\wt h := h+\xi^{-1} \log(\BB s/(2S))$. By Weyl scaling (Axiom~\ref{item-metric-f}), $\mcl B_{2S}(0, D_{\wt h}) = \mcl B_{\BB s}$. Clearly, the definitions~\eqref{eqn-thick-def} and~\eqref{eqn-metric-thick} of $\mcl T_h^\alpha$ and $\wh{\mcl T}_h^\alpha$ are unaffected by adding a constant to $h$, so $\mcl T_h^\alpha  = \mcl T_{\wt h}^{\alpha}$ and $\wh{\mcl T}_h^\alpha  = \wh{\mcl T}_{\wt h}^{\alpha}$  Hence,
\allb \label{eqn-dim-field-compare}
&\left\{\dim_{\mcl H}^0\left( \bdy\mcl B_{\BB s} \cap \mcl T_h^\alpha \cap \wh{\mcl T}_h^\alpha \right) \geq \Delta \right\} \cap \left\{\mcl B_{\BB s} \subset B_{7/8}(0) \right\} \notag\\
&\qquad\qquad= \left\{\dim_{\mcl H}^0\left( \bdy\mcl B_{2S}(0;D_{\wt h}) \cap \mcl T_{\wt h}^\alpha \cap \wh{\mcl T}_{\wt h}^\alpha \right) \geq \Delta \right\} \cap \left\{\mcl B_{2S}(0;D_{\wt h}) \subset B_{7/8}(0) \right\} .
\alle
By locality (Axiom~\ref{item-metric-f}), the event $\{\dim_{\mcl H}^0\left( \bdy\mcl B_{2S} \cap \mcl T_h^\alpha \cap \wh{\mcl T}_h^\alpha \right) \geq \Delta\} \cap \{\mcl B_{2S} \subset B_{7/8}(0)\}$ is determined by $h|_{B_{7/8}(0)}$. By Lemma~\ref{lem-gff-abs-cont} (applied to a smooth bump function which equals $\xi^{-1} \log(\BB s/(2S)$ on $B_{7/8}(0)$ and vanishes outside of $\BB D$), the laws of the restrictions of $h$ and $\wt h$ to $B_{7/8}(0)$ are mutually absolutely continuous. Note that it is important here that we work with $B_{7/8}(0)$ instead of $\BB D$, since the circle averages of $h$ and $\wt h$ over $\bdy\BB D$ are different. 
Hence, the version of the lemma with $2S$ in place of $\BB s$ implies that the event in~\eqref{eqn-dim-field-compare} has positive probability, i.e., the lemma statement holds for $\BB s$. 
\end{proof}

\begin{proof}[Proof of Theorem~\ref{thm-thick-dim}, lower bound]
The lower bound for $\esssup \dim_{\mcl H}^0\left( \bdy\mcl B_{\BB s} \cap \mcl T_h^\alpha \cap \wh{\mcl T}_h^\alpha \right)$ is immediate from Proposition~\ref{prop-dim-contain}, which is a strictly stronger statement. 
The lower bound for\\ 
$\esssup \dim_{\mcl H}^\gamma\left( \bdy\mcl B_{\BB s} \cap \mcl T_h^\alpha \cap \wh{\mcl T}_h^\alpha \right)$ follows from the lower bound for the Euclidean dimension together with~\cite[Proposition 2.5]{gp-kpz} which says that a.s.\ for every Borel set $X\subset\BB C$ simultaneously, we have $\dim_{\mcl H}^\gamma (X\cap \mcl T_h^\alpha) \geq \frac{1}{\xi(Q-\alpha)} \dim_{\mcl H}^0(X\cap\mcl T_h^\alpha)$ (here we apply the proposition with $X = \bdy\mcl B_{\BB s} \cap\wh{\mcl T}_h^\alpha$). 
\end{proof}

\begin{proof}[Proof of Theorem~\ref{thm-bdy-dim}, lower bound]
Taking $\alpha = \xi$ in Theorem~\ref{thm-thick-dim} shows that $\esssup \dim_{\mcl H}^0 \bdy\mcl B_{\BB s} \geq 2 - \xi Q +\xi^2/2$. 
Taking $\alpha=\gamma$ in Theorem~\ref{thm-thick-dim} shows that $\esssup \dim_{\mcl H}^\gamma \bdy \mcl B_{\BB s} \geq d_\gamma-1$. 
\end{proof}

\section{One-point estimate for the event at a single scale}
\label{sec-onescale-prob}

In this section we will prove Proposition~\ref{prop-E-prob}. In Sections~\ref{sec-around-pos}, \ref{sec-internal-diam}, and~\ref{sec-rn-sup} we prove general estimates for the GFF which will help us lower-bound the probabilities of conditions~\ref{item-onescale-around}, \ref{item-onescale-diam}, and~\ref{item-onescale-rn} in the definition of $E_{z,j}$, respectively.
In Section~\ref{sec-E-prob} we conclude the proof of Proposition~\ref{prop-E-prob}.

\subsection{Comparison of distances in annuli with positive probability}
\label{sec-around-pos}

In this brief subsection we prove an estimate which allows us to lower-bound the probability of condition~\ref{item-onescale-around} in the definition of $E_{z,j}$. 
We remark that this is the only condition in the definition of $E_{z,j}$ which occurs with uniformly positive probability, but not probability close to 1. 

\begin{lem} \label{lem-around-pos}
Fix $0 < a < b < c <\infty$ and $\delta>0$. With positive probability, there is a path in $\BB A_{b,c}(0)$ which disconnects the inner and outer boundaries of $\BB A_{b,c}(0)$ and whose $D_h$-length is at most $\delta D_h\left( \bdy B_{a}(0) , \bdy B_{b}(0) \right)$.
\end{lem}
\begin{proof}
Fix $r_1,r_2 \in (b,c)$ with $r_1 < r_2$. 
For a constant $C > 1$, let $E_C$ be the event that $D_h(\bdy B_a(0) , \bdy B_b(0)) \geq C^{-1}$ and there is a path in $\BB A_{r_1,r_2}(0)$ which disconnects the inner and outer boundaries of $\BB A_{r_1,r_2}(0)$ with $D_h$-length at most $C$. 
Since $D_h$ induces the Euclidean topology, there is some $C>1$ such that $\BB P[E_C] > 0$.
Let $\phi$ be a smooth bump function which is identically equal to 1 on $\BB A_{r_1,r_2}(0)$ and is supported on $\BB A_{b,c }(0)$ and let
\eqb
\wt h := h + \xi^{-1} \log(\delta/C^2) .
\eqe

By Axiom~\ref{item-metric-f}, if $E_C$ occurs then there is a path in $\BB A_{r_1,r_2}(0)$ which disconnects the inner and outer boundaries of $\BB A_{r_1,r_2}(0)$ (and hence also those of $\BB A_{b,c}(0)$) whose $D_{\wt h}$-length is at most $\delta C^{-1}$.
Furthermore, since $\phi$ vanishes on $\BB A_{a,b}(0)$, we have $D_{\wt h} (\bdy B_{a}(0) , \bdy B_b(0))  = D_h(\bdy B_{a}(0) , \bdy B_b(0))  \geq C^{-1}$. 
Consequently, it holds with positive probability that the event in the lemma statement occurs with $\wt h$ in place of $h$.
Since the laws of $h$ and $\wt h$ are mutually absolutely continuous, we now obtain the lemma statement. 
\end{proof}

\subsection{Bounds for internal diameters of annuli}
\label{sec-internal-diam}

In this subsection we prove an upper bound for the diameter of an annulus $\BB A_{a,1}(0)$ with respect to the internal metric $D_h(\cdot,\cdot ; \BB A_{a,1}(0))$. 
That is, we control internal distances all the way up to the boundary of $\BB A_{a,1}(0)$.
The estimate is needed to lower-bound the probabilities of condition~\ref{item-onescale-diam} in the definition of $E_{z,j}$ and the events $H_{z,j}^{\op{in}}$ and $H_{z,j}^{\op{out}}$ of~\eqref{eqn-H-def}. 
Note that the estimates of this section are not immediate from the results in~\cite{lqg-metric-estimates} which only control either (a) the $D_h(\cdot,\cdot;U)$-diameter of $K$, where $K\subset U$ is a compact subset; or (b) the $D_h(\cdot,\cdot;S)$-diameter of a square $S$.
We will, however, extract our bound for the internal $D_h$-diameter of an annulus from the bounds for the internal $D_h$-diameter of a square from~\cite{lqg-metric-estimates}.

\begin{lem} \label{lem-internal-diam}
There is a constant $\eta  = \eta(\gamma) > 0$ such that for each $0 < a < 1$, 
\eqb \label{eqn-internal-diam} 
\BB E\left[ \left( \sup_{u,v\in\BB A_{a,1}(0)} D_h\left(  u,v ; \BB A_{a,1}(0) \right) \right)^\eta \right] < \infty .
\eqe
\end{lem}

We expect that Lemma~\ref{lem-internal-diam} is true with any $\eta < 4d_\gamma/\gamma^2$, as in the moment estimates for diameters from~\cite{lqg-metric-estimates}.
However, our goal here is to quickly establish an estimate which is sufficient for our purposes, rather than to establish an optimal estimate.

\begin{proof}[Proof of Lemma~\ref{lem-internal-diam}]
Fix $\chi \in (0,\xi(Q-2))$, chosen in a manner depending only on $\gamma$. 
For $\ep > 0$ and $k\in\BB N_0$, let $\mcl S_k^\ep$ be the set of closed $2^{-k}\ep \times 2^{-k}\ep$ squares with corners in $2^{-k}\ep \BB Z^2$ which are contained in $\BB A_{a,1}(0)$. 
By~\cite[Lemma 3.20]{lqg-metric-estimates}, there is a constant $\eta_0  =\eta_0(\gamma) > 0$ such that with probability $1-O_\ep(\ep^{\eta_0})$,  
\eqb \label{eqn-use-holder-cont}
\sup_{u,v\in S} D_h(u,v ; S) \leq (2^{-k} \ep)^\chi ,\quad\forall k\in\BB N_0 ,\quad\forall S\in\mcl S_k^\ep .
\eqe

Let $\wt{\mcl S}_0^\ep = \mcl S_0^\ep$ and for $k\in\BB N$, let $\wt{\mcl S}_k^\ep$ be the set of squares $S\in\mcl S_k^\ep$ which are maximal in the sense that there is no square of $\mcl S_{k-1}^\ep$ containing $S$ (i.e., the $2^{-k+1}\ep \times 2^{-k+1}\ep$ square with corners in $2^{-k+1}\ep \BB Z^2$ which contains $S$ is not contained in $\BB A_{a,1}(0)$). 

If $\ep < (1-a)/4$, then the middle circle $\bdy B_{(a+1)/2}(0) \subset \BB A_{a,1}(0)$ is contained in the union of at most $O_\ep(\ep^{-1})$ squares of $\wt{\mcl S}_0^\ep$. 
Therefore,~\eqref{eqn-use-holder-cont} implies that
\eqb \label{eqn-circle-diam}
\sup_{u,v \in \bdy B_{(a+1)/2}(0)} D_h\left(u,v ; \BB A_{a,1}(0) \right)  \preceq \ep^{-1+\chi}  ,
\eqe
with a deterministic implicit constant which does not depend on $\ep$. 

For each $z \in \BB A_{a,1}(0)$, the radial line segment from $z$ to $\bdy B_{(a+1)/2}(0)$ intersects at most a constant-order number of squares of $\wt{\mcl S}_k^\ep$ for each $k \in \BB N$.
For $k=0$, we bound the number of squares of $\wt{\mcl S}_0^\ep$ intersected by this line segment simply by $\#\wt{\mcl S}_0^\ep \preceq  \ep^{-2}$. 
Consequently,~\eqref{eqn-use-holder-cont} implies that
\allb \label{eqn-radial-diam}
&\sup_{z \in \BB A_{a,1}(0) } D_h\left(z ,   \bdy B_{(a+1)/2}(0)  ; \BB A_{a,1}(0) \right)   \notag\\
&\qquad \qquad \preceq \sum_{k=1 }^{\infty}  (2^{-k} \ep)^\chi 
+  \ep^{-2} 
\preceq \ep^\chi + \ep^{-2} \preceq \ep^{-2} .
\alle

By~\eqref{eqn-circle-diam} and~\eqref{eqn-radial-diam}, if~\eqref{eqn-use-holder-cont} holds then $\sup_{u,v\in\BB A_{a,1}(0)} D_h\left(  u,v ; \BB A_{a,1}(0) \right) $ is bounded above by a deterministic constant times $\ep^{-2}$.
Hence,
\eqb
\BB P\left[  \sup_{u,v\in\BB A_{a,1}(0)} D_h\left(  u,v ; \BB A_{a,1}(0) \right)  > \ep^{-2} \right]  = O_\ep(\ep^{ \eta_0})
\eqe
which implies~\eqref{eqn-internal-diam} for any $\eta < \eta_0/2$.
\end{proof}

In Section~\ref{sec-E-prob}, we will use the following variant of Lemma~\ref{lem-internal-diam}. 

\begin{lem} \label{lem-meanzero-diam}
There is a constant $\eta  = \eta(\gamma) > 0$ such that for each $0 < a < 1$, each $z\in\BB C$, and each $r >0$, 
\eqb \label{eqn-meanzero-diam} 
\BB E\left[ \left( r^{-\xi Q} \sup_{u,v\in \BB A_{a r ,r}(z)} D_{h - h_{|\cdot-z|}(z)} \left(  u,v ; \BB A_{a r,r}(z) \right) \right)^\eta \right] \preceq 1 
\eqe
with the implicit constant depending only on $a,\gamma$. 
\end{lem}
\begin{proof}
The law of the field $h-h_{|\cdot-z|}(z)$ is exactly scale invariant and translation invariant (not just scale invariant modulo additive constant) in the the sense that the law of $h(r\cdot) - h_{|r\cdot-z|}(z)$ does not depend on $r$ or $z$. 
By this and the LQG coordinate change formula (Axiom~\ref{item-metric-f}), we see that the law of 
\eqb
r^{-\xi Q} \sup_{u,v\in \BB A_{a r ,r}(z)} D_{h - h_{|\cdot-z|}(z)} \left(  u,v ; \BB A_{a r ,r}(z) \right)
\eqe
depends only on $a$ and $\gamma$. Hence we only need to prove the lemma in the case when $z=0$ and $r=1$. 
In this case, we can use Weyl scaling (Axiom~\ref{item-metric-f}) to obtain
\eqb \label{eqn-meanzero-diam-split}
\sup_{u,v\in \BB A_{a   ,1}(0)} D_{h - h_{|\cdot|}(0)} \left(  u,v ; \BB A_{a ,1}(0) \right)
\leq \left(\sup_{s\in [a,1]} e^{-\xi h_s(0)} \right) \left( \sup_{u,v\in \BB A_{a   ,1}(0)} D_{h  } \left(  u,v ; \BB A_{a ,1}(0) \right) \right) .
\eqe
Since $t\mapsto h_{e^{-t}}(0)$ is a standard linear Brownian motion, the first factor on the right in~\eqref{eqn-meanzero-diam-split} has finite moments of all positive orders.
By Lemma~\ref{lem-internal-diam}, the second factor has a finite moment of some positive order.
We now conclude by means of H\"older's inequality.
\end{proof}

\subsection{Uniform bounds for Radon-Nikodym derivatives}
\label{sec-rn-sup}

In this subsection we prove a bound for the Radon-Nikodym derivative between a whole-plane GFF and a zero-boundary GFF plus a harmonic function which holds \emph{simultaneously} for a certain class of harmonic functions.
This estimate is needed to say that condition~\ref{item-onescale-rn} in the definition of $E_{z,j}$ occurs with high probability.

\begin{lem} \label{lem-rn-sup-plane}
Fix $A > 0$ and $0 < r < 1$. Let $\frk F = \frk F(A , r )$ the set of harmonic functions $\frk f : \BB D\rta \BB R$ which satisfy $\frk f(0) = 0$ and $\sup_{u\in \bdy B_r(0)} |\frk f(u)| \leq A$. 
Let $h$ be a whole-plane GFF normalized so that $h_1(0) = 0$, let $\rng h$ be a zero-boundary GFF on $\BB D$, and for $\frk f \in \frk F $ and $s > 0$, let $ M_{\frk f}^s =  M_{\frk f}^s(h|_{B_s(0)})$ be the Radon-Nikodym derivative of law of $(\rng h + \frk f)|_{B_s(0)}$ w.r.t.\ the law of $h|_{B_s(0)}$.
For each $s_0 \in (0,r)$ and each $p\in (0,1)$, there exists $L = L(A,r,p,s_0) > 1$ such that 
\eqb \label{eqn-rn-sup-plane}
\BB P\left[  \inf_{\frk f \in \frk F}  M_{\frk f}^s \geq L^{-1}  \: \text{and} \: \sup_{\frk f \in \frk F }  M_{\frk f}^s \leq L \right] \geq p,\quad\forall s \in (0,s_0].
\eqe 
\end{lem}

The main step in the proof of Lemma~\ref{lem-rn-sup-plane} is a variant where we compare to a zero-boundary GFF instead of a whole-plane GFF. 

\begin{lem} \label{lem-rn-sup}
Fix $A > 0$ and $0 < r < 1$ and let $\frk F$ be as in Lemma~\ref{lem-rn-sup-plane}. 
Let $\rng h$ be a zero-boundary GFF on $\BB D$ and for $\frk f \in \frk F $ and $s > 0$, let $\rng M_{\frk f}^s = \rng M_{\frk f}^s(\rng h|_{B_s(0)})$ be the Radon-Nikodym derivative of law of $(\rng h + \frk f)|_{B_s(0)}$ w.r.t.\ the law of $\rng h|_{B_s(0)}$.
For each $s_0 \in (0,r)$ and each $p\in (0,1)$, there exists $L = L(A,r,p,s_0) > 1$ such that 
\eqb \label{eqn-rn-sup}
\BB P\left[  \inf_{\frk f \in \frk F} \rng M_{\frk f}^s \geq L^{-1}  \: \text{and} \: \sup_{\frk f \in \frk F } \rng M_{\frk f}^s \leq L \right] \geq p,\quad\forall s \in (0,s_0].
\eqe 
\end{lem}

The idea of the proof of Lemma~\ref{lem-rn-sup} is that $\rng M_{\frk f}^s$ can be expressed in terms of $(\rng h,\psi\frk f)_\nabla$ where $\psi$ is a smooth bump function which is identically equal to 1 on $B_s(0)$. An elementary estimate for harmonic functions (Lemma~\ref{lem-rn-sup} below) allows us to bound all of the partial derivatives of $\psi\frk f$, of all orders, for $\frk f \in\frk F$ in terms of $A$. Since $h$ is a distribution, the mapping $f\mapsto (h,f)_\nabla$ is continuous with respect to the topology of uniform convergence of partial derivatives of all orders. This gives the needed uniform bound for $(\rng h,\psi\frk f)_\nabla$.

\begin{lem} \label{lem-harmonic-sup}
Let $0 < s < r$ and let $\frk f$ be a harmonic function on a neighborhood of $B_r(0)$ such that $\frk f(0)= 0$.
For each $m\in\BB N$, there is a constant $C_m = C_m(s,r) > 0$ such that for each multi-index $\BB k \in \{1,2\}^m$, 
\eqb \label{eqn-harmonic-sup}
\sup_{u \in B_s(0)} |\partial_{\BB k} \frk f(u)| \leq  C_m \sup_{u\in \bdy B_r(0)} |\frk f(u)| .
\eqe
\end{lem}
\begin{proof}
The case when $m =1$ is a standard estimate for harmonic functions, and can be easily deduced from the mean value property. 
In general, assume $m\geq 2$ and the needed estimate has been proven with $m-1$ in place of $m$. 
Let $s' := (r+s)/2$ and for $\BB k\in \{1,2\}^m$, let $\BB k' \in \{1,2\}^{m-1}$ be obtained by deleting the $m$th component. 
Note that $\partial_{\BB k'}\frk f$ is harmonic. 
We now apply~\eqref{eqn-harmonic-sup} with $m=1$ and $s'$ in place of $r$; followed by~\eqref{eqn-harmonic-sup} with $m-1$ in place of $m$, $s'$ in place of $s$, and $\partial_{\BB k'}\frk f$ in place of $\frk f$. This gives
\eqb
\sup_{u \in B_s(0)} |\partial_{\BB k} \frk f(u)|
\leq C_1(s,s') \sup_{u \in \bdy B_{s'}(0)} |\partial_{\BB k'} \frk f(u)| 
\leq  C_1(s,s') C_{m-1}(s',r) \sup_{u\in \bdy B_r(0)} |\frk f(u)| .
\eqe
Thus~\eqref{eqn-harmonic-sup} holds with $C_m(s,r) =  C_1(s,s') C_{m-1}(s',r)$. 
\end{proof}

\begin{proof}[Proof of Lemma~\ref{lem-rn-sup}]
Fix $s_0 \in (0,r)$. 
We first write down a formula for $\rng M_{\frk f}^s$ which is valid for $s \in (0,s_0]$ 
Let $\psi : \BB C \rta [0,1]$ be a smooth bump function which is identically equal to 1 on $B_{s_0}(0)$ and which vanishes outside of $B_{(r+s_0)/2}(0)$. 
By the zero-boundary GFF analog of Lemma~\ref{lem-gff-abs-cont}, for each fixed $\frk f \in \frk F $, the law of $ \rng h + \psi\frk f$ is absolutely continuous with respect to the law of $\rng h$, with Radon-Nikodym derivative
\eqb \label{eqn-zero-bdy-rn0}
 \wt M_{\frk f}  =   \exp\left( (\rng h , \psi \frk f )_\nabla - \frac12 \left(\psi \frk f , \psi \frk f   \right)_\nabla \right) .
\eqe  
For $s \in (0,s_0]$, we have $(\rng h + \frk f)|_{B_s(0)} = (\rng h + \psi \frk f)|_{B_{s}(0)} $, so we can take the conditional expectation given $\rng h|_{B_{s}(0)}$ to get that
\eqb \label{eqn-zero-bdy-rn}
\rng M_{\frk f}^s = \BB E\left[ \wt M_{\frk f}  \,\big|\, \rng h|_{B_{s}(0)} \right] .
\eqe

By Lemma~\ref{lem-harmonic-sup} (applied with $s = (r+s_0)/2$) and the product rule (recall that $\psi$ is supported on $B_{(r+s_0)/2}(0)$), for each $m\in\BB N$ there is a constant $C_m = C_m(r,s_0,\psi)$ such that for each $\frk f \in \frk F$, each $m\in\BB N$, and each multi-index $\BB k \in \{1,2\}^m$, 
\eqb
\sup_{u\in B_s(0)} | \partial_{\BB k} (\psi \frk f) (u)  | \leq C_m \sup_{u\in \bdy B_r(0)} |\frk f(u)|   \leq C_m A .
\eqe
In particular, 
\eqb \label{eqn-rn-deriv-sup}
\sup_{\frk f \in \frk F} \sup_{\BB k \in \{1,2\}^m} \sup_{u\in B_s(0)} | \partial_{\BB k} (\psi \frk f) (u)  | < \infty,  \quad \forall m \in \BB N 
\eqe
and
\eqb \label{eqn-rn-dirichlet-sup}
\sup_{\frk f \in \frk F} \left( \psi \frk f  , \psi \frk f    \right)_\nabla < \infty .
\eqe
Since $\rng h$ is a random distribution, i.e., a continuous linear functional from the space of smooth compactly supported functions on $\BB D$ into $\BB R$, it follows from~\eqref{eqn-rn-deriv-sup} that a.s.\
\eqb \label{eqn-rn-dist-sup}
\sup_{\frk f \in \frk F} (\rng h , \psi \frk f )_\nabla < \infty .
\eqe

By plugging~\eqref{eqn-rn-dirichlet-sup} and~\eqref{eqn-rn-dist-sup} into~\eqref{eqn-zero-bdy-rn0}, we get that a.s.\ $\inf_{\frk f \in \frk F} \wt M_{\frk f} > 0$, which by~\eqref{eqn-zero-bdy-rn} implies that a.s.\
\eqb \label{eqn-rn-sup-lower}
\inf_{\frk f\in \frk F} \rng M_{\frk f}^s 
= \inf_{\frk f\in \frk F} \BB E\left[\wt M_{\frk f} \,\big|\, \rng h|_{B_{s}(0)} \right]
\geq \BB E\left[ \inf_{\frk f\in \frk F} \wt M_{\frk f} \,\big|\, \rng h|_{B_{s}(0)} \right] .
\eqe
The right side of~\eqref{eqn-rn-sup-lower} is an a.s.\ positive random variable which depends only on $A,r,s_0$. 
This leads to the lower bound in~\eqref{eqn-rn-sup}. 
 
We now upper-bound $\sup_{\frk f\in\frk F} \rng M_{\frk f}^s$. The random variables $(\rng h , \psi \frk f)_\nabla$ are jointly centered Gaussian with variances $(\psi \frk f , \psi \frk f)_\nabla$. 
By~\eqref{eqn-rn-dirichlet-sup}, \eqref{eqn-rn-dist-sup}, and the Borell-TIS inequality~\cite{borell-tis1,borell-tis2} (see, e.g.,~\cite[Theorem 2.1.1]{adler-taylor-fields}), $\sup_{\frk f \in \frk F} (\rng h , \psi \frk f )_\nabla$ has a Gaussian tail, so in particular the exponential of this supremum has finite expectation.  
By~\eqref{eqn-zero-bdy-rn0} and~\eqref{eqn-zero-bdy-rn}, we now obtain that a.s.\ 
\eqb \label{eqn-rn-sup-upper}
\sup_{\frk f \in \frk F} \rng M_{\frk f}^s
\leq \sup_{\frk f \in \frk F} \BB E\left[  \wt M_{\frk f} \,\big|\, \rng h|_{B_{s}(0)} \right] 
\leq \BB E\left[  \sup_{\frk f \in \frk F} \ \wt M_{\frk f} \,\big|\, \rng h|_{B_{s}(0)} \right] ,
\eqe
which is an a.s.\ finite random variable depending only on $A,r,s_0$.
\end{proof}

\begin{proof}[Proof of Lemma~\ref{lem-rn-sup-plane}]
By the Markov property of $h$ and the zero-boundary GFF analog of Lemma~\ref{lem-gff-abs-cont}, we know that the law of $\rng h|_{B_s(0)}$ is absolutely continuous w.r.t.\ the law of $h|_{B_s(0)}$, and the Radon-Nikodym derivative is a.s.\ finite and positive.
The Radon-Nikodym derivative of the law of $(\rng h + \frk f)|_{B_s(0)}$ w.r.t.\ the law of $h|_{B_s(0)}$ is equal to the product of the Radon-Nikodym derivative $\rng M_{\frk f}^s$ from Lemma~\ref{lem-rn-sup}, evaluated at $h$ instead of at $\rng h$, with the Radon-Nikodym derivative of the law of $\rng h|_{B_s(0)}$ w.r.t.\ the law of $h|_{B_s(0)}$. 
Combining this with Lemma~\ref{lem-rn-sup} concludes the proof.
\end{proof}

\subsection{Proof of Proposition~\ref{prop-E-prob}}
\label{sec-E-prob}

In this subsection we conclude the proof of our one-point estimate for the events $E_{z,j}$ of Section~\ref{sec-short-def}. 
Note that it suffices to prove Proposition~\ref{prop-E-prob} with $K=A$, since increasing $K$ only increases the probability of $E_{z,1}$ (and does not affect the definition of $E_{z,j}$ for $j\not=1$). 
We will use the following lemma to decompose the event $E_{z,j}$ into the intersection of two independent events such that one has probability of order $(\ep_j/\ep_{j-1})^{\alpha^2/2  + o_j(1)}$ and the other has constant-order probability. 

\begin{lem} \label{lem-gff-split}
Let $h$ be a whole-plane GFF normalized so that $h_1(0) =0$ and let $0 < a < b < 1$. 
Then circle average process $\{h_r(0) - h_b(0)\}_{r\in [a,b]}$ is independent from the triple $\left(  (h-h_a(0))|_{B_a(0)} , h|_{\BB C\setminus B_b(0)},  h - h_{|\cdot|}(0) \right)$. 
\end{lem}
\begin{proof}
The random distributions $h_{|\cdot|}(0)$ and $h-h_{|\cdot|}(0)$ are the projections of $h$ onto the space of radially symmetric functions on $\BB C$ and the space of functions which have mean zero on every circle in $\BB C$, respectively.
By~\cite[Lemma 4.9]{wedges}, these spaces are orthogonal w.r.t.\ the Dirichlet inner product, so $h_{|\cdot|}(0)$ and $h-h_{|\cdot|}(0)$ are independent.
The process $t\mapsto h_{e^{-t}}(0)$ is a standard linear Brownian motion. By the independent increments property of Brownian motion, $\{h_r(0) - h_b(0)\}_{r\in [a,b]}$, $\{h_r(0)\}_{r\geq b}$, and $\{(h_r(0) - h_a(0)\}_{r\leq a}$ are independent. 
The distribution $ (h-h_a(0))|_{B_a(0)}$ (resp.\ $h|_{\BB C\setminus B_b(0)}$) is determined by $h-h_{|\cdot|}(0)$ and $\{(h_r(0) - h_a(0)\}_{r\leq a}$ (resp.\ $\{h_r(0)\}_{r\geq b}$). The lemma statement thus follows.
\end{proof}

We will apply Lemma~\ref{lem-gff-split} to the field $h(\ep_{j-1}\cdot + z) -h_{\ep_{j-1}}(z) \eqD h$ with $a = 6 \ep_j /\ep_{j-1} $ and $b = 1/3$. 
We now describe a decomposition of $E_{z,j} \cap H_{z,j-1}^{\op{out}} \cap H_{z,j-1}^{\op{in}}$ as the intersection of events depending on the four parts of the field involved in in Lemma~\ref{lem-gff-split}. That is, we will define events depending on $(h-h_{\ep_{j-1}}(z))|_{B_{6\ep_j}(z)}$, $(h-h_{\ep_{j-1}}(z))|_{\BB C\setminus B_{\ep_{j-1}/3}(z)}$, $h - h_{|\cdot-z|}(z)$, and $\{h_r(z) - h_{\ep_{j-1}}(z) \}_{r\in [6\ep_j , \ep_{j-1}/3]}$ whose intersection is contained in $h$. 
We will also lower-bound the probabilities of these events. The probability of the intersection of the events depending on  $(h-h_{\ep_{j-1}}(z))|_{B_{6\ep_j}(z)}$, $(h-h_{\ep_{j-1}}(z))|_{\BB C\setminus B_{\ep_{j-1}/3}(z)}$, and $h - h_{|\cdot-z|}(z)$ will be of constant order (Lemmas~\ref{lem-Earound-prob}, \ref{lem-Ein-prob}, \ref{lem-Eout-prob}, and~\ref{lem-Edagger-prob}). The probability of the event depending on $\{h_r(z) - h_{\ep_{j-1}}(z) \}_{r\in [6\ep_j , \ep_{j-1}/3]}$ will be $(\ep_j/\ep_{j-1})^{\alpha^2/2 + o_j(1)}$ (Lemma~\ref{lem-Ecirc-prob}). Using the independence statement of Lemma~\ref{lem-gff-split}, we can multiply the probabilities of the events to get $\BB P[E_{z,j} ] \geq (\ep_j/\ep_{j-1})^{\alpha^2/2 + o_j(1)}$.

\subsubsection{Condition~\ref{item-onescale-around}}

Let $E_{z,j}^{\ref{item-onescale-around}}$ be the event that condition~\ref{item-onescale-around} in the definition of $E_{z,j}$ occurs.
By the locality of the metric (Axiom~\ref{item-metric-f}),
\eqb \label{eqn-Earound-msrble}
E_{z,j}^{\ref{item-onescale-around}} \in \sigma\left( (h-h_{\ep_{j-1}}(z))|_{B_{6\ep_j}(z)} \right) . 
\eqe
We treat $E_{z,j}^{\ref{item-onescale-around}}$ separately since condition~\ref{item-onescale-around} is the only condition in the definition of $E_{z,j}$ which occurs with uniformly positive probability, but not probability close to 1.

\begin{lem} \label{lem-Earound-prob}
There is a constant $p_0 = p_0(\gamma) > 0$ such that for each $z\in\BB C$ and $j\in\BB N$, we have $\BB P[E_{z,j}^{\ref{item-onescale-around}}] \geq p_0$. 
\end{lem}
\begin{proof}
By the translation and scale invariance of the law of $h$ and the Weyl scaling and coordinate change properties of the metric (Axioms~\ref{item-metric-f} and~\ref{item-metric-coord}), the probability of condition~\ref{item-onescale-around} does not depend on $z$ or $j$. 
Hence the lemma statement follows from Lemma~\ref{lem-around-pos}. 
\end{proof}

\subsubsection{Event depending on the inner part of the field}

We now define an event depending on $(h-h_{\ep_{j-1}}(z))|_{B_{6\ep_j}(z)}$ and show that it occurs with high probability. 
Let $E_{z,j}^{\op{in}}$ be the event that the following is true. 
\begin{enumerate}[a.]
\item Conditions \ref{item-onescale-annulus}, \ref{item-onescale-dirichlet}, and~\ref{item-onescale-harmonic} in the definition of $E_{z,j} $ occur, with $K=A$. \label{item-in-event-E}
\item $\sup_{r \in [\ep_j , 6\ep_j]} |h_r(z) - h_{\ep_j}(z)| \leq A$ (which is related to condition~\ref{item-onescale-thick}).\label{item-in-event-thick}
\item $\sup_{u\in \BB A_{\ep_{j }/3,\ep_{j }/2}(z)} |\frk h_{z,j }^{\op{out}}(z)| \leq A$ (which is part of condition~\ref{item-onescale-out}). \label{item-in-event-out}
\item $\BB P\left[H_{z,j}^{\op{in}}   \,\big|\, (h-h_{\ep_j}(z))|_{\BB A_{\ep_j,6\ep_j}(z)} \right] \geq \frac78 $ (which is related to condition~\ref{item-onescale-cond}). \label{item-in-event-cond}
\end{enumerate}
Due to the Weyl scaling and locality properties of the metric (Axioms~\ref{item-metric-local} and~\ref{item-metric-f}) along with Lemma~\ref{lem-H-msrble}, the event $E_{z,j}^{\op{in}}$ is determined by $(h-h_{\ep_j}(z))|_{\BB A_{\ep_j , 6\ep_j}(z)}$. 
Since $h_{\ep_j}(z) - h_{ \ep_{j-1} }(z)$ is determined by $(h-h_{ \ep_{j-1} }(z))|_{\BB A_{\ep_j , 6\ep_j}(z)}$, it follows that
\eqb \label{eqn-Ein-msrble}
E_{z,j}^{\op{in}} \in \sigma\left( (h-h_{\ep_{j-1}}(z))|_{\BB A_{\ep_j , 6\ep_j}(z)} \right) \subset \sigma\left( (h - h_{\ep_{j-1}}(z) ) |_{B_{6\ep_j}(z)} \right) . 
\eqe

To lower-bound the probability of $E_{z,j}^{\op{in}}$, we need the following lemma.

\begin{lem} \label{lem-H-prob}
For each $p\in (0,1)$, there exists $A = A(p,\gamma) > 0$ such that for each $z \in \BB C$ and $j\in\BB N$, 
\allb
\BB P\left[ H_{z,j}^{\op{out}} \right] \geq p  ,\quad 
\BB P\left[ H_{z,j}^{\op{in}} \right] \geq p , \quad
\BB P\left[ \BB P\left[H_{z,j-1}^{\op{out}}    \,\big|\, (h-h_{\ep_j})|_{\BB A_{\ep_j,\ep_{j-1}/3}(z)} \right]  \geq p \right] \geq  p  \notag \\
\qquad \text{and} \qquad \BB P\left[ \BB P\left[H_{z,j}^{\op{in}}    \,\big|\, (h-h_{\ep_j})|_{\BB A_{\ep_j,\ep_{j-1}/3}(z)} \right]  \geq p \right] \geq  p .
\alle
\end{lem}
\begin{proof}
We prove the statement of the lemma for $H_{z,j}^{\op{out}}$; the statement for $H_{z,j}^{\op{in}}$ is proven similarly. 
By the definition~\eqref{eqn-H-def} of $H_{z,j}^{\op{out}}$, the Weyl scaling and conformal covariance properties of $D_h$ (Axioms~\ref{item-metric-f} and~\ref{item-metric-coord}), and the translation and scale invariance of the law of $h$ modulo additive constant, $\BB P[H_{z,j}^{\op{out}}]$ does not depend on $z$ or $j$. 
By Lemma~\ref{lem-internal-diam} and Markov's inequality, together with the continuity of the circle average process, we can find some $A  = A(p,\gamma) > 0$ such that $\BB P[H_{0,1}^{\op{out}}] \geq 1 -(1-p)^2$, and hence $\BB P[H_{z,j}^{\op{out}}] \geq 1-(1-p)^2$ for every $z \in \BB C$ and $j\in\BB N$. 
This estimate with $j-1$ in place of $j$ shows that
\eqbn
\BB E\left[ \BB P\left[  (H_{z,j-1}^{\op{out}})^c \,\big| \,  (h-h_{\ep_j})|_{\BB A_{\ep_j,\ep_{j-1}/3}(z)} \right]   \right]  \leq (1-p)^2 .
\eqen
By Markov's inequality, this implies that
\eqbn
\BB P\left[ \BB P\left[  (H_{z,j-1}^{\op{out}})^c \,\big| \,  (h-h_{\ep_j})|_{\BB A_{\ep_j,\ep_{j-1}/3}(z)} \right]  \geq 1-p  \right]  \leq 1-p ,
\eqen
equivalently, $\BB P\left[ \BB P\left[H_{z,j-1}^{\op{out}}    \,\big|\, (h-h_{\ep_j})|_{\BB A_{\ep_j,\ep_{j-1}/3}(z)} \right]  \geq p \right] \geq  p$. 
\end{proof}

\begin{lem} \label{lem-Ein-prob}
If $A $ is chosen to be sufficiently large, in a manner depending only on $\gamma$, then $\BB P[E_{z,j}^{\op{in}}] \geq 1 - p_0/100$ for each $z\in\BB C$ and $j\in\BB N$. 
\end{lem}
\begin{proof}
By Axioms~\ref{item-metric-f} and~\ref{item-metric-coord} (Weyl scaling and coordinate change) and the scale and translation invariance of the law of $h$, modulo additive constant, the probabilities of each of the conditions in the definition of $ E_{z,j}^{\op{in}} $ does not depend on $z$ or $j$.   
Each part of conditions~\ref{item-in-event-E}, \ref{item-in-event-thick}, and~\ref{item-in-event-out} in the definition of $E_{z,j}^{\op{in}}$ simply requires that some a.s.\ finite (resp.\ positive) random variable is bounded above by $A$ (resp.\ below by $A^{-1}$). 
Consequently, if $A$ is chosen sufficiently large, in a manner depending only on $  \gamma$, then the probability of the intersection of these three conditions is at least $1-p_0/200$. 
Since $H_{z,j}^{\op{in}}$ is determined by $(h-h_{\ep_j}(z)) |_{\BB A_{\ep_j/3,\ep_j}(z)}$ (by~Lemma~\ref{lem-H-msrble}), the Markov property of the field implies that
\eqbn
\BB P\left[H_{z,j}^{\op{in}}   \,\big|\, (h-h_{\ep_j}(z))|_{\BB A_{\ep_j,6\ep_j}(z)} \right]
= \BB P\left[H_{z,j}^{\op{in}}   \,\big|\, (h-h_{\ep_j}(z))|_{\BB A_{\ep_j, \ep_{j-1}/3}(z)} \right] .
\eqen
By Lemma~\ref{lem-H-prob}, after possibly increasing $A$, we can arrange that the probability of condition~\ref{item-in-event-cond} in the definition of $E_{z,j}^{\op{in}}$ is also at least $1-p_0/200$.
\end{proof}

\subsubsection{Event depending on the outer part of the field}

Let $E_{z,j}^{\op{out}}$ be the event that the following is true.
\begin{enumerate}[a.]
\item Condition~\ref{item-onescale-rn} in the definition of $E_{z,j}$ occurs.\label{item-out-event-E}
\item $|h_{\ep_{j-1}/3}(z) - h_{\ep_{j-1}}(z)| \leq A$ (which is relevant to condition~\ref{item-onescale-thick}). \label{item-out-event-thick}
\item $\sup_{u\in \BB A_{\ep_{j-1}/2,\ep_{j-1} }(z)} |\frk h_{z,j-1}^{\op{in}}(z)| \leq A$ (which is part of condition~\ref{item-onescale-out}). \label{item-out-event-out}
\item $\BB P\left[H_{z,j-1}^{\op{out}}   \,\big|\, (h-h_{\ep_{j-1}})|_{\bdy B_{ \ep_{j-1}/3}(z)} \right] \geq \frac78 $ (which is part of condition~\ref{item-onescale-cond}). \label{item-out-event-cond}
\item The event $H_{z,j-1}^{\op{out}} \cap H_{z,j-1}^{\op{in}}$ occurs. \label{item-out-event-H}
\end{enumerate}
By Weyl scaling and locality (Axioms~\ref{item-metric-local} and~\ref{item-metric-f}) together with Lemma~\ref{lem-rn-msrble} (for condition~\ref{item-out-event-E}) and Lemma~\ref{lem-H-msrble} (for condition~\ref{item-out-event-H}),
\eqb \label{eqn-Eout-msrble}
E_{z,j}^{\op{out}} \in \sigma\left( (h-h_{\ep_{j-1}}(z))|_{\BB A_{\ep_{j-1}/3 , \ep_{j-1}}(z)} \right) .
\eqe

\begin{lem} \label{lem-Eout-prob}
If $A $ is chosen to be sufficiently large, in a manner depending only on $\gamma$, and then $L$ is chosen sufficiently large, in a manner depending only on $A$, then $\BB P[E_{z,j}^{\op{out}}] \geq 1 - p_0/100$ for each $z\in\BB C$ and $j\in\BB N$. 
\end{lem}
\begin{proof}
As in the proof of Lemma~\ref{lem-Ein-prob}, $\BB P[E_{z,j}^{\op{out}}]$ does not depend on $z$ or $j$. 
By Lemma~\ref{lem-rn-sup-plane} (applied to the field $h(\ep_{j-1} \cdot  +z) - h_{\ep_{j-1}}(z) \eqD h$ and with $r = 1/2$ and $s =1/3$), if we are given $A$ and we choose $L$ to be sufficiently large, in a manner depending only on $A$, then the probability of condition~\ref{item-out-event-E} in the definition of $E_{z,j}^{\op{out}}$ (i.e., condition~\ref{item-onescale-rn} in the definition of $E_{z,j}$) is at least $1-p_0/500$. 
Trivially, if $A$ is at least some constant depending only $p_0$ (and hence only on $\gamma$) then the probabilities of each of conditions~\ref{item-out-event-thick} and~\ref{item-out-event-out} in the definition of $E_{z,j}^{\op{out}}$ is at least $1-p_0/500$. 
By the Markov property of the GFF (applied as in the proof of Lemma~\ref{lem-Ein-prob}) and Lemma~\ref{lem-H-prob}, after possibly increasing $A$ we can arrange that the probabilities of conditions~\ref{item-out-event-cond} and~\ref{item-out-event-H} are also each at least $1-p_0/500$. 
\end{proof}

\subsubsection{Event depending on the mean-zero part of the field}

We now define an event depending on the mean-zero part $h-h_{|\cdot-z|}(z)$ of the field.
Let $M_j := \lfloor \log(\ep_{j-1}/(3\ep_j)) \rfloor + 1$. 
We can find a deterministic collection of $M_j$ annuli $\{\mcl A_{z,j}^m\}_{m\in [1,M_j]_{\BB Z}}$, each of which is centered at $z$ and has aspect ratio $e$, such that  
\eqb \label{eqn-Edagger-annulus}
 \BB A_{\ep_j , \ep_{j-1}/3}(z) = \bigcup_{m=1}^{M_j}  \mcl A_{z,j}^m  .
\eqe 
Note that the annuli $\mcl A_{z,j}^m$ have a small amount of overlap. 
We enumerate the annuli $\mcl A_{z,j}^m$ from outside to inside, so that 
\eqb \label{eqn-Edagger-radius}
\left(\text{radius of outer boundary of $\mcl A_{z,j}^m$} \right) \asymp e^{-m} \ep_{j-1} ,
\eqe
with universal implicit constants.

Fix $\zeta \in (0,\xi(Q-\alpha))$ and define the event 
\eqb \label{eqn-Edagger-def}
E_{z,j}^\dagger := \left\{ \sup_{u,v\in \mcl A_{z,j}^m } D_{h-h_{|\cdot-z|}(z)}\left( u ,v ; \mcl A_{z,j}^m \right) \leq A   e^{-(\xi Q - \zeta) m}  \ep_{j-1}^{\xi Q} ,\: \forall m \in [1,M_j]_{\BB Z} \right\} .
\eqe
Then
\eqb \label{eqn-Edagger-msrble}
E_{z,j}^\dagger \in \sigma\left( h-h_{|\cdot-z|}(z) \right) .
\eqe

\begin{lem} \label{lem-Edagger-prob}
If $A$ is chosen to be sufficiently large, in a manner depending only on $\gamma$ and $\zeta$, then $\BB P\left[ E_{z,j}^\dagger \right] \geq 1 - p_0/100$ for each $z\in\BB C$ and $j\in \BB N$.  
\end{lem}
\begin{proof}
By Lemma~\ref{lem-meanzero-diam}, the Chebyshev inequality, and~\eqref{eqn-Edagger-radius}, there are constants $c_0,c_1 > 0$ depending only on $\gamma$ such that for each $z \in \BB C$ and $j,m\in\BB N$,
\eqb
\BB P\left[ \sup_{u,v\in \mcl A_{z,j}^m } D_{h-h_{|\cdot-z|}(z)}\left( u ,v ; \mcl A_{z,j}^m \right)  > A  e^{-(\xi Q  - \zeta) m}  \ep_{j-1}^{\xi Q}  \right]  \leq c_0 A^{-c_1} e^{-c_1 \zeta m} .
\eqe 
Summing this estimate over all $m\in [1,M_j]_{\BB Z}$ gives
\allb
\BB P\left[ (E_{z,j}^\dagger)^c \right] 
\leq c_0 A^{-c_1}   \sum_{m=1}^{M_j} e^{-c_1\zeta m} ,
\alle
which is smaller than $p_0/100$ for a large enough choice of $A$, depending only on $\gamma$ and $\zeta$. 
\end{proof}

\subsubsection{Event depending on the circle average process}

Finally, let
\eqb \label{eqn-Ecirc-def}
E_{z,j}^\circ := \left\{ | h_r(z) - h_{\ep_{j-1}/3}(z)  - \alpha  \log( \ep_{j-1} / r)| \leq A \log(\ep_{j-1}/\ep_j)^{3/4}  ,\: \forall r \in [ 6\ep_j , \ep_{j-1}/3 ]   \right\} .
\eqe
Then
\eqb \label{eqn-Ecirc-msrble}
E_{z,j}^\circ \in \sigma\left( \left\{  h_r(z) - h_{\ep_{j-1}/3}(z) \right\}_{r\in [ 6\ep_j , \ep_{j-1}/3 ] }   \right) .
\eqe

\begin{lem} \label{lem-Ecirc-prob}
If $A$ is chosen to be sufficiently large, in a manner depending only on $\alpha,\gamma$, then $\BB P[E_{z,j}^\circ] \geq (\ep_j/\ep_{j-1})^{\alpha^2/2 + o_j(1)}$, with the rate of the $o_j(1)$ uniform over all $z\in\BB C$. 
\end{lem}
\begin{proof}
The process $t\mapsto h_{(\ep_{j-1}/3) e^{-t}}(z) - h_{\ep_{j-1}/3}(z)$ is a standard linear Brownian motion~\cite[Section 3.1]{shef-kpz}, so this follows from a straightforward Brownian motion estimate. More precisely, the explicit formula for the Gaussian density shows that with probability at least $(\ep_j / \ep_{j-1})^{\alpha^2/2 + o_j(1)}$, we have $|h_{6\ep_j}(z) - h_{\ep_{j-1}/3}(z)  - \alpha  \log(\ep_j /\ep_{j-1}) | \leq 1$. If we condition on a particular realization of $h_{6\ep_j}(z) - h_{\ep_{j-1}/3}(z)$ for which this is the case, then $t\mapsto h_{(\ep_{j-1}/3) e^{-t}}(z) - h_{\ep_{j-1}/3}(z)$ evolves as a Brownian bridge and it is easily seen that the conditional probability of $E_{z,j}^\circ$ is bounded below by a deterministic $\alpha,\gamma$-dependent constant for a large enough choice of $A$. 
\end{proof}

\subsubsection{Conclusion of the proof of Proposition~\ref{prop-E-prob}}

Obviously, $\BB P[E_{z,1}]$ is increasing in $K$, and $E_{z,j}$ for $j\geq 2$ does not depend on $K$. 
Hence we only need to prove the proposition in the case when $K=A$.

By combining Lemmas~\ref{lem-Earound-prob}, \ref{lem-Ein-prob}, \ref{lem-Eout-prob}, and \ref{lem-Edagger-prob}, we find that for every $z \in \BB C$ and $j\in\BB N$, 
\eqb \label{eqn-E-pos}
\BB P\left[ E_{z,j}^{\ref{item-onescale-around}} \cap E_{z,j}^{\op{in}} \cap E_{z,j}^{\op{out}} \cap E_{z,j}^\dagger \right] 
\geq p_0/2 .
\eqe
By the measurability conditions~\eqref{eqn-Ein-msrble}, \eqref{eqn-Eout-msrble}, and \eqref{eqn-Edagger-msrble}, and~\eqref{eqn-Ecirc-msrble} together with Lemma~\ref{lem-gff-split} (applied to the field $h(\ep_{j-1}\cdot + z) - h_{\ep_{j-1}}(z)\eqD h$) we find that the event $ E_{z,j}^{\ref{item-onescale-around}} \cap E_{z,j}^{\op{in}} \cap E_{z,j}^{\op{out}} \cap E_{z,j}^\dagger$ is independent of the event $E_{z,j}^{\circ}$. 
By multiplying the estimates of~\eqref{eqn-E-pos} and Lemma~\ref{lem-Ecirc-prob}, we therefore get that for an appropriate choice of $A$ and $L$ (depending only on $\alpha,\gamma$),
\eqb
\BB P\left[ \wt E_{z,j} \right] \geq (\ep_j/\ep_{j-1})^{\alpha^2/2 + o_j(1)} ,\quad \text{where} \quad 
\wt E_{z,j} = \wt E_{z,j}(A,L) := E_{z,j}^{\ref{item-onescale-around}} \cap E_{z,j}^{\op{in}} \cap E_{z,j}^{\op{out}} \cap E_{z,j}^\dagger \cap E_{z,j}^\circ .
\eqe

We will now conclude the proof by showing that $\wt E_{z,j}(A,L)  \subset E_{z,j}(A',L) \cap H_{z,j-1}^{\op{out}}(A) \cap H_{z,j }^{\op{in}}(A)$ for a constant $A' = A'(A,\gamma) \geq A$. 
Henceforth assume that $\wt E_{z,j}(A,L)$ occurs.
By the definition~\eqref{eqn-Ecirc-def} of $E_{z,j}^\circ$ combined with condition~\ref{item-in-event-thick} in the definition of $E_{z,j}^{\op{in}}$ and condition~\ref{item-out-event-thick} in the definition of $E_{z,j}^{\op{out}}$, we find that condition~\ref{item-onescale-thick} in the definition of $E_{z,j}$ occurs with $3A$ in place of $A$.
By the definition of $E_{z,j}^{\ref{item-onescale-around}}$, condition~\ref{item-onescale-around} in the definition of $E_{z,j}(A,L)$ occurs.
By condition~\ref{item-in-event-E} in the definition of $E_{z,j}^{\op{in}}$, conditions \ref{item-onescale-annulus}, \ref{item-onescale-dirichlet}, and~\ref{item-onescale-harmonic} in the definition of $E_{z,j} $ occur.
By condition~\ref{item-out-event-E} in the definition of $E_{z,j}^{\op{out}}$, condition~\ref{item-onescale-rn} in the definition of $E_{z,j}$ occurs. 
By conditions~\ref{item-in-event-out} and~\ref{item-in-event-cond} in the definition of $E_{z,j}^{\op{in}}$ and conditions~\ref{item-out-event-out} and~\ref{item-out-event-cond} in the definition of $E_{z,j}^{\op{out}}$, conditions~\ref{item-onescale-out} and~\ref{item-onescale-cond} in the definition of $E_{z,j}$ occur.
By condition~\ref{item-out-event-H} in the definition of $E_{z,j}^{\op{out}}$, also $H_{z,j-1}^{\op{out}}  \cap H_{z,j-1 }^{\op{in}}$ occurs.

It remains to deal with condition~\ref{item-onescale-sup} (behavior of $h_r(z)$ for $z\in [\ep_j,\ep_{j-1}/3]$) and~\ref{item-onescale-diam} ($D_h$-diameter of $\BB A_{\ep_j,\ep_{j-1}/3}(z)$). This will require us to further increase $A$. 
The definition~\eqref{eqn-Ecirc-def} of $E_{z,j}^\circ$ bounds $h_r(z) - h_{\ep_{j-1}/3}(z)$ for $r\in [6\ep_j,\ep_{j-1}/3]$ whereas condition~\ref{item-in-event-thick} in the definition of $E_{z,j}^{\op{in}}$ and condition~\ref{item-out-event-thick} in the definition of $E_{z,j}^{\op{out}}$ bound $h_r(z) - h_{6\ep_j}(z)$ for $r\in [\ep_j,6\ep_j]$ and $h_{\ep_{j-1}/3}(z) - h_{\ep_{j-1}}(z)$, respectively. Combining these conditions shows that condition~\ref{item-onescale-sup} in the definition of $E_{z,j}$ occurs for an appropriate choice of $\wt A > A$. 

To deal with condition~\ref{item-onescale-diam}, we recall the annuli $\mcl A_{z,j}^m \subset \BB A_{\ep_j,\ep_{j-1}/3}(z)$ as in~\eqref{eqn-Edagger-annulus}. 
Let $a_{z,j}^m$ and $b_{z,j}^m$ be the radii of the inner and outer boundaries of $\mcl A_{z,j}^m$, so that $\mcl A_{z,j}^m = \BB A_{a_{z,j}^m , b_{z,j}^m}(z)$. By~\eqref{eqn-Edagger-radius}, $a_{z,j}^m = b_{z,j}^m / e \asymp e^{-m} \ep_{j-1}$. 
For each $m\in [1,M_j]_{\BB Z}$, 
\allb \label{eqn-annulus-diam-split}
 \sup_{u,v\in \mcl A_{z,j}^m } D_h\left( u ,v ; \mcl A_{z,j}^m \right) 
 \leq \left( \sup_{r \in [a_{z,j}^m , b_{z,j}^m] } e^{\xi h_r(z)}   \right) \left( \sup_{u,v\in \mcl A_{z,j}^m } D_{h-h_{|\cdot-z|}(z)}\left( u ,v ; \mcl A_{z,j}^m \right) \right)  .
\alle
By condition~\ref{item-onescale-sup} in the definition of $E_{z,j}$ (with $\wt A$ in place of $A$), the first factor on the right side of~\eqref{eqn-annulus-diam-split} is bounded above by $e^{\xi h_{\ep_{j-1}}(z)} e^{(\xi \alpha + o_j(1)) m}$, where the rate of convergence of the $o_j(1)$ depends only on $\wt A$, $\alpha$, and $\gamma$.  
By the definition~\eqref{eqn-Edagger-def} of $E_{z,j}^\dagger$, the second factor on the right side of~\eqref{eqn-annulus-diam-split} is bounded above by $A e^{-(\xi Q -\zeta) m} \ep_{j-1}^{\xi Q}$. 
Plugging these bounds into~\eqref{eqn-annulus-diam-split} then summing over all $m\in [1,M_j]_{\BB Z}$ gives
\allb
\sup_{u,v\in \BB A_{ \ep_j,\ep_{j-1}/3}(z)} D_h\left(u,v ; \BB A_{ \ep_j , \ep_{j-1}/3}(z)\right) 
&\leq \sum_{m=1}^{M_j} \sup_{u,v\in \mcl A_{z,j}^m } D_h\left( u ,v ; \mcl A_{z,j}^m \right)  \notag \\
&\leq A \ep_{j-1}^{\xi Q} e^{\xi h_{\ep_{j-1}}(z)} \sum_{m=1}^{M_j} e^{-( \xi Q - \xi \alpha  -\zeta - o_j(1) ) m} . 
\alle
Since $\alpha \leq 2 <  Q$, we can choose $\zeta < \xi(Q-\alpha)$, so this last quantity is bounded above by a constant $A' = A'(\wt A  , \alpha ,\gamma) > \wt A $ times $\ep_{j-1}^{\xi Q} e^{\xi h_{\ep_{j-1}}(z)} $. 
Therefore, condition~\ref{item-onescale-diam} in the definition of $E_{z,j}$ holds with $A'$ in place of $A$.

Thus, we have established~\eqref{eqn-E-prob}. Since $h_{\ep_0}(z)$ for $z\in\BB D$ is centered Gaussian with variance bounded above by a universal constant and by Lemma~\ref{lem-H-prob}, we trivially obtain $\BB P[E_{z,0}]\succeq 1$, which concludes the proof. \qed

\appendix

\section{Gaussian free field review}
\label{sec-gff-prelim}

Here we briefly review the definition and basic properties of the zero-boundary and whole-plane Gaussian free field (GFF). More detail can be found in~\cite{shef-gff}, the introductory sections of~\cite{ss-contour,shef-kpz,ig1,ig4,dubedat-coupling}, and/or the notes~\cite{berestycki-lqg-notes}.   

\subsubsection*{GFF definition}

For an open domain $U\subset \BB C$ with harmonically non-trivial boundary (i.e., Brownian motion started from a point in $U$ a.s.\ hits $\bdy U$), we define $\mcl H(U)$ be the Hilbert space completion of the set of smooth, compactly supported functions on $U$ with respect to the \emph{Dirichlet inner product},
\eqb \label{eqn-dirichlet}
(\phi,\psi)_\nabla = \frac{1}{2\pi} \int_U \nabla \phi(z) \cdot \nabla \psi(z) \,dz .
\eqe
In the case when $U= \BB C$, constant functions $c$ satisfy $(c,c)_\nabla = 0$, so to get a positive definite norm in this case we instead take $\mcl H(\BB C)$ to be the Hilbert space completion of the set of smooth, compactly supported functions $\phi$ on $\BB C$ with $\int_{\BB C} \phi(z) \,dz = 0$, with respect to the same inner product~\eqref{eqn-dirichlet}.  
 
The \emph{(zero-boundary) Gaussian free field} on $U$ is defined by the formal sum
\eqb \label{eqn-gff-sum}
h = \sum_{j=1}^\infty X_j \phi_j 
\eqe
where the $X_j$'s are i.i.d.\ standard Gaussian random variables and the $\phi_j$'s are an orthonormal basis for $\mcl H(U)$. The sum~\eqref{eqn-gff-sum} does not converge pointwise, but it is easy to see that for each fixed $\phi \in \mcl H(U)$, the formal inner product $(h ,\phi)_\nabla$ is a mean-zero Gaussian random variable and these random variables have covariances $\BB E [(h,\phi)_\nabla (h,\psi)_\nabla] = (\phi,\psi)_\nabla$. In the case when $U \not=\BB C$ and $U$ has harmonically non-trivial boundary, one can use integration by parts to define the ordinary $L^2$ inner products $(h,\phi) := -2\pi (h,\Delta^{-1}\phi)_\nabla$, where $\Delta^{-1}$ is the inverse Laplacian with zero boundary conditions, whenever $\Delta^{-1} \phi \in \mcl H(U)$. 

The case when $U=\BB C$ is our primary interest in this paper. In this case, one can similarly define $(h ,\phi) := -2\pi (h ,\Delta^{-1}\phi)_\nabla$ where $\Delta^{-1}$ is the inverse Laplacian normalized so that $\int_{\BB C} \Delta^{-1} \phi(z) \, dz = 0$. With this definition, one has $(h+c , \phi) = (h ,\phi) + (c,\phi) = (h,\phi)$ for each $\phi \in \mcl H(\BB C)$, so the whole-plane GFF is only defined modulo a global additive constant. 
We almost always fix the additive constant by requiring that the circle average $h_1(0)$ of $h$ over the unit circle $\bdy\BB D$ (whose definition we recall just below) is zero, i.e., we consider the field $h - h_1(0)$ which is well-defined not just modulo additive constant. 
See~\cite[Section 2.2]{ig4} for more on the whole-plane GFF.

\subsubsection*{Circle averages}

If $U$ and $h$ are as above, then for $z\in U$ and $r > 0$ such that $\bdy B_r(z) \subset U$ we can define the circle average $h_r(z)$ over $\bdy B_r(z)$, following~\cite[Section 3.1]{shef-kpz}. One way to define $h_r(z)$ is as follows. Let $f_{z,r}$ be the function on $U$ such that $- \Delta f_{z,r}$ (with the Laplacian defined in the distributional sense) is $2\pi$ times the uniform measure on $\bdy B_r(z)$ and $f_{z,r}|_{\bdy U} = 0$ (if $U\not=\BB C$) or $\int_{\BB C} f_{z,r}(u)\,du = 0$ (if $U=\BB C$). 
We then set $h_r(z) := (h,f_{z,r})_\nabla$. 

It is shown in~\cite[Section 3.1]{shef-kpz} that the circle average process a.s.\ admits a modification which is continuous in $z$ and $r$. We always assume that $h_r(z)$ has been replaced by such a modification. Furthermore,~\cite[Section 3.1]{shef-kpz} provides an explicit description of the law of the circle average process (it is a centered Gaussian process with an explicit covariance structure). For our purposes, we only need the following basic facts about this process. Let $h$ be a whole-plane GFF normalized so that $h_1(0) = 0$. 
\begin{itemize}
\item The law of $h$ is scale and translation invariant modulo additive constant, which means that for $z\in\BB C$ and $r>0$ one has $h(r\cdot +z) - h_r(z) \eqD h  $. 
\item For each fixed $z\in\BB C$, the process $\BB R\ni t \mapsto h_{e^{-t}}(z) - h_1(z)$ is a standard two-sided linear Brownian motion. 
\item If $U\subset \BB C$ is bounded, then for $z\in U$ and $r\in (0,1]$, $h_r(z)$ is centered Gaussian with variance $\log r^{-1} + O (1)$, with the $O (1)$ only depending on $U$. 
\end{itemize}

\subsubsection*{Markov property and restrictions}

If $h$ is a (whole-plane or zero-boundary, as appropriate) GFF on $U$ and $V\subset U$ is open, then we can make sense of $h|_V$ as a random distribution on $V$: it is simply the restriction of the mapping $f\mapsto (h,f)$ to functions $f$ which are supported on $V$. If $K\subset U$ is closed, we define the $\sigma$-algebra $\sigma(h|_K) := \bigcap_{\ep > 0} \sigma(h|_{B_\ep(K)})$. Hence we can talk about conditioning on $h|_K$ or say that a random variable is determined by $h|_K$. 

The zero-boundary GFF on a domain $U$ with harmonically non-trivial boundary possesses the following Markov property (see, e.g.,~\cite[Section 2.6]{shef-gff}).  Let $V\subset U$ be a sub-domain with harmonically non-trivial boundary.
Then we can write $h = \frk h + \rng h$, where $\frk h$ is a random distribution on $U$ which is harmonic on $V$ and is determined by $h|_{U\setminus V}$; and $\rng h$ is a zero-boundary GFF on $V$ which is independent from $h|_{U\setminus V}$.
We call $\rng h$ and $\frk h|_V$ the \emph{zero-boundary part} and \emph{harmonic part} of $h|_V$, respectively.  

In the whole-plane case, the Markov property is slightly more complicated due to the need to fix the additive constant.  
We use the following version, which is proven in~\cite[Lemma 2.2]{gms-harmonic}.

\begin{lem}[Markov property of the whole-plane GFF] \label{lem-whole-plane-markov}
Let $h$ be a whole-plane GFF with the additive constant chosen so that $h_1(0) = 0$. 
For each open set $V\subset\BB C$ with harmonically non-trivial boundary, we have the decomposition $h = \rng h + \frk h$
where $\frk h$ is a random distribution which is harmonic on $V$ and is determined by $h|_{\BB C\setminus V}$ and $\rng h$ is independent from $\frk h$ and has the law of a zero-boundary GFF on $V$ minus its average over $\bdy \BB D \cap V$. If $V$ is disjoint from $\bdy \BB D$, then $\rng h$ is a zero-boundary GFF and is independent from $h|_{\BB C\setminus V}$. 
\end{lem}

\subsubsection*{Absolute continuity}

Adding a function in $\mcl H(U)$ to the GFF affects its law in an absolutely continuity way, with an explicit Radon-Nikodym derivative.
In this paper we frequently use the whole-plane version of this fact, which follows from~\cite[Proposition 2.9]{ig4} and reads as follows.

\begin{lem} \label{lem-gff-abs-cont}
Let $h$ be a whole-plane GFF normalized so that $h_1(0) = 0$ and let $f : \BB C\rta\BB C$ be a deterministic continuous function such that $(f,f)_\nabla < \infty$ and the average of $f$ over $\bdy\BB D$ is zero. Then the laws of $h$ and $h+f$ are mutually absolutely continuous. The Radon-Nikodym derivative of the law of $h+f$ w.r.t.\ the law of $h$ and the Radon-Nikodym derivative of the law of $h$ w.r.t.\ the law of $h+f$ are given, respectively, by
\eqb
\exp\left( (h,f)_\nabla - \frac12 (f,f)_\nabla \right) \quad \text{and} \quad \exp\left( -(h+f,f)_\nabla + \frac12(f,f)_\nabla \right).
\eqe
\end{lem}

\bibliography{cibib}
\bibliographystyle{hmralphaabbrv}

\end{document}